\RequirePackage{fix-cm}
\documentclass[smallextended]{svjour3}       
\smartqed  
\usepackage{graphicx}
\begin{document}

\title{Finite Element Methods For Wave Propagation With
 Debye Polarization In Nonlinear Dielectric Materials 
}

\titlerunning{FEMs  For Wave Propagation In Nonlinear Dielectric Materials}        

\author{Qiumei Huang       \and
         Shanghui Jia      \and
      Fei Xu \and
        Zhongwen Xu \and
       Changhui Yao
        }

\authorrunning{Q.M. Huang, S.H. Jia,  F. Xu, Z.W. Xu,and C.H. Yao} 

\institute{          Qiumei Huang \at
              Beijing Institute for Scientific and Engineering Computing, Beijing University of technology, Beijing, 100124,P. R. China.
               \email{qmhuang@bjut.edu.cn}  \\
              Shanghui Jia  \at  School of Statistic and Mathematics, Central University of Finance and Economics, Beijing 100081,P. R. China. \email{shjia@lsec.cc.ac.cn}.\\
              Fei Xu \at  Beijing Institute for Scientific and Engineering Computing, Beijing University of technology, Beijing, 100124,P. R. China.\email{xufei@lsec.cc.ac.cn}.\\
               Zhongwen Xu \at  School of Mathematics and Statistics, Zhengzhou University, Zhengzhou,Henan 450001, P.R. China. \email{chyao@zzu.edu.cn}.\\
             Corresponding author:   Changhui Yao \at  School of Mathematics and Statistics, Zhengzhou University, Zhengzhou,Henan 450001, P.R. China. \email{chyao@lsec.cc.ac.cn}  \\
}
\date{Received: date / Accepted: date}

\maketitle

\begin{abstract}
In this paper, we consider the wave propagation with
 Debye polarization in nonlinear dielectric materials. For this model, the Rother's method is employed to derive  the well-posedness of the electric fields and the existence of the polarized fields by monotonicity theorem as well as the boundedness of the two fields  are established. Then, the time errors are derived for the semi-discrete solutions  by the order $O(\Delta t)$.
 Subsequently, decoupled the full-discrete scheme of the Euler  in time and  Raviart-Thomas-N$\acute{e}$d$\acute{e}$lec element $k\geq 2$ in spatial is established.  Based on the truncated error, we present the convergent analysis with the order $O(\Delta t+h^s) $ under the technique of a-prior $L^\infty$ assumption. For the $k=1$, we employ the superconvergence technique to ensure the a-prior $L^\infty$ assumption. In the end, we give some numerical examples to demonstrate our theories.

\keywords{Maxwell's Equations  \and Nonlinear
 \and   Dielectric Materials   \and Finite Element Methods \and Error Estimates}
 \PACS{65N30 \and 65N15  \and 35J25}
\end{abstract}

\section{Introduction}

The wave propagation  can be controlled by the Maxwell's system
\begin{eqnarray*}
&&\nabla\times{\bf E}=-\frac{\partial{\bf B}}{\partial t},\ \
\nabla\times{\bf H}=\frac{\partial{\bf D}}{\partial t}+{\bf J},\\
&&\nabla\cdot{\bf D}=\rho,\hspace{0.85cm}
\nabla\cdot{\bf B}=0,
\end{eqnarray*}
where ${\bf E}$ and ${\bf H}$ denote the strengths of the electric and magnetic fields, respectively. ${\bf D}$ and ${\bf B}$ are the electric and magnetic flux densities, respectively. ${\bf J}$ and $\rho$ represent the current density and the density of free electric charge, respectively. The above equations will be supplemented with the constitutive laws that describe the behavior of the material under the influence of the fields. Those constitutive laws are given by
\begin{eqnarray*}
&&{\bf D}=\epsilon_0{\bf E}+{\bf P},\ \
{\bf B}=\mu_0{\bf H}+\mu_0{\bf M},\ \
{\bf J}={\bf J_c}+{\bf J_s},
\end{eqnarray*}
where ${\bf P}$ and ${\bf M}$ represent the electric and magnetic polarization, respectively. ${\bf J_c}$ and ${\bf J_s}$ stand for the conduction current density and source current density, respectively. $\epsilon_0$ and $\mu_0$ are the electric permittivity of free space and the magnetic permeability, respectively. We assume ${\bf M=0}$ since we can choose to ignore the magnetic effect among the dielectric materials

\par
The investigation of the represented  polarization term is a key issue in Maxwell's equations. In \cite{MR1}, the authors employed the Maxwell's equations with
linear constitutive relationship describing the material
polarization in a convolution representation
\begin{eqnarray}
\label{equation:eq-1}
{\bf P}({\bf x},t)=\int_0^tg(t-s,{\bf x}){\bf E}({\bf x},s)ds.
\end{eqnarray}

Practically,  there were some evidences that this relationship should be better described by a nonlinear law \cite{MR2,MR3,MR4}. By the efforts in \cite{MR5}, the authors derived the theoretical results after a representation of the polarization by a nonlinear convolution
\begin{eqnarray}
\label{equation:eq-2}
{\bf P}({\bf x},t)=\int_0^tg(t-s,{\bf x})({\bf E}({\bf x},s)+f({\bf E}({\bf x},s)))ds.
\end{eqnarray}

The polarization representation in (\ref{equation:eq-1}) originates from the model proposed by Debye \cite{MR6}
\begin{eqnarray}
\label{equation:eq-3}
\tau\dot {\bf P}+{\bf P}=\epsilon_0(\epsilon_s-\epsilon_\infty){\bf E},
\end{eqnarray}
with the kernel $g(t)=\frac{\epsilon_0(\epsilon_s-\epsilon_\infty)}{\tau}e^{-\frac{t}{\tau}}$
, where $ \epsilon_s$, $\epsilon_\infty$ and $\tau $ stand by the static relative permittivity,
the value of permittivity for an extremely
high frequency field and the relaxation time of the dielectric material, respectively. Similar to this representation, the
nonlinearly forced Debye model based on the differential equation can be rewritten by
\begin{eqnarray}
\label{equation:eq-4}
\tau\dot {\bf P}+{\bf P}=\epsilon_0(\epsilon_s-\epsilon_\infty)({\bf E}+f({\bf E})).
\end{eqnarray}
In this paper, we explore a more realistic model that includes a nonlinear function of the polarization given by the nonlinear
Debye equation
\begin{eqnarray}
\label{equation:eq-5}
\tau\dot {\bf P}+f({\bf P})=\epsilon_0(\epsilon_s-\epsilon_\infty){\bf E},
\end{eqnarray}
where the nonlinear function $f$: $R^d\rightarrow R$, $d=1,2,3$  in $C^1$ with $f(0)=0$ and $0<f'(x)<B$ for all $x\in R^d$, $B$ is a fixed positive constant. Moreover, $f$ is supposed to be strongly monotone, Lipschitz continuous and bounded, which concludes as follows
\begin{eqnarray}
\label{equation:eq-006}
&&(f(x)-f(y))\cdot(x-y)\geq \omega_f|x-y|^2,\ \omega_f>0,\ \forall x,y \in R^d,\\
\label{equation:eq-007}
&&|f(x)-f(y)|\leq C_L|x-y|\ ,C_L>0,\ \forall x,y\in R^d,\\
\label{equation:eq-008}
&&|f(x)|\leq M|x|\ ,M>0,\ \forall x\in R^d.
\end{eqnarray}
It's reasonable of such the assumptions on the nonlinear function $f$ since one would choose $f({
\bf P})=\delta_1{\bf P}+\delta_2{\bf P}^3=:{\bf P}(\delta_1+\delta_2|{\bf P}|^2)$ in \cite{MR7}.

Assumed that the domain  is any convex, bounded and simply connected Lipschitz polyhedron $\Omega$.  For a given time $T$, we can derive the coupling model as follows
\begin{eqnarray}
\label{equation:eq-6}
\mu_0\epsilon_0\partial_{tt}{\bf E}&+&\frac{\epsilon_0\mu_0(\epsilon_s-\epsilon_\infty)}{\tau}\partial_t{\bf E}+\nabla\times\nabla\times{\bf E}
-\frac{\epsilon_0\mu_0(\epsilon_s-\epsilon_\infty)}{\tau^2}f'({\bf P}){\bf E}\nonumber\\
\label{equation:eq-7}
&+&\frac{\mu_0}{\tau^2}f'({\bf P})f({\bf P})=-\mu_0\partial_t{\bf J}_s, \ \ \ \ \ \ \ in\ \ \Omega \times (0,T], \\
\tau\partial_t{\bf P}&+&f({\bf P})=\epsilon_0(\epsilon_s-\epsilon_\infty){\bf E},\ \ \ \hspace{0.8cm}in\ \ \Omega \times (0,T],
\end{eqnarray}
with the initial data
\begin{eqnarray}
\label{equation:eq-8}
&&{\bf E}({\bf x},0)={\bf E}_0\ ,\ \partial_t{\bf E}({\bf x},0)={\bf E}_0'\ ,\ {\bf P}({\bf x},0)={\bf P}_0,\ \ \ {\bf x}\in\Omega,
\end{eqnarray}
and the perfectly electric boundary condition
\begin{eqnarray}
\label{equation:eq-9}
{\bf E}\times {\bf n}=0,\ \ \ on  \ \ \partial\Omega\times(0,T].
\end{eqnarray}

\par
For the mathematical model of  nonlinear
 materials, there are some wonderful and interesting results.
 In \cite{MR71}, the authors concerned with a mathematical model in one dimension describing the electromagnetic interrogation dielectric materials and addressed the well-posedness and regularity solution. They also illustrated that the solution had higher regularity in time even though the input source may be a windowed signal of distributional type.
 The high frequency pulse propagation
in nonlinear dielectric materials in one dimension was investigated and a Galerkin method to derive existence, uniqueness
and continuous dependence of the resulting system was employed in \cite{MR7}.
  In \cite{MR8},  the authors considered a setup where the waves propagate toward a preferred direction,
called range, analyzed the solution  with the Markov limit
theorem and obtained a detailed asymptotic characterization of the electromagnetic wave
field in the long range limit.

\par
Numerically,
for nonlinear electromagnetic systems, there are plenty of efforts in increasing efficient finite element  methods \cite{MR0007,MR012,MR016,MR020,MR018,MR003,MR116}. In \cite{MR012}, the authors investigated the nonlinear eddy current model in GO silicon steel laminations and studied an $H-\psi$ formulation  in laminated conductors. They then explored homogenization of quasi-static
Maxwell's equations and they also derived the three-dimensional eddy current problem in laminated structures in [12].
In \cite{MR026}, the authors  started
with derivation of a mathematical model assuming a nonlinear dependency of magnetic
field ${\bf H}$.
A nonlinear degenerate transient eddy current model was studied and the convergence
and the deduced error estimates of the approximation to the weak solution were provided in [18].

\par
 In this paper, we select the variational problem in  $H(curl,\Omega)$ for the dimensional $d=2,3$ space, not in ${\bf H}^1(\Omega)$ like in \cite{MR71,MR7} . Then, we  present the Euler semi-discrete scheme
 in time. This scheme is decoupled between the electric fields and the polarization, which means that the derived electric fields equation is the linear problem and the polarized equation is nonlinear at each time step.
 The existence and uniqueness of the linear problem can be discussed by Lax-Milgram  lemma directly.
 We use Rothe's method  to prove that the polarized solution converges to that of the variational problem, which is based on the bounded,  coercive, strictly monotone and hemi-continuous of the defined nonlinear operator.
 We further discuss the convergence in time for the decoupled systems based on the boundedness in $L^2$-norm. Morover, we analyze the error estimates of the discrete scheme by the  Euler
and N\'{e}d\'{e}lec-Raviart-Thomas element, which has to ensure $k\geq 2$ element since the a-prior $L^\infty$ assumption is employed.
 For the $k=1$ element, we use the global superconvergence analysis to ensure  the a-prior $L^\infty$ assumption.
 At last, we give some numerical examples to demonstrate our theories.

\par
The outline of the paper is as follows: in section 2, we introduce the spaces and the variational problem. The Rothe's
method is pushed and the convergent analysis in time is set up in section 3. In section 4, We derive the solvability of
the full discrete scheme and the error estimates for higher finite element. The superconvergence to ensure the a-prior $L^\infty$
assumption is given in section 5. The theoretical results are illustrated by a broad range of numerical examples (including
the convergent data, figures of the numerical solutions and error values at grids on the mesh) in the final section.

\section{Variational Formulation}

The most frequently used spaces in the subsequent analysis
are the following two Sobolev spaces
\begin{eqnarray*}
&&H(curl,\Omega)=\{{\bf u}\in L^2(\Omega)^d;\nabla\times{\bf u}\in  L^2(\Omega)^d  \},\ d=2,3,
\end{eqnarray*}
and its subspace
\begin{eqnarray*}
&&X:=H_0(curl,\Omega)=\{{\bf u}\in H(curl,\Omega), {\bf u}\times {\bf n}=0, \ on\ \   \partial\Omega \},
\end{eqnarray*}
which is the equipped with the inner product
$$({\bf u},{\bf v})_{H(curl, \Omega)}=({\bf u},{\bf v})+(\nabla\times {\bf u},\nabla\times {\bf v}),$$
and the norm
$$ \|{\bf u}\|_X^2=\|{\bf u}\|_0^2+\|\nabla\times {\bf u}\|_0^2.$$
Throughout this article, we use boldface notation to represent vector-valued quantities, such as $Y:={\bf L}^2(\Omega):=(L^2(\Omega))^d$.
Now we can define the weak formulation.

${\bf Definition\ 2.1}$\ \ The potential fields ${\bf E}$ and ${\bf P}$, satisfying ${\bf E}\in H^2(0,T;X)$ and ${\bf P}\in H^1(0,T;Y)$, are the weak solutions of the Maxwell's equations (\ref{equation:eq-6})-(\ref{equation:eq-9}). That is, for any $\Phi\in X,\Psi\in Y$, with the initial data ${\bf E}({\bf x},0)={\bf E}_0\ ,\ \partial_t{\bf E}({\bf x},0)={\bf E}_0'\ ,\ {\bf P}({\bf x},0)={\bf P}_0,$
 there holds
\begin{eqnarray}
\epsilon_0\mu_0(\partial_{tt}{\bf E},{\bf\Phi})
+\frac{\epsilon_0\mu_0(\epsilon_s-\epsilon_\infty)}{\tau}(\partial_t{\bf E},{\bf \Phi})&+&(\nabla\times{\bf E},\nabla\times{\bf\Phi})\nonumber\\
\label{equation:eq-10}
-
\frac{\epsilon_0\mu_0(\epsilon_s-\epsilon_\infty)}{\tau^2}(f'({\bf P}){\bf E},{\bf \Phi}) +\frac{\mu_0}{\tau^2}(f'({\bf P})f({\bf P}),{\bf\Phi})&=&-\mu_0(\partial_t{\bf J}_s,{\bf\Phi}),\\
\label{equation:eq-11}
(\tau\partial_t{\bf P},{\bf \Psi})+(f({\bf P}),{\bf \Psi})&=&\epsilon_0(\epsilon_s-\epsilon_\infty)({\bf E},{\bf \Psi}).
\end{eqnarray}

\section{Well-posedness of weak formulation}
In this section, we first introduce the semi-discretization in time for (13)-(14). We then discuss the stability of the
solutions of the semi-discretization in time and its existence and uniqueness. The convergence and error analysis in time
are discussed in subsections 3.4 and 3.5, respectively.
\subsection{Semi-discretization in time}
In this subsection, we use Rothe's method to study the  solutions of (\ref{equation:eq-10})-(\ref{equation:eq-11}). Let $n$ be a positive integer and $\{t_i=i\Delta t:i=0,1,\cdots,n\}$ be a equidistant partition of $[0,T]$ with $\Delta t=T/n$. Now set
\begin{eqnarray*}
&&u_i=u(x_i)\ ,\ \delta u_i=\frac{u_i-u_{i-1}}{\Delta t}\ ,\ \delta^2u_i=\frac{\delta u_i-\delta u_{i-1}}{\Delta t}.
\end{eqnarray*}
The decoupled semi-discrete approximation to the equations (\ref{equation:eq-10})-(\ref{equation:eq-11}) reads: for the given ${\bf E}_0$, ${\bf P}_0,$ ${\bf E}_0'$, and for any ${\bf \Phi}\in X,{\bf \Psi}\in Y$, find ${\bf E}_i\in X$, ${\bf P}_i\in Y$,
$1\leq i\leq n$,  such that
\begin{eqnarray}
\mu_0\epsilon_0(\delta^2{\bf E}_i,{\bf \Phi})
                   &+&A_1(\delta{\bf E}_i,{\bf \Phi})+(\nabla\times{\bf E}_i,\nabla\times{\bf \Phi})
-A_2(f'({\bf P}_{i-1}){\bf E}_i,{\bf \Phi})\nonumber\\
\label{equation:eq-12}
&+&\frac{\mu_0}{\tau^2}(f'{\bf P}_{i-1})f({\bf P}_{i-1}),{\bf \Phi})=-\mu_0({\bf g}_i,{\bf \Phi}),\\
\label{equation:eq-13}
&&\tau(\delta{\bf P}_i,{\bf \Psi})+(f({\bf P}_i),{\bf \Psi})=\epsilon_0(\epsilon_s-\epsilon_\infty)({\bf E}_i,{\bf \Psi}),
\end{eqnarray}
where $A_1=\frac{\epsilon_0\mu_0(\epsilon_s-\epsilon_\infty)}{\tau}$, $A_2=\frac{\epsilon_0\mu_0(\epsilon_s-\epsilon_\infty)}{\tau^2}$, ${\bf g}=\partial_t{\bf J}_s$.


Next, we give some lemmas about the boundedness of the potential fields ${\bf E}_i$ and $ {\bf P}_i$.
\subsection{The stability of the solutions of semi-discretization in time}

\begin{lemma}
For $j=1,\cdots,n$, there exists a positive constant $C$ depending on the parameter $\epsilon_0,\epsilon_s,\epsilon_\infty,\tau,\mu_0,$ and $\|{\bf g}\|_{L(0,T;{\bf L}^2(\Omega))},\|{\bf E}_0\|_0$, $\|{\bf P}_0\|_0$   such that
\begin{eqnarray}
& \ &\frac{\mu_0\epsilon_0}{2}\|{\bf E}_j\|_0^2+\frac{\epsilon_0\mu_0}{2}\sum\limits_{i=1}^j\|{\bf E}_i-{\bf E}_{i-1}\|_0^2\nonumber\\
\label{equation:eq-14}
&+&\frac{\Delta t^2}{2}\sum\limits_{i=1}^j\|\nabla\times{\bf E}_i\|_0^2
 +\frac{\tau}{2}\|{\bf P}_j\|_0^2+\frac{\tau}{2}\sum_{i=1}^j\|{\bf P}_i-{\bf P}_{i-1}\|_0^2\leq C.
\end{eqnarray}

\end{lemma}
\begin{proof}
Considering the semi-discrete scheme of equation (\ref{equation:eq-12}),
substituting ${\bf \Phi}$ by ${\bf E}_i$, defining a ghost point ${\bf E}_{-1}={\bf E}_0$ and making summation for $k=1,\cdots,i$,  we have
\begin{eqnarray*}
&&\frac{\epsilon_0\mu_0}{\Delta t}(\delta{\bf E}_i,{\bf E}_i)+\frac{A_1}{\Delta t}({\bf E}_i,{\bf E}_i)-\frac{A_1}{\Delta t}({\bf E}_0,{\bf E}_i)-A_2(\sum_{k=1}^if'({\bf P}_{k-1}){\bf E}_k,{\bf E}_i)\\
&&+\frac{\mu_0}{\tau^2}(\sum_{k=1}^if({\bf P}_{k-1})f'({\bf P}_{k-1}),{\bf E}_i)+(\sum_{k=1}^i\nabla\times{\bf E}_k,\nabla\times{\bf E}_i)=-\mu_0(\sum_{k=1}^i{\bf g}_k,{\bf E}_i)\ ,
\end{eqnarray*}
%
which is equal to
\begin{eqnarray*}
&&\frac{\epsilon_0\mu_0}{2\Delta t^2}\|{\bf E}_i-{\bf E}_{i-1}\|_0^2+\frac{\epsilon_0\mu_0}{2\Delta t^2}\|{\bf E}_i\|_0^2-\frac{\epsilon_0\mu_0}{2\Delta t^2}\|{\bf E}_{i-1}\|_0^2\\
&&+\frac{A_1}{\Delta t}\|{\bf E}_i\|_0^2-\frac{A_1}{\Delta t}({\bf E}_0,{\bf E}_i)-A_2(\sum_{k=1}^if'({\bf P}_{k-1}){\bf E}_k,{\bf E}_i)+\frac{\mu_0}{\tau^2}(\sum_{k=1}^if({\bf P}_{k-1})f'({\bf P}_{k-1}),{\bf E}_i)\\
&&+\frac{1}{2}\|\nabla\times{\bf E}_i\|_0^2+\frac{1}{2}\|\sum_{k=1}^i\nabla\times{\bf E}_k\|_0^2-\frac{1}{2}\|\sum_{k=1}^{i-1}\nabla\times{\bf E}_k\|_0^2=-\mu_0(\sum_{k=1}^i{\bf g}_k,{\bf E}_i).
\end{eqnarray*}
Summing $i=1,\cdots,j$ up again, from Abel's summation, we have
\begin{eqnarray}
&\ &\sum_{i=1}^j\Delta t^2A_2(\sum_{k=1}^if'({\bf P}_{k-1}){\bf E}_k,{\bf E}_i)\nonumber\\
&\leq&\sum_{i=1}^j\Delta t^2A_2\cdot B(\sum_{k=1}^i{\bf E}_k,\sum_{k=1}^i{\bf E}_k-\sum_{k=1}^{i-1}{\bf E}_k)\nonumber\\
\label{equation:eq-18}
&=&\frac{\Delta t^2A_2}{2}\cdot B\sum_{i=1}^j\|{\bf E}_i\|_0^2+\frac{\Delta t^2A_2}{2}\cdot B\|\sum_{i=1}^j{\bf E}_i\|_0^2
\end{eqnarray}
and
\begin{eqnarray}
&&\frac{\epsilon_0\mu_0}{2}\|{\bf E}_j\|_0^2+\frac{\epsilon_0\mu_0}{2}\sum_{i=1}^j\|{\bf E}_i-{\bf E}_{i-1}\|_0^2
+\frac{\Delta t^2A_2}{2}\cdot B\sum_{i=1}^j\|{\bf E}_i\|_0^2+\Delta tA_1\sum_{i=1}^j\|{\bf E}_i\|_0^2\nonumber\\
&&\hspace{0.5cm}+\frac{\Delta t^2A_2}{2}\cdot B\|\sum_{i=1}^j{\bf E}_i\|_0^2
+\frac{\Delta t^2}{2}\sum_{i=1}^j\|\nabla\times{\bf E}_i\|_0^2\nonumber\\
&&\leq\frac{\epsilon_0\mu_0}{2}\|{\bf E}_0\|_0^2+\Delta tA_1\sum_{i=1}^j({\bf E}_0,{\bf E}_i)
+\frac{\Delta t^2\mu_0}{\tau^2}\sum_{i=1}^j|(\sum_{k=1}^if({\bf P}_{k-1})f'({\bf P}_{k-1}),{\bf E}_i)|\nonumber\\
\label{equation:eq-15}
&&\hspace{0.5cm}+\mu_0\Delta t^2\sum_{i=1}^j|(\sum_{k=1}^i{\bf g}_k,{\bf E}_i)|=\sum_{i=1}^4S_i\ .
\end{eqnarray}
Now we can analyze the each term of (\ref{equation:eq-15}). The first term $S_1=\frac{\epsilon_0\mu_0}{2}\|{\bf E}_0\|_0^2 $
is trivial.
Using the Young's inequality, we have
\begin{eqnarray}
\label{equation:eq-16}
&&S_2=\Delta tA_1\sum_{i=1}^j({\bf E}_0,{\bf E}_i)\leq\frac{A_1T}{4\epsilon_1}\|{\bf E}_0\|_0^2+\epsilon_1\Delta tA_1\sum_{i=1}^j\|{\bf E}_i\|_0^2,\\
\label{equation:eq-17}
&&S_4=\mu_0\Delta t^2\sum_{i=1}^j|(\sum_{k=1}^i{\bf g}_k,{\bf E}_i)|\leq\frac{\mu_0\Delta t^2}{4\epsilon_2}\sum_{i=1}^j\|\sum_{k=1}^i{\bf g}_k\|_0^2+\epsilon_2\mu_0\Delta t^2\sum_{i=1}^j\|{\bf E}_i\|_0^2.
\end{eqnarray}
Note that $0<f'({ x})<B$ for any ${ x}\in R^d$,
by employing the boundedness of $f$ s.t. $|f({ x})|\leq M|{ x}|,\forall x\in R^d$, we have
\begin{eqnarray}
&&S_3=\frac{\Delta t^2\mu_0}{\tau^2}\sum_{i=1}^j|(\sum_{k=1}^if({\bf P}_{k-1})f'({\bf P}_{k-1}),{\bf E}_i)|\leq\frac{\Delta t^2\mu_0}{\tau^2}MB\sum_{i=1}^j|(\sum_{k=1}^i{\bf P}_{k-1},{\bf E}_i)|\nonumber\\
&&\ \ \ \ \ \leq\frac{\epsilon_3\Delta t^2\mu_0}{\tau^2}MBC_5(\sum_{i=1}^j\|{\bf P}_{i-1}\|_0)^2+\frac{\Delta t^2\mu_0}{4\epsilon_3\tau^2}MBC_5(\sum_{i=1}^j\|{\bf E}_i\|_0)^2\nonumber\\
\label{equation:eq-19}
&&\ \ \ \ \ \leq\frac{\Delta t\mu_0}{4\epsilon_3\tau^2}MBC_5T(\|{\bf P}_0\|_0^2
+\sum_{i=1}^{j-1}\|{\bf P}_i\|_0^2)
+\frac{\epsilon_3\Delta t\mu_0}{\tau^2}MBC_5T\sum_{i=1}^j\|{\bf E}_i\|_0^2.
\end{eqnarray}

On the other hand,  taking ${\bf \Psi}={\bf P}_i$ in the equation (\ref{equation:eq-13}),  we have
\begin{eqnarray*}
\tau(\delta{\bf P}_i,{\bf P}_i)+(f({\bf P}_i),{\bf P}_i)=\epsilon_0(\epsilon_s-\epsilon_\infty)({\bf E}_i,{\bf \Psi})\ ,
\end{eqnarray*}
which means
\begin{eqnarray*}
\frac{\tau}{2\Delta t}\|{\bf P}_i-{\bf P}_{i-1}\|_0^2+\frac{\tau}{2\Delta t}\|{\bf P}_i\|_0^2-\frac{\tau}{2\Delta t}\|{\bf P}_{i-1}\|_0^2+(f({\bf P}_i),{\bf P}_i)=\epsilon_0(\epsilon_s-\epsilon_\infty)({\bf E}_i,{\bf P}_i)\ .
\end{eqnarray*}
Summing for $i=1,\cdots,j$, from (\ref{equation:eq-006}), using the Young's inequality, we have
\begin{eqnarray}
&&\frac{\tau}{2}\|{\bf P}_j\|_0^2+\frac{\tau}{2}\sum_{i=1}^j\|{\bf P}_i-{\bf P}_{i-1}\|_0^2\leq\frac{\tau}{2}\|{\bf P}_0\|_0^2
\label{equation:eq-20}
+\frac{\epsilon_0(\epsilon_s-\epsilon_\infty)\Delta t}{4\epsilon_4}\sum_{i=1}^j\|{\bf E}_i\|_0^2+(\epsilon_0(\epsilon_s-\epsilon_\infty)\epsilon_4)\Delta t\sum_{i=1}^j\|{\bf P}_i\|_0^2,
\end{eqnarray}
where $\epsilon_1,\cdots,\epsilon_4$ are the constants, which satisfies $\frac{\epsilon_0\mu_0}{2}
-\epsilon_1\Delta tA_1-\epsilon_2\mu_0\Delta t^2-\epsilon_3\Delta t
\mu_0MBC_5T-\frac{\epsilon_0(\epsilon_s-\epsilon_\infty)\Delta t}
{4\epsilon_4}>0$ and $\frac{\tau}{2}
-\epsilon_4\Delta t\epsilon_0(\epsilon_s-\epsilon_\infty)>0$.

Adding (\ref{equation:eq-16})-(\ref{equation:eq-20}) and combining (\ref{equation:eq-15}) and the discrete Gr\"{o}nwall's inequality,
 we can complete the proof.
 \end{proof}

\begin{lemma}
For $j=1,\cdots,n$, there exists a positive constant $C$ depending on the parameter $\epsilon_0,\epsilon_s,\epsilon_\infty,\tau,\mu_0,$ and $\|{\bf g}\|_{L(0,T;{\bf L}^2(\Omega))},\|{\bf E}_0\|_0$, $\|{\bf P}_0\|_0$   such that
\begin{eqnarray*}
&&\frac{\mu_0\epsilon_0}{2}\|\delta{\bf E}_j\|_0^2+\frac{\mu_0\epsilon_0}{2}\sum\limits_{i=1}^j\|\delta{\bf E}_i-\delta{\bf E}_{i-1}\|_0^2+\frac{1}{2}\|\nabla\times{\bf E}_j\|_0^2+\frac{1}{2}\sum\limits_{i=1}^j\|\nabla\times({\bf E}_i-{\bf E}_{i-1})\|_0^2\leq C.
\end{eqnarray*}

\end{lemma}

\begin{proof}
Considering the semi-discretescheme of (\ref{equation:eq-12}),
 substituting $\Phi$ by $\delta{\bf E}_i$ and making summation for $i=1,\cdots,j$, we have
\begin{eqnarray}
&&\frac{\epsilon_0\mu_0}{\Delta t}\sum_{i=1}^j(\delta{\bf E}_i-\delta{\bf E}_{i-1},\delta{\bf E}_i)+A_1\sum_{i=1}^j(\delta{\bf E}_i,\delta{\bf E}_i)-A_2\sum_{i=1}^j(f'({\bf P}_{i-1}){\bf E}_i,\delta{\bf E}_i) \nonumber\\
&&+\sum_{i=1}^j(\nabla\times{\bf E}_i,\nabla\times\delta{\bf E}_i)+\frac{\mu_0}{\tau^2}\sum_{i=1}^j(f({\bf P}_{i-1})f'({\bf P}_{i-1}),\delta{\bf E}_i)=-\mu_0\sum_{i=1}^j({\bf g}_i,\delta{\bf E}_i),
\end{eqnarray}
that is
\begin{eqnarray*}
&&\frac{\epsilon_0\mu_0}{2}\|\delta{\bf E}_j\|_0^2+\frac{\epsilon_0\mu_0}{2}\sum_{i=1}^j\|\delta{\bf E}_i-\delta{\bf E}_{i-1}\|_0^2
+A_1\sum_{i=1}^j\|\delta{\bf E}_i\|_0^2
+\frac{1}{2}\sum_{i=1}^j\|\nabla\times({\bf E}_i-{\bf E}_{i-1})\|_0^2
+\frac{1}{2}\|\nabla\times{\bf E}_j\|_0^2\\
&&\leq \frac{1}{2}\|\nabla\times{\bf E}_0\|_0^2+A_2\Delta t\sum_{i=1}^j(f'({\bf P}_{i-1}){\bf E}_i,\delta{\bf E}_i)+|\frac{\mu_0\Delta t}{\tau^2}\sum_{i=1}^j(f({\bf P}_{i-1})f'({\bf P}_{i-1}),\delta{\bf E}_i)|\\
&&\ \ \ +|\mu_0\Delta t\sum_{i=1}^j({\bf g}_i,\delta{\bf E}_i)|\ =\sum\limits_{i=1}^5S_i.
\end{eqnarray*}%
By the similar proof to Lemma 3.1, and by employing
 the Gr\"{o}nwall's inequality, we can complete the proof.

\end{proof}

\begin{lemma}
For $i=1,\cdots,n$, there exists a positive constant $C$ depending on the parameter $\epsilon_0,\epsilon_s,\epsilon_\infty,\tau,\|{\bf g}\|_{L(0,T;{\bf L}^2(\Omega))},\|{\bf E}_0\|_0$ and $\|{\bf P}_0\|_0$   such that
\begin{eqnarray}
\label{equation:eq-21}
\|\delta{\bf P}_i\|_0^2\leq C.
\end{eqnarray}
\end{lemma}

\begin{proof}

For the equation (\ref{equation:eq-13}), substituting ${\bf \Psi}$ by $\delta {\bf P}_i$, we have
\begin{eqnarray*}
\tau(\delta{\bf P}_i,\delta{\bf P}_i)+(f({\bf P}_i),\delta{\bf P}_i)&=&\epsilon_0(\epsilon_s-\epsilon_\infty)({\bf E}_i,\delta{\bf P}_i).
\end{eqnarray*}
 Using Cauchy inequality, Young inequality  and the boundedness of $f$, we have
\begin{eqnarray}
\tau\|\delta{\bf P}_i\|_0^2&\leq&|(f({\bf P}_i),\delta{\bf P}_i)|+\epsilon_0(\epsilon_s-\epsilon_\infty)({\bf E}_i,\delta{\bf P}_i)\nonumber\\
&\leq&\|f({\bf P}_i)\|_0\|\delta{\bf P}_i\|_0+\epsilon_0(\epsilon_s-\epsilon_\infty)\|{\bf E}_i\|_0\|\delta{\bf P}_i\|_0\nonumber\\
\label{equation:eq-22}
&\leq&\frac{M}{4\epsilon_5}\|{\bf P}_i\|_0^2+\epsilon_5\|\delta{\bf P}_i\|_0^2+\frac{\epsilon_0(\epsilon_s-\epsilon_\infty)}{4\epsilon_6}\|{\bf E}_i\|_0^2+\epsilon_0(\epsilon_s-\epsilon_\infty)\epsilon_6\|\delta{\bf P}_i\|_0^2,
\end{eqnarray}
where $\epsilon_5,\epsilon_6$ are constants, which satisfies $\tau-\epsilon_5-\epsilon_6\epsilon_0(\epsilon_s-\epsilon_\infty)>0$. Then by the Lemma 3.1 and Gr\"{o}wall's inequality, we can complete the proof.

\end{proof}

\begin{lemma}
\label{lemma:lem3-4}For $i=1,2,\cdots,n$, there exits a constant $C>0$ such that
\begin{eqnarray}
\label{equation:delta2-bound}
\|\delta^2{\bf E}_i\|_{X^*}\leq C.
\mathcal{}\end{eqnarray}
\end{lemma}
\begin{proof}
From (\ref{equation:eq-12}), and Lemma 3.1-3.3,  we have
\begin{eqnarray}
|\mu_0\epsilon_0(\delta^2{\bf E}_i,{\bf \Phi})|&\leq &
                   |(\nabla\times{\bf E}_i,\nabla\times{\bf \Phi})|
                   +|\frac{\epsilon_0\mu_0(\epsilon_s-\epsilon_\infty)}{\tau}
                   (\delta{\bf E}_i,{\bf \Phi})|
+|\frac{\epsilon_0\mu_0(\epsilon_s-\epsilon_\infty)}{\tau^2}(f'({\bf P}_{i-1}){\bf E}_i,{\bf \Phi})|\nonumber\\
&+&|\frac{\mu_0}{\tau^2}(f'{\bf P}_{i-1})f({\bf P}_{i-1}),{\bf \Phi})|+|\mu_0(\partial_t{\bf J}_s,{\bf \Phi})|\nonumber\\
&\leq & C(\|{\bf \Phi}\|_0^2+\|\nabla\times{\bf \Phi}\|_0^2)^{\frac{1}{2}}.
\end{eqnarray}
By the definition of the operator norm in $ X^{*}$, which is the dual space of $X$, we can finish the proof.

\end{proof}
\begin{lemma}
\label{lemma:lem5}
There exists a positive constant $C>0$ such that
\begin{eqnarray}
\label{equation:delta-bound}
\|\delta{\bf E}_i\|_{X^{*}}\leq C.
\end{eqnarray}
\end{lemma}
\begin{proof}
From (\ref{equation:eq-12}) we have
\begin{eqnarray*}
&&A_1(\delta{\bf E}_i,{\bf \Phi})=-\mu_0(g_i,{\bf \Phi})-\epsilon_0\mu_0(\delta^2{\bf E}_i,{\bf \Phi})+A_2(f'({\bf P}_{i-1}){\bf E}_i,{\bf \Phi})\\
&&\ \ \ \ \ \ \ \ \ \ \ \ \ \ \ \ \ \ \ \ \ \ -\frac{\mu_0}{\tau^2}(f({\bf P}_{i-1})f'({\bf P}_{i-1}),{\bf \Phi})
-(\nabla\times{\bf E}_i,\nabla\times{\bf \Phi})\\
&&\ \ \ \ \ \ \ \ \ \ \ \ \ \ \ \ \ \ \leq C(\|{\bf g}_i\|_0^2+\|\delta^2{\bf E}_i\|_0^2+\|{\bf E}_i\|_0^2+\|{\bf P}_{i-1}\|_0^2+\|\nabla\times{\bf E}_i\|_0^2)^{\frac{1}{2}}\|{\bf \Phi}\|_{X}.
\end{eqnarray*}
Considering the Lemma 3.1-3.3, we have
\begin{eqnarray*}
(\delta{\bf E}_i,{\bf \Phi})\leq C\|{\bf \Phi}\|_{X},
\end{eqnarray*}
which finishes the proof.
\end{proof}

\subsection{The existence and uniqueness of semi-discretization scheme in time}

We give the existence and uniqueness of the equation (15)- (16) in this subsection.
\begin{theorem}

The weak form $(\ref{equation:eq-12})$ has a unique solution ${\bf E}_i$, for each $1\leq i\leq n$.

\end{theorem}

\begin{proof}
 Considering the bilinear  form
\begin{eqnarray}
a({\bf E},{\bf \Phi})=\frac{\epsilon_0\mu_0}{\Delta t^2}({\bf E},{\bf\Phi})+(\nabla\times{\bf E},\nabla\times{\bf \Phi})+\frac{\epsilon_0\mu_0(\epsilon_s-\epsilon_\infty)}{\tau\Delta t}({\bf E},{\bf\Phi})-\frac{\epsilon_0\mu_0(\epsilon_s-\epsilon_\infty)}{\tau^2}(f'({\bf P}_{i-1}){\bf E},{\bf \Phi})\ ,
\end{eqnarray}
for any ${\bf E},{\bf \Phi}\in X$.
Using the boundedness of $f'({\bf P})$, we have
\begin{eqnarray*}
&&a({\bf E},{\bf E})=(\nabla\times{\bf E},\nabla\times{\bf E})+\frac{\epsilon_0\mu_0}{\Delta t^2}({\bf E},{\bf E})+\frac{\epsilon_0\mu_0(\epsilon_s-\epsilon_\infty)}{\tau\Delta t}({\bf E},{\bf E})-\frac{\epsilon_0\mu_0(\epsilon_s-\epsilon_\infty)}{\tau^2}(f'({\bf P}_{i-1}){\bf E},{\bf E})\\
&&\ \ \ \ \ \ \ \ \ \ \geq\|\nabla\times{\bf E}\|_0^2+\frac{\epsilon_0\mu_0}{\Delta t^2}\|{\bf E}\|_0^2+\frac{\epsilon_0\mu_0(\epsilon_s-\epsilon_\infty)}{\tau\Delta E}\|{\bf E}\|_0^2-\frac{\epsilon_0\mu_0(\epsilon_s-\epsilon_\infty)}{\tau^2}B\|{\bf E}\|_0^2\\
&&\ \ \ \ \ \ \ \ \ \ =\|\nabla\times{\bf E}\|_0^2+(\frac{\epsilon_0\mu_0}{\Delta t^2}+\frac{\epsilon_0\mu_0(\epsilon_s-\epsilon_\infty)}{\tau\Delta t}-\frac{\epsilon_0\mu_0(\epsilon_s-\epsilon_\infty)}{\tau^2}B)\|{\bf E}\|_0^2\\
&&\ \ \ \ \ \ \ \ \ \ \geq\min\{1,C_0\}(\|\nabla\times{\bf E}\|_0^2+\|{\bf E}\|_0^2),
\end{eqnarray*}
where taking   $0<\Delta t<\frac{\tau(\epsilon_s-\epsilon_\infty)+\tau\sqrt{{(\epsilon_s-\epsilon_\infty)}^2+4B(\epsilon_s-\epsilon_\infty)}} {2B(\epsilon_s-\epsilon_\infty)}$ such that $C_0=(\frac{\epsilon_0\mu_0}{\Delta t^2}+\frac{\epsilon_0\mu_0(\epsilon_s-\epsilon_\infty)}{\tau\Delta t}-\frac{\epsilon_0\mu_0(\epsilon_s-\epsilon_\infty)}{\tau^2}B)>0 $,
which implies that the bilinear form is coercive. It is easy to see that the bilinear form is bounded. According to the Lax-Milgram lemma, we can complete the proof.

\end{proof}

\begin{theorem}
 The weak form $(\ref{equation:eq-13})$ has a unique solution ${\bf P}_i$, for each $1\leq i\leq n$.
\end{theorem}
\begin{proof}
Let $\mathcal L$ be an operator from $Y$ to $Y^*$, where $Y^*$ is the dual space of $Y$, defined as
\begin{eqnarray*}
&&\langle{\bf \mathcal L}{\bf u},{\bf w}\rangle=(\tau{\bf u},{\bf w})+\Delta t(f({\bf u}),{\bf w}),\ \ \ \forall {\bf {u,w}}\in Y.
\end{eqnarray*}
Then the strict monotonicity of $\mathcal L$ comes directly from the monotonicity of the function $f$. That is, for any ${\bf v},\ {\bf w}\in Y$,
\begin{eqnarray}
\label{equation:eq-26}
\langle{\bf {\mathcal L}}{\bf v}-{\bf {\mathcal L}}{\bf w},{\bf v}-{\bf w}\rangle&=&(\tau({\bf v}-{\bf w}),{\bf v}-{\bf w})+\Delta t(f({\bf v})-f({\bf w}),{\bf v}-{\bf w})\geq 0.
\end{eqnarray}
Moreover, from (\ref{equation:eq-008}), we have
\begin{eqnarray}
\label{equation:eq-27}
|\langle{\bf {\mathcal L}}{\bf u},{\bf w}\rangle|&=&(\tau{\bf u},{\bf w})+\Delta t(f({\bf u}),{\bf w})
\leq\tau\|{\bf u}\|_0\|{\bf w}\|_0+\Delta tM\|{\bf u}\|_0\|{\bf w}\|_0,
\end{eqnarray}
which leads to the boundedness of ${\bf {\mathcal L}}$
\begin{eqnarray}
\label{equation:eq-29}
\|{\bf {\mathcal L}}{\bf u}\|_{Y'}&\leq&\tau\|u\|_0+\Delta t M\|{\bf u}\|_0
\leq C\|{\bf u}\|_0.
\end{eqnarray}
 Furthermore, from (\ref{equation:eq-006}), we have
\begin{eqnarray}
\label{equation:eq-30}
\langle{\bf {\mathcal L}}{\bf u},{\bf u}\rangle&=&(\tau{\bf u},{\bf u})+\Delta t(f({\bf u}).{\bf u})
\geq C\|{\bf u}\|_Y,
\end{eqnarray}
which means that $\mathcal L$ is a coercive operator. And in the end we prove that the operator $\mathcal L$ is hemi-continuous, namely,
\begin{eqnarray*}
&&F(t)=\langle{\bf {\mathcal L}}({\bf u}+t{\bf v}),{\bf w}\rangle
\end{eqnarray*}
is continuous on $[0,1]$ for any ${\bf u},{\bf v},{\bf w}\in Y$.

For convenience, we write ${\bf u}(t)={\bf u}+t{\bf v}$, then ${\bf u}(t_0)={\bf u}+t_0{\bf v}\ and\ |{\bf u}|\leq|{\bf u}|+|{\bf v}|,\ {\bf u}(t)-{\bf u}(t_0)=(t-t_0){\bf v}. $
we have
\begin{eqnarray}
|F(t)-F(t_0)|&=&\langle{\bf {\mathcal L}}{\bf u}(t)-{\bf {\mathcal L}}{\bf u}(t_0),{\bf w}\rangle\nonumber\\
&=&(\tau({\bf u}(t)-{\bf u}(t_0)),{\bf w})+\Delta t((f({\bf u}(t))-f({\bf u}(t_0)),{\bf w})\nonumber\\
&\leq&|t-t_0|(\tau{\bf v},{\bf w})+\Delta t(|f({\bf u}(t))-f({\bf u}(t_0))|,{\bf w})\nonumber\\
&\leq&|t-t_0|(\tau{\bf v},{\bf w})+C_L\Delta t(|{\bf u}(t)-{\bf u}(t_0)|,{\bf w})\nonumber\\
\label{equation:eq-31}
&=&|t-t_0|[\tau({\bf v},{\bf w})+C_L\Delta t({\bf v},{\bf w})].
\end{eqnarray}
The above means that the $\mathcal L$ is monotone, bounded, coercive and hemi-continuous from $Y$ to $Y^*$. From the Minty-Browder method, we declare that there exists a unique solution $\ {\bf P}_i$ in $Y$ for the problem (\ref{equation:eq-13})  according to  Theorem 18.2 in \cite{MR028}.

\end{proof}

\subsection{The convergence of semi-discretization in time}

We start with constructing the piecewise-linear or piecewise-constant functions in time
\begin{eqnarray*}
&&\bar{\bf E}_n(t)={\bf E}_i,\ t\in(t_{i-1},t_i],\ \bar{\bf E}_n(0)={\bf E}_n(0)={\bf E}_0,\\
&&{\bf E}_n(t)={\bf E}_{i-1}+(t-t_{i-1})\delta{\bf E}_i,\ t\in(t_{i-1},t_i],\\
&&\widetilde{\bf E}_n(t)=\delta{\bf E}_{i-1}+(t-t_{i-1})\delta(\delta{\bf E}_i),\ t\in(t_{i-1},t_i],\\
&&\bar{\bf P}_n(t)={\bf P}_i,\ t\in(t_{i-1},t_i],\ \bar{\bf P}_n(0)={\bf P}_n(0)={\bf P}_0,\\
&&{\bf P}_n(t)={\bf P}_{i-1}+(t-t_{i-1})\delta{\bf P}_i,\ t\in(t_{i-1},t_i],\\
&&\bar f_n(t)=f({\bf P}_i),\widetilde f_n(t)=f'({\bf P}_i),\bar {\bf g}_n(t)={\bf g}(t_i),t\in(t_{i-1},t_i].
\end{eqnarray*}
Using this notation we are able to rewrite (\ref{equation:eq-12}) and (\ref{equation:eq-13}) as
\begin{eqnarray}
&&\epsilon_0\mu_0(\partial_t\widetilde{\bf E}_n,{\bf \Phi})+A_1(\partial_t{\bf E}_n,{\bf \Phi})-A_2(\widetilde f_n(t-\Delta t)\bar{\bf E}_n,{\bf \Phi})+(\nabla\times\bar{\bf E}_n,\nabla\times{\bf \Phi})\nonumber\\
\label{equation:eq-391}
&&\ \ \ \ \ \ \ \ \ \ \ \ \ \ \ =-\mu_0(\bar {\bf g}_n,{\bf \Phi})-\frac{\mu_0}{\tau^2}(\bar f_n(t-\Delta t)\widetilde f_n(t-\Delta t),{\bf \Phi}),\\
\label{equation:eq-401}
&&\tau(\partial_t{\bf P}_n,{\bf \Psi})+(\bar f_n,{\bf \Psi})=\epsilon_0(\epsilon_s-\epsilon_\infty)(\bar{\bf E}_n,{\bf \Psi}).
\end{eqnarray}


\par Now, we are in a position to prove the convergence of the approximation solutions of (\ref{equation:eq-12}) and (\ref{equation:eq-13}) to the weak solutions of (\ref{equation:eq-10}) and (\ref{equation:eq-11}). The following theorem is the main result of this subsection. Referring to the framework in \cite{MR66}, we divide it into five parts.

\begin{theorem}
Suppose that the ${\bf {Definition}}$ 2.1 holds, and the function $f$ satisfies the (\ref{equation:eq-006})- (\ref{equation:eq-008}), then there exist subsequences of ${\bf E}_n$ and ${\bf P}_n$ such that
\begin{eqnarray*}
&&(A.)\ \ \bar{\bf E}_n(t)\rightarrow{\bf E}(t)\ in\ C^1(0,T; X),\\
&&\hspace{0.75cm}\bar{\bf P}_n(t)\rightharpoonup{\bf P}(t)\ in\ L^2(0,T;Y).\\
&&(B.)\ \ \ \partial_t{\bf E}_n(t)\rightharpoonup\partial_t{\bf E}(t)\ in\ L^2(0,T;X^*),\\
&&\hspace{0.75cm}\partial_t{\bf P}_n(t)\rightharpoonup\partial_t{\bf P}(t)\ in\ L^2(0,T;Y),\\
&&\hspace{0.75cm}\partial_t\widetilde{\bf E}_n(t)\rightharpoonup\partial_{tt}{\bf E}(t)\ in\ L^2(0,T;X^*).\\
&&(C.)\ \ \ \bar f_n(t)\rightharpoonup f({\bf P})\ in\ L^2(0,T;Y).\\
&&(D.)\ \ \ \bar{\bf P}_n\rightarrow{\bf P}\ in\ L^2(0,T;Y),\\
&&\hspace{0.75cm}\bar f_n(t)\rightarrow f({\bf P}),\ in\ L^2(0,T;Y),\\
&&\hspace{0.75cm}\widetilde f_n(t)\rightarrow f'({\bf P}),\ in\ L^2(0,T;Y).\\
&&(E.)\ \ \ {\bf E}\ and\ {\bf B}\ {solve}\ (\ref{equation:eq-10})- (\ref{equation:eq-11}).
\end{eqnarray*}
\end{theorem}
\begin{proof}
Part A. Thanks to the Lemma 3.1-3.2, we have
\begin{eqnarray*}
\int_0^T\|\partial_t\widetilde{\bf E}_n\|_0^2dt+\int_0^T\|\partial_t{\bf E}_n\|_0^2dt+\max_{t\in[0,T]}(\|\bar{\bf E}_n\|_0^2+\|\nabla\times\bar{\bf E}_n\|_0^2)\leq C.
\end{eqnarray*}
Hence we can apply Lemma 1.3.13 from \cite{MR65} to obtain ${\bf E}\in C^2(0,T;X)\bigcap L_\infty(0,T;X)$ with $\partial_t{\bf E}\in C^1(0,T;X)$ and $\partial_{tt}{\bf E}\in L^2(0,T;X)$ such that
\begin{eqnarray}
\label{equation:eq-677}
\bar{\bf E}_n\rightarrow{\bf E}\ in\ C^1(0,T;X).
\end{eqnarray}
Moreover, considering Lemma 3.1, we have
\begin{eqnarray*}
\int_0^T\|\bar{\bf P}_n\|_0^2dt\leq C,
\end{eqnarray*}
which concludes that the sequence $\bar{\bf P}_n$ is bounded in $L^2(0,T;Y)$. Following the reflexivity of this space, we have
\begin{eqnarray}
\label{equation:eq-667}
&&\bar{\bf P}_n(t)\rightharpoonup{\bf P}(t)\ in\ L^2(0,T;Y).
\end{eqnarray}

Part B. Based on Lemma 3.2, there exists  a ${\bf w}\in L^2(0,T;X)$ such that
\begin{eqnarray}
\label{equation:eq-680}
\int_0^t(\partial_t{\bf E}_n,{\bf \Phi})ds\rightarrow\int_0^t({\bf w},{\bf \Phi})ds\ (n\rightarrow\infty),
\end{eqnarray}
since space $L^2(0,T;X)$ is reflexive.

The sequence ${\bf E}_n$ is equi-bounded in $X^*$
\begin{eqnarray*}
&\ &({\bf E}_n(t),{\bf \Phi})-({\bf E}_n(0),{\bf \Phi})=\int_0^t(\partial_t{\bf E}_n,{\bf \Phi})\\
&\leq&\int_0^t\|\partial_t{\bf E}_n\|_{X^*}\|{\bf \Phi}\|_Xds
\leq C\|{\bf \Phi}\|_X.
\end{eqnarray*}
Hence, we have
\begin{eqnarray*}
({\bf E}_n(t),{\bf \Phi})\leq C\|{\bf \Phi}\|_X+({\bf E}_0,{\bf \Phi})\leq C\|{\bf \Phi}\|_X,
\end{eqnarray*}
which brings us to get
\begin{eqnarray}
\label{equation:eq-681}
\|{\bf E}_n(t)\|_{X^*}\leq C.
\end{eqnarray}

${\bf E}_n$ is also equi-continous. In fact, for any $t_1,t_2\in[0,T]$, the following holds
\begin{eqnarray*}
|({\bf E}_n(t_2)-{\bf E}_n(t_1),{\bf \Phi})|&=& |\int_{t_1}^{t_2}(\partial_t{\bf E}_n,{\bf \Phi})ds|\leq\int_{t_1}^{t_2}\|\partial_t{\bf E}_n\|_{X^*}\|{\bf \Phi}\|_{X}ds\\
&\leq&(\int_{t_1}^{t_2}1^2ds)^{\frac{1}{2}}\cdot(\int_{t_1}^{t_2}\|\partial_t{\bf E}_n\|_{X^*}^2ds)^{\frac{1}{2}}\cdot\|{\bf \Phi}\|_{X}\\
&\leq&C|t_2-t_1|^{\frac{1}{2}}\cdot\|{\bf \Phi}\|_{X}.
\end{eqnarray*}
Then, we have
\begin{eqnarray}
\label{equation:eq-682}
\|{\bf E}_n(t_2)-{\bf E}_n(t_1)\|_{X^*}\leq C|t_2-t_1|^{\frac{1}{2}}.
\end{eqnarray}
From (\ref{equation:eq-681}) and (\ref{equation:eq-682}), using the modification of Arzela-Ascoli theorem (seeing Lemma 1.3.10 of \cite{MR64}), we have
\begin{eqnarray}
\label{equation:eq-683}
\lim_{n\rightarrow\infty}({\bf E}_n(t),{\bf \Phi})=({\bf E}(t),{\bf \Phi}),
\end{eqnarray}
for any ${\bf \Phi}\in X$ and for any $t\in[0,T]$.

Furthermore, from (\ref{equation:eq-683}), we have
\begin{eqnarray}
&&\int_0^t(w,{\bf \Phi})ds=\lim_{n\rightarrow\infty}\int_0^t(\partial_t{\bf E}_n,{\bf \Phi})ds=\lim_{n\rightarrow\infty}({\bf E}_n(t)-{\bf E}_n(0),{\bf \Phi})\nonumber\\
\label{equation:eq-684}
&&\ \ \ \ \ \ \ \ \   \ \  \ \ \ \ \ \ \ \ \ =({\bf E}(t)-{\bf E}(0),{\bf \Phi})=(\int_0^t\partial_t{\bf E}ds,{\bf \Phi}).
\end{eqnarray}
Now, we can conclude that ${\bf w}=\partial_t{\bf E}$ and
\begin{eqnarray}
\label{equation:eq-669}
\int_0^t(\partial_t{\bf E}_n,{\bf \Phi})ds\rightarrow\int_0^t(\partial_t{\bf E},{\bf \Phi})ds,\ \ \ \ in\ L^2(0,T;X).
\end{eqnarray}

Moreover, by the same way, we can conclude that
\begin{eqnarray}
\label{equation:eq-670}
\int_0^t(\partial_t{\bf P}_n,{\bf \Psi})ds\rightarrow\int_0^t(\partial_t{\bf P},{\bf \Psi})ds,\ \ \ \ in\ L^2(0,T;Y).
\end{eqnarray}

From Lemma 3.4, we have
\begin{eqnarray*}
&\ &\int_0^T\|\widetilde{\bf E}_n(t)-\partial_t{\bf E}_n(t)\|_{X^{*}}^2dt=\int_0^T\|\delta{\bf E}_{i-1}+(t-t_{i-1})\delta(\delta{\bf E}_i)-\delta{\bf E}_i\|_{X^{*}}^2dt\\
&\leq&\Delta t^2\int_0^T\|\delta^2{\bf E}_i\|_{X^{*}}^2dt
\leq C\Delta t^2\rightarrow0,
\end{eqnarray*}
which implies $\widetilde{\bf E}_n(t)$ and $\partial_t{\bf E}_n(t)$ share with the same limit and
\begin{eqnarray}
\label{equation:eq-668}
\int_0^t(\widetilde{\bf E}_n(t),{\bf \Phi})ds\rightarrow\int_0^t(\partial_t{\bf E}(t),{\bf \Phi})ds,\ \ \ in\ L^2(0,T;X).
\end{eqnarray}
And then using the same method, we conclude
\begin{eqnarray}
\label{equation:eq-671}
\int_0^t(\partial_t\widetilde{\bf E}_n(t),{\bf \Phi})ds\rightarrow\int_0^t(\partial_{tt}{\bf E}(t),{\bf \Phi})ds,\ \ \ in\ L^2(0,T;X).
\end{eqnarray}
Part C. Due to Lipschitz continuity of the function $f$ and Lemma 3.1, we can write
\begin{eqnarray}
\int_0^T\|f(\bar{\bf P}_n)\|_0^2dt&=&\sum_{i=1}^n\int_\Omega|f({\bf P}_i)|^2dx\Delta t\leq M\sum_{i=1}^n\int_\Omega|{\bf P}_i|^2dx\Delta t\nonumber\\
\label{equation:eq-041}
&\leq&M\sum_{i=1}^n\|{\bf P}_i\|_0^2\Delta t\leq C,
\end{eqnarray}
and
\begin{eqnarray}
\int_0^T\|{\bf P}_n\|_0^2dt&=&\sum_{i=1}^n\int_{t_{i-1}}^{t_i}\|{\bf P}_{i-1}+(t-t_{i-1})\delta{\bf P}_i\|_0^2\nonumber\\
&\leq&\sum_{i=1}^n(\|{\bf P}_{i-1}\|_0^2+\|{\bf P}_i-{\bf P}_{i-1}\|_0^2)\Delta t\nonumber\\
\label{equation:eq-042}
&\leq&\|{\bf P}_0\|_0^2+C\sum_{i=1}^n\|{\bf P}_i\|_0^2\Delta t\leq C.
\end{eqnarray}
There exist functions ${\bf B}$ and ${\bf P}$, for the subsequence of $f(\bar{\bf P}_n)$ and ${\bf P}_n$(still denoted with n) such that $f(\bar{\bf P}_n)\rightharpoonup{\bf B}$ and ${\bf P}_n\rightharpoonup{\bf P}$ in this space, since the space $L^2(0,T;Y)$ is a reflexive Banach space. With all the knowledge, we can invoke Lemma 3.1 in \cite{MR65} to prove that
\begin{eqnarray}
\label{equation:eq-696}
\lim_{n\rightarrow\infty}\int_0^T(f(\bar{\bf P}_n),\Phi{\bf P}_n)dt=\int_0^T({\bf B},\Phi{\bf P})dt,
\end{eqnarray}
for any $\Phi\in C_0^{\infty}(\bar\Omega)$. The technique, which we use in the following part of the proof, is called Minty-Browder, and it is based on the monotone character of the function $f$. We can write
\begin{eqnarray}
\label{equation:eq-685}
\lim_{n\rightarrow\infty}\int_0^T(f(\bar{\bf P}_n)-f({\bf q}),\Phi(\bar{\bf P}_n-{\bf q}))dt\geq0,
\end{eqnarray}
where ${\bf q}\in L^2(0,T;Y)$ is arbitrary and $\Phi\in C_0^{\infty}(\bar\Omega)$ is non-negative.
The basic idea is to split the left term of (\ref{equation:eq-685}) into four terms and then investigate them separately
\begin{eqnarray*}
&&I_1=\int_0^T(f(\bar{\bf P}_n),\Phi\bar{\bf P}_n)dt,\ \ \ \ \ \ I_2=\int_0^T(f({\bf q}),\Phi\bar{\bf P}_n)dt,\\
&&I_3=\int_0^T(f(\bar{\bf P}_n),\Phi{\bf q})dt,\ \ \ \ \ \ \ \
I_4=\int_0^T(f({\bf q}),\Phi{\bf q})dt.
\end{eqnarray*}
We can rewrite the first term as follows
\begin{eqnarray*}
I_1=\int_0^T(f(\bar{\bf P}_n),\Phi(\bar{\bf P}_n-{\bf P}_n))dt+\int_0^T(f(\bar{\bf P}_n),\Phi{\bf P}_n)dt.
\end{eqnarray*}
From (\ref{equation:eq-041}) and (\ref{equation:eq-042}) we have
\begin{eqnarray*}
|\int_0^T(f(\bar{\bf P}_n),\Phi(\bar{\bf P}_n-{\bf P}_n))dt|&\leq& C\int_0^T\|f(\bar{\bf P}_n)\|_0\|\bar{\bf P}_n-{\bf P}_n\|_0dt\\
&\leq&C\int_0^T\|\bar{\bf P}_n-{\bf P}_n\|_0dt\rightarrow0.
\end{eqnarray*}
Hence we can write
\begin{eqnarray}
\label{equation:eq-691}
\lim_{n\rightarrow\infty}I_1=\lim_{n\rightarrow\infty}\int_0^T(f(\bar{\bf P}_n),\Phi{\bf P}_n)dt=\int_0^T({\bf B},\Phi{\bf P})dt.
\end{eqnarray}

The space $L^2(0,T;C^\infty(\Omega))$ is dense in $L^2(0,T;Y)$. Thus for any $\epsilon>0$, there exists $f_\epsilon\in L^2(0,T;C^\infty(\Omega))$ such that $
\|f({\bf q})-f_\epsilon\|_{L^2(0,T;{\bf L}^2(\Omega))}\leq\epsilon.$
Let  us now investigate the following identity
\begin{eqnarray*}
I_2=\int_0^T(f_\epsilon,\Phi(\bar{\bf P}_n-{\bf P}_n))dt+\int_0^T(f({\bf q})-f_\epsilon,\Phi(\bar{\bf P}_n-{\bf P}_n))dt+\int_0^T(f({\bf q}),\Phi{\bf P}_n)dt.
\end{eqnarray*}
Using (\ref{equation:eq-042}) and the statement above, we can bound the first two terms of $I_2$
\begin{eqnarray*}
&\ &|\int_0^T(f_\epsilon,\Phi(\bar{\bf P}_n-{\bf P}_n))dt|\leq C\int_0^T\|f_\epsilon\|_0\|\bar{\bf P}_n-{\bf P}_n\|_0dt\\
&\leq&C_\epsilon\int_0^T\|\bar{\bf P}_n-{\bf P}_n\|_0dt\rightarrow0,
\end{eqnarray*}
\begin{eqnarray*}
|\int_0^T(f({\bf q})-f_\epsilon,\Phi(\bar{\bf P}_n-{\bf P}_n))dt|&\leq&C\int_0^T\|f({\bf q})-f_\epsilon\|_0\|\bar{\bf P}_n-{\bf P}_n\|_0dt\\
&\leq&C_\epsilon\int_0^T\|\bar{\bf P}_n-{\bf P}_n\|_0dt\rightarrow0.
\end{eqnarray*}
From $f({\bf q})\in L^2(0,T;Y)$, we have $\int_0^T(f({\bf q}),\Phi{\bf P}_n)dt\rightarrow\int_0^T(f({\bf q}),\Phi{\bf P})dt$. Therefore, we can pass to the limit
\begin{eqnarray}
\label{equation:eq-692}
\lim_{n\rightarrow\infty}I_2=\lim_{n\rightarrow\infty}\int_0^T(f({\bf q}),\Phi{\bf P}_n)dt=\int_0^T(f({\bf q}),\Phi{\bf P})dt.
\end{eqnarray}

We can easily see that
\begin{eqnarray}
\label{equation:eq-693}
\lim_{n\rightarrow\infty}I_3=\int_0^T({\bf B},\Phi{\bf q})dt,\ \
\lim_{n\rightarrow\infty}I_4=\int_0^T(f({\bf q}),\Phi{\bf q})dt.
\end{eqnarray}
Therefore, gathering all partial results of (\ref{equation:eq-691}), (\ref{equation:eq-692}) and  (\ref{equation:eq-693}), we get
\begin{eqnarray}
\label{equation:eq-694}
\lim_{n\rightarrow\infty}\int_0^T(f(\bar{\bf P}_n)-f({\bf q}),\Phi(\bar{\bf P}_n-{\bf q}))dt
=\int_0^T({\bf B}-f({\bf q}),\Phi({\bf P}-{\bf q}))dt\geq0.
\end{eqnarray}

Now, let ${\bf q}={\bf P}+\epsilon {\bf v}$,  $\forall {\bf v}\in L^2(0,T;Y)$ and $\epsilon>0$, then we have
\begin{eqnarray*}
\int_0^T({\bf B}-f({\bf P}+\epsilon {\bf v}),\Phi {\bf v})dt\leq0.
\end{eqnarray*}
Passing to $\epsilon\rightarrow0$ brings us to
\begin{eqnarray}
\label{equation:eq-695}
\int_0^T({\bf B}-f({\bf P}),\Phi {\bf v})dt\leq0.
\end{eqnarray}
Since the inequality of (\ref{equation:eq-695}) is valid for any ${\bf v}\in L^2(0,T;Y)$, we can replace ${\bf v}$ with $-{\bf v}$ and the reversed inequality also holds. Hence we get
\begin{eqnarray*}
\int_0^T({\bf B}-f({\bf P}),\Phi {\bf v})dt=0,
\end{eqnarray*}
which is true for any ${\bf v}\in L^2(0,T;Y)$ and all non-negative $\Phi\in C_0^\infty(\bar\Omega).$ Then we have ${\bf B}=f({\bf P}),$ a.e. in $\Omega\times(0,T)$. From (\ref{equation:eq-696}), we have
\begin{eqnarray}
\label{equation:eq-673}
\bar f_n(t)\rightharpoonup f({\bf P})\ in\ L^2(0,T;Y).
\end{eqnarray}

Part D. Let $\Phi\in C_0^\infty(\bar\Omega)$ be non-negative. From (\ref{equation:eq-006}), we have
\begin{eqnarray*}
\int_0^T(f(\bar{\bf P}_n)-f({\bf q}),\Phi(\bar{\bf P}_n-{\bf q}))dt\geq\omega_f\int_0^T(\Phi,|\bar{\bf P}_n-{\bf q}|^2)dt\geq0.
\end{eqnarray*}
Setting ${\bf q}={\bf P}$ and combining (\ref{equation:eq-006}), we also have
\begin{eqnarray*}
0=\lim_{n\rightarrow\infty}\int_0^T(f(\bar{\bf P}_n)-f({\bf P}),\Phi(\bar{\bf P}_n-{\bf P}))dt\geq\omega_f\lim_{n\rightarrow\infty}\int_0^T(\Phi,|\bar{\bf P}_n-{\bf P}|^2)dt\geq0.
\end{eqnarray*}
There holds
\begin{eqnarray}
\label{equation:eq-672}
\bar{\bf P}_n\rightarrow{\bf P}\ in\ L^2(0,T;Y),
\end{eqnarray}
since the inequality is valid for any non-negative $\Phi\in C_0^\infty(\bar\Omega)$.
Using (\ref{equation:eq-672}) and the continuous of $f$ and $f'$, we have
\begin{eqnarray}
\label{equation:eq-674}
\bar f_n(t)\rightarrow f({\bf P})\ in\ L^2(0,T;Y),\ \ \ \ \
\label{equation:eq-675}
\widetilde f_n(t)\rightarrow f'({\bf P})\ in\ L^2(0,T;Y).
\end{eqnarray}
Now, using the property of Lipschitz continuity of $f$, let us demonstrate that $\bar f_n(t-\Delta t)$ and $\bar f_n(t)$ share with the same limit in $L^2(0,T;Y)$. Actually, we have
\begin{eqnarray*}
&\ &\int_0^T \|\bar f_n(t-\Delta t)-\bar f_n(t)\|_0^2dt=\sum_{i=1}^n\|f({\bf P}_i)-f({\bf P}_{i-1})\|_0^2\Delta t\\
&\leq&C\sum_{i=1}^n\|{\bf P}_i-{\bf P}_{i-1}\|_0^2\Delta t
=C\Delta t^2\sum_{i=1}^n\|\delta{\bf P}_i\|_0^2\Delta t
\leq C\Delta t^2.
\end{eqnarray*}
Thus we have
\begin{eqnarray}
\label{equation:eq-692}
\lim_{n\rightarrow}\int_0^T\|\bar f_n(t-\Delta t)-\bar f_n(t)\|_0^2dt=0.
\end{eqnarray}
By the same way, using the Lipschitz continuity of $f'$, we can get to that $\widetilde f_n(t-\Delta t)$ and $\widetilde f_n(t)$ share with the same limit.\\
In the last, due to Lipschitz continuity of ${\bf g}$ we have
\begin{eqnarray}
\label{equation:eq-676}
\int_0^T\|\bar {\bf g}_n-{\bf g}\|_0^2=\sum_{i=1}^n\int_{t_{i-1}}^{t_i}\|{\bf g}(t_i)-{\bf g}(t)\|_0^2dt\leq C\Delta t^2\rightarrow0.
\end{eqnarray}

Part E. Taking $\Phi\in C_0^\infty(\bar\Omega)$ in (\ref{equation:eq-391}) and integrating it in time, $\xi\in[0,T]$, we have
\begin{eqnarray*}
&&\int_0^\xi\epsilon_0\mu_0(\partial_t\widetilde{\bf E}_n,{\bf \Phi})+A_1\int_0^\xi(\partial_t{\bf E}_n,{\bf \Phi})-A_2\int_0^\xi(\bar f_n(t-\Delta t)\bar{\bf E}_n,{\bf \Phi})+\int_0^\xi(\nabla\times\bar{\bf E}_n,\nabla\times{\bf \Phi})\\
&&\ \ \ \ \ \ \ \ \ \ \ \ \ \ \ \ \ \ \ \ \ \ =-\mu_0\int_0^\xi(\bar {\bf g}_n,{\bf \Phi})-\frac{\mu_0}{\tau^2}\int_0^\xi(\widetilde f_n(t-\Delta t)\bar f_n(t-\Delta t),{\bf \Phi}).
\end{eqnarray*}
Due to the results of (\ref{equation:eq-677}), (\ref{equation:eq-669}), (\ref{equation:eq-671}), (\ref{equation:eq-674}) and (\ref{equation:eq-676}), we can pass to the limit for $n\rightarrow\infty$ to have
\begin{eqnarray*}
&&\int_0^\xi\epsilon_0\mu_0(\partial_{tt}{\bf E},{\bf \Phi})+A_1\int_0^\xi(\partial_t{\bf E},{\bf \Phi})-A_2\int_0^\xi(f'({\bf P}){\bf E},{\bf \Phi})+\int_0^\xi(\nabla\times{\bf E},\nabla\times{\bf \Phi})\\
&&\ \ \ \ \ \ \ \ \ \ \ \ \ \ \ \ \ \ \ \ \ \ =-\mu_0\int_0^\xi({\bf g},{\bf \Phi})-\frac{\mu_0}{\tau^2}\int_0^\xi(f({\bf P})f'({\bf P}),{\bf \Phi}).
\end{eqnarray*}
Now by the fact that $C_0^\infty(\bar\Omega)$ is dense in $X$, differentiating with respect to time variable, we see that ${\bf E}$ and ${\bf P}$ solve (\ref{equation:eq-10}).

Moreover integrating (\ref{equation:eq-401}) in time, we have
\begin{eqnarray*}
\tau(\bar{\bf P}_n(t),{\bf \Psi})-\tau({\bf P}_n(0),{\bf \Psi})+({\bf P}_n(t)-\bar{\bf P}_n(t),{\bf \Psi})+\int_0^t(\bar f_n,{\bf \Psi})ds=\epsilon_0(\epsilon_s-\epsilon_\infty)\int_0^t(\bar{\bf E}_n,{\bf \Psi})d s.
\end{eqnarray*}
Using the results of (\ref{equation:eq-677}), (\ref{equation:eq-672}), (\ref{equation:eq-674}), and
$\lim\limits_{n\rightarrow\infty}({\bf P}_n(t)-\bar{\bf P}_n(t),{\bf \Psi})=0,$
for every $t\in[0,T],$
hence we can pass to the limit for $n\rightarrow\infty$ to obtain
\begin{eqnarray*}
\tau({\bf P},{\bf \Psi})-\tau({\bf P}_0,{\bf \Psi})+\int_0^t(f({\bf P}),{\bf \Psi})ds=\epsilon_0(\epsilon_s-\epsilon_\infty)\int_0^t({\bf E},{\bf \Psi})ds.
\end{eqnarray*}
Now, differentiating in time shows that ${\bf P}$ and ${\bf E}$ solve (\ref{equation:eq-11}).
\end{proof}


\subsection{ Error  estimates of semi-discretization in time}

Next , we will discuss the error estimates of the semi-discrete problem.
\begin{theorem}
\label{theorem:th4-2}
Assume that the equation (\ref{equation:eq-10}) holds, for any $t\in[0,T]$ we have
\begin{eqnarray}
\label{equation:eq-39}
\max\limits_{t\in[0,T]}\|{\bf E}_n-{\bf E}\|_0^2+\int_0^t\|{\bf E}_n-{\bf E}\|_0^2+\|\int_0^t\nabla\times(\bar{\bf E}_n-{\bf E})\|_0^2\leq C\Delta t^2,
\end{eqnarray}
where C is a positive constant which depends on a series of parameters $\epsilon_0,\mu_0,\epsilon_s,\epsilon_\infty,M$ and $B$.
\end{theorem}
\begin{proof}
 Subtracting (\ref{equation:eq-10}) from (\ref{equation:eq-391}) and integrating over $(0,t)$, it yields
\begin{eqnarray*}
&&\epsilon_0\mu_0(\partial_t{\bf E}_n-\partial_t{\bf E},{\bf \Phi})+\frac{\epsilon_0\mu_0(\epsilon_s-\epsilon_\infty)}{\tau}({\bf E}_n-{\bf E}(t),{\bf \Phi})\ \ \ \ \ \ \ \ \ \ \ \ \ \\
&&\ \ \ -\frac{\epsilon_0\mu_0(\epsilon_s-\epsilon_\infty)}{\tau^2}(\int_0^t\widetilde f_n(t-\Delta t)\bar{\bf E}_n-f'({\bf P}){\bf E},{\bf \Phi})+(\int_0^t\nabla\times(\bar{\bf E}_n-{\bf E}),\nabla\times{\bf \Phi})\\
&&=-\mu_0(\int_0^t\bar {\bf g}_n-{\bf g},{\bf \Phi})-\frac{\mu_0}{\tau^2}(\int_0^t\bar f_n(t-\Delta t) \widetilde f_n(t-\Delta t)-f({\bf P})f'({\bf P}),{\bf \Phi})+\epsilon_0\mu_0( \partial_t{\bf E}_n-\widetilde{\bf E}_n,{\bf \Phi}).
\end{eqnarray*}
Setting ${\bf \Phi}=\bar{\bf E}_n-{\bf E}$, and integrating in time again, we have
\begin{eqnarray}
&&\frac{\epsilon_0\mu_0}{2}\|{\bf E}_n-{\bf E}\|_0^2+\frac{\epsilon_0\mu_0(\epsilon_s-\epsilon_\infty)}{\tau}\int_0^t\|{\bf E}_n-{\bf E}\|_0^2+\|\int_0^t\nabla\times(\bar{\bf E}_n-{\bf E})\|_0^2\nonumber\\
&&\leq\mu_0|\int_0^t(\int_0^s\bar {\bf g}_n-{\bf g},\bar{\bf E}_n-{\bf E})|+\frac{\mu_0}{\tau^2}|\int_0^t(\int_0^s\bar f_n(t-\Delta t)\widetilde f_n(t-\Delta t)-f({\bf P})f'({\bf P}),\bar{\bf E}_n-{\bf E})|\nonumber\\
&&\ \ \ \ +\epsilon_0\mu_0\int_0^t(\partial_t{\bf E}_n-\widetilde{\bf E}_n,\bar{\bf E}_n-{\bf E})+\frac{\epsilon_0\mu_0(\epsilon_s-\epsilon_\infty)}{\tau^2}\int_0^t(\int_0^s
\widetilde f_n(t-\Delta t)\bar{\bf E}_n-f'({\bf P}){\bf E},\bar{\bf E}_n-{\bf E})\nonumber\\
&&\ \ \ \ +\epsilon_0\mu_0\int_0^t(\partial_t{\bf E}_n-\partial_t{\bf E},{\bf E}_n-\bar{\bf E}_n)+\frac{\epsilon_0\mu_0(\epsilon_s-\epsilon_\infty)}{\tau}\int_0^t({\bf E}_n-{\bf E},{\bf E}_n-\bar{\bf E}_n)\nonumber\\
\label{equation:eq-40}
&&=\sum_{i=1}^6S_i.
\end{eqnarray}
Using Young's inequality, Lemma 3.1-Lemma 3.4 and (\ref{equation:eq-006})-(\ref{equation:eq-008}), to deal with the each terms $S_i,i=1,2,\cdots, 6$, we have
\begin{eqnarray*}
&&S_1=\mu_0|\int_0^t(\int_0^s\bar {\bf g}_n-{\bf g},\bar{\bf E}_n-{\bf E})|\leq\mu_0(\int_0^t\|\int_0^s\bar {\bf g}_n-{\bf g}\|_0^2)^{\frac{1}{2}}\cdot(\int_0^t\| \bar{\bf E}_n-{\bf E}\|_0^2)^{\frac{1}{2}}\leq C\Delta t^2,\\
&&S_2=\frac{\mu_0}{\tau^2}|\int_0^t(\int_0^s\bar f_n(t-\Delta t)\widetilde f_n(t-\Delta t)-f({\bf P})f'({\bf P}),\bar{\bf E}_n-{\bf E})|\\
&&\ \ \ \ \leq\frac{\mu_0}{\tau^2}(\int_0^t\|\int_0^s \widetilde f_n(t-\Delta t)(\bar f_n(t-\Delta t)-f({\bf P}))\|_0^2)^{\frac{1}{2}}\cdot(\int_0^t\|\bar{\bf E}_n-{\bf E}\|_0^2)^{\frac{1}{2}}\\
&&\ \ \ \ \ \ \ +\frac{\mu_0}{\tau^2}(\int_0^t(\|\int_0^sf({\bf P})(\widetilde f_n(t-\Delta t)-f'({\bf P}))\|_0^2)^{\frac{1}{2}}\cdot(\int_0^t\|\bar{\bf E}_n-{\bf E}\|_0^2)^{\frac{1}{2}}\leq C\Delta t^2,\\
&&S_3=\epsilon_0\mu_0\int_0^t(\partial_t{\bf E}_n-\widetilde{\bf E}_n,\bar{\bf E}_n-{\bf E})\leq\epsilon_0\mu_0(\int_0^t\| \partial_t{\bf E}_n-\widetilde{\bf E}_n\|_0^2)^{\frac{1}{2}}\cdot(\int_0^t\|\bar{\bf E}_n-{\bf E}\|_0^2)^{\frac{1}{2}}\leq C\Delta t^2,\\
&&S_4=\frac{\epsilon_0\mu_0(\epsilon_s-\epsilon_\infty)}{\tau^2}\int_0^t(\int_0^s\widetilde f_n(t-\Delta t) \bar{\bf E}_n-f'({\bf P}){\bf E},\bar{\bf E}_n-{\bf E})\\
&&\ \ \ \ \leq\frac{\epsilon_0\mu_0(\epsilon_s-\epsilon_\infty)}{\tau^2}(\int_0^t(\|\int_0^s\bar{\bf E}_n(\widetilde f_n(t-\Delta t)-f'({\bf P}))\|_0^2)^{\frac{1}{2}}\cdot(\int_0^t\|\bar{\bf E}_n-{\bf E}\|_0^2)^{\frac{1}{2}}\\
&&\ \ \ \ \ \ \ +\frac{\epsilon_0\mu_0(\epsilon_s-\epsilon_\infty)}{\tau^2}(\int_0^t\|\int_0^sf'({\bf P})(\bar{\bf E}_n-{\bf E})\|_0^2)^{\frac{1}{2}}\cdot(\int_0^t\|\bar{\bf E}_n-{\bf E}\|_0^2)^{\frac{1}{2}}\leq C\Delta t^2\ ,\\
&&S_5=\epsilon_0\mu_0\int_0^t(\partial_t{\bf E}_n-\partial_t{\bf E},{\bf E}_n-\bar{\bf E}_n) \leq\epsilon_0\mu_0(\int_0^t\|\partial_t{\bf E}_n-\partial_t{\bf E}\|_0^2)^{\frac{1}{2}}\cdot(\int_0^t\|{\bf E}_n-\bar{\bf E}_n\|_0^2)^{\frac{1}{2}}\leq C\Delta t^2,\\
&&S_6=\frac{\epsilon_0\mu_0(\epsilon_s-\epsilon_\infty)}{\tau^2}\int_0^t({\bf E}_n-{\bf E},{\bf E}_n-\bar{\bf E}_n) \leq\frac{\epsilon_0\mu_0(\epsilon_s-\epsilon_\infty)}{\tau^2}(\int_0^t\|{\bf E}_n-{\bf E}\|_0^2)^{\frac{1}{2}}\cdot(\int_0^t\|{\bf E}_n-\bar{\bf E}_n\|_0^2)^{\frac{1}{2}}\leq C\Delta t^2,
\end{eqnarray*}
which completes the proof by the formula (\ref{equation:eq-40}).
\end{proof}

\begin{theorem}
\label{theorem:th4-3}
Assume that the equation (\ref{equation:eq-11}) holds, we have
\begin{eqnarray}
\label{equation:eq-41}
&&\max\limits_{t\in[0,T]}\|{\bf P}_n(t)-{\bf P}(t)\|_0\leq C\Delta t,
\end{eqnarray}
where C is positive constant.
\end{theorem}
\begin{proof}
 Subtracting $(\ref{equation:eq-11})$ from $(\ref{equation:eq-401})$, and setting ${\bf \Psi}=\bar{\bf P}_n-{\bf P}$, it yields
\begin{eqnarray*}
(\tau\partial_t{\bf P}_n+\bar f_n-(\tau\partial_t{\bf P}+f({\bf P})),\bar{\bf P}_n-{\bf P})=\epsilon_0(\epsilon_s-\epsilon_\infty)(\bar{\bf E}_n-{\bf E},\bar{\bf P}_n-{\bf P}).
\end{eqnarray*}
Integrating in time $(0,t)$, we have
\begin{eqnarray*}
&&\tau({\bf P}_n-{\bf P},{\bf P}_n-{\bf P})+\int_0^t(\bar f_n-f({\bf P}),\bar{\bf P}_n-{\bf P})\\
&=&\epsilon_0(\epsilon_s-\epsilon_\infty)\int_0^t(\bar{\bf E}_n-{\bf E},\bar{\bf P}_n-{\bf P})+\tau\int_0^t(\partial_t{\bf P}_n-\partial_t{\bf P},{\bf P}_n-\bar{\bf P}_n).
\end{eqnarray*}
Employing the monotonicity of $f$ and Lemma 3.3, we have
\begin{eqnarray*}
&\ &\tau\|{\bf P}_n-{\bf P}\|_0^2\leq\epsilon_0(\epsilon_s-\epsilon_\infty)(\int_0^t\|\bar{\bf E}_n-{\bf E}\|_0^2)^{\frac{1}{2}}(\int_0^t\|\bar{\bf P}_n-{\bf P}\|_0^2)^{\frac{1}{2}}\\
&+&\tau(\int_0^t\|\partial_t{\bf P}_n-\partial_t{\bf P}\|_0^2)^{\frac{1}{2}}\Delta t(\int_0^t\|\partial t{\bf P}_n\|_0^2)^{\frac{1}{2}}\leq C\Delta t^2\ ,
\end{eqnarray*}
which completes the proof of this step.
\end{proof}

\section{Error estimates of full discretization scheme}
Now, we focus on the investigation of the fully discrete scheme of the equations (\ref{equation:eq-10})-(\ref{equation:eq-11}). Let $\mathcal T^h$ be the standard cubic partitions of $\Omega$ with size $h$. We consider the  Raviart-Thomas-N$\acute{e}$d$\acute{e}$lec element space in three-dimension
\begin{eqnarray*}
&&N_h:=\{ {\bf \Phi}\in {\bf H}(curl,\Omega),{\bf \phi}|_K\in Q_{k-1,k,k}\times Q_{k,k-1,k}\times Q_{k,k,k-1}, \forall K\in \mathcal{T}_h  \},\\
&&W_h:=\{ {\bf \Psi}\in [L^2(\Omega)]^3,{\bf \psi}|_K\in Q_{k,k-1,k-1}\times Q_{k-1,k,k-1}\times Q_{k-1,k-1,k}, \forall K\in \mathcal{T}_h  \}.
\end{eqnarray*}

Denote $\Pi^h$  as the interpolation operator  on $N_h$  and $ \pi_h$ as the interpolation operator on $W_h$. The interpolation error estimates are given by the following lemma in \cite{MR030,MR031}.

\begin{lemma}
\label{lemma:lem4-1}
Assume that $\frac{1}{2}<\delta<1$, and $\mathcal{T}^h$ is a regular family of mesh on $\Omega $ with faces aligning with the coordinate axes. Then if ${
\bf u,v}\in [H^s(\Omega)]^d, 0<d\leq 3,\frac{1}{2}+\delta <s\leq k$, there is a constant $C>0$ independent of $h$ and ${\bf u} $ such that
\begin{eqnarray*}
&&\|{\bf u}-\Pi^h{\bf u}\|_0+\|\nabla\times({\bf u}-\Pi^h{\bf u})\|_0\leq Ch^s(\|{\bf u}\|_{H^s(\Omega)}+\|\nabla\times{\bf u}\|_{H^s(\Omega)}),\\ &&\|{\bf v}-\pi_h{\bf v}\|_0\leq Ch^s\|{\bf v}\|_{[H^s(\Omega)]^d}.
\end{eqnarray*}

\end{lemma}

Now, we can establish the full discrete finite element approximation to the equations (\ref{equation:eq-10})- (\ref{equation:eq-11}) as follows:
find $({\bf E}_i^h,{\bf P}_i^h)\in N_h\times W_h$ such that
\begin{eqnarray}
\epsilon_0\mu_0(\delta^2{\bf E}_i^h,{\bf \Phi})+\frac{\epsilon_0\mu_0(\epsilon_s-\epsilon_\infty)}{\tau}(\delta{\bf E}_i^h,{\bf \Phi})&+&(\nabla\times{\bf E}_i^h,\nabla\times{\bf \Phi}^h))-\frac{\epsilon_0\mu_0(\epsilon_s-\epsilon_\infty)}{\tau^2}(f'({\bf P}_{i-1}^h){\bf E}_i^h,{\bf \Phi}^h)\nonumber\\
\label{equation:eq-43}
&+&\frac{\mu_0}{\tau^2}(f'({\bf P}_{i-1}^h)f({\bf P}_{i-1}^h),{\bf \Phi}^h)=-\mu_0({\bf g}_i^h,{\bf \Phi}^h), \ \ \forall {\bf \Phi}^h\in N_h,\\
\label{equation:eq-44}
\tau(\delta{\bf P}_i^h,{\bf \Psi}^h)&+&(f({\bf P}_i^h),{\bf \Psi}^h)=\epsilon_0(\epsilon_s-\epsilon_\infty)({\bf E}_i^h,{\bf \Psi}^h), \ \hspace{0.85cm}\forall {\bf \Psi}^h\in W_h,\\
{\bf E}_0^h&=&\Pi^h{\bf E}_0,\ \ {\bf P}_0^h=\pi_h{\bf P}_0,\ \
\partial_t{\bf E}({\bf x},0)=\delta\Pi^h{\bf E}_0'.\nonumber
\end{eqnarray}


The existence and uniqueness of the solution of the equation (\ref{equation:eq-43}) at each time step is similar to Theorem 3.1 and the solvability of the equation (\ref{equation:eq-44}) can be follows the lemma 6.1.1 in \cite{MR029}.

\begin{lemma}
\label{lemma:lem5-1}
 Let $G:x\longrightarrow G(x) $ be a continuous mapping ${\bf R}^n$ in itself such that for a suitable $\rho>0$
\begin{eqnarray*}
(G(x),x)_{{\bf R}^n}>0,\ \ \forall x: |x|_{{\bf R}^n}=\rho.
\end{eqnarray*}
 then we can find a $x_0$ from the ball $|x|_{{\bf R}^n}\leq \rho$ such that
  $ G(x_0)=0$.
\end{lemma}

The following Lemma cited from Theorem 6.1 in [20] is to show the uniqueness for the nonlinear equation.
\begin{lemma}
\label{lemma:th5-1}
Assume that the form $\mathcal {L}$ is strong-monotone, Lipschitz-continuous, and besides bounded at zero in the second arguent
 \begin{eqnarray}
 \label{equation:eq-45}
\mathcal {L}(0, {\bf v})\leq c\|{\bf v}\|_1,\ \ \forall {\bf v}\in W_h.
\end{eqnarray}
Then the nonlinear problem $ \mathcal {L}({\bf u}^h, {\bf v})=0, \forall {\bf v}\in W_h$ has a unique solution ${\bf u}^h$ satisfying the estimate $\|{\bf u}^h\|_1\leq \frac{c}{c_0}.$
\end{lemma}

Then we have the existence and uniqueness of the equation (\ref{equation:eq-44}).
\begin{theorem}
\label{theorem:th5-2}
For any i=1,$\cdots$,n, there exists a unique solution ${\bf P}_i^h\in W_h$ to solve the fully discrete problem (\ref{equation:eq-44}).
\end{theorem}

\begin{proof}

For the nonlinear problem (\ref{equation:eq-44}), we consider its algebra system of the following scheme
\begin{eqnarray}
 \label{equation:eq-46}
(\frac{\tau}{\Delta t}{\bf P}^h,{\bf \Psi}^h)+(f({\bf P}^h),{\bf \Psi}^h)=(F_1,{\bf \Psi}^h),\forall {\bf \Psi}^h\in W_h,
\end{eqnarray}
where $F_1=\epsilon_0(\epsilon_s-\epsilon_\infty){\bf E}_i^h+\frac{\tau}{\Delta t}{\bf P}_{i-1}^h$.

Let $\varphi_i=(\varphi_{ix},\varphi_{iy},\varphi_{iz}), i=1,\cdots, n,$ be the standard basis functions of $W_h$ on $\mathcal T^h$. By representing ${\bf P}^h=\sum\limits_{i=1}^n{\bf \alpha}_i\varphi_i$ with ${\bf \alpha}_i=(\alpha_{ix},\alpha_{iy},\alpha_{iz})$, we need only to find ${\bf \gamma}=({\bf \alpha}_1,{\bf \alpha}_2,\cdots,{\bf \alpha}_n)$ for this problem. Define $\mathcal L$ from ${\bf R}^{3n}\ to\ {\bf R}^{3n}:\mathcal L({\bf \gamma})=(\mathcal L_1({\bf \gamma}),\cdots,\mathcal L_n({\bf \gamma}))$ with
\begin{eqnarray}
 \label{equation:eq-47}
\mathcal L_j({\bf \gamma})=(\frac{\tau}{\Delta t}\sum\limits_{\i=1}^n{\bf \alpha}_i\varphi_i,\varphi_j)+(f(\sum\limits_{i=1}^n{\bf \alpha}_i\varphi_i),\varphi_j)-(F_1,\varphi_j),
\end{eqnarray}
for $j=1,\cdots,n$. Then, we can reduce this problem to the nonlinear algebraic equation ${\mathcal {L}}({\bf \gamma})=0$. From (\ref{equation:eq-006}) we have
\begin{eqnarray}
{\mathcal L}({\bf \gamma}){\bf \gamma}&=&(\frac{\tau}{\Delta t}\sum\limits_{i=1}^n{\bf \alpha}_i\varphi_i,\sum\limits_{i=1}^n{\bf \alpha}_i\varphi_i)+(f(\sum\limits_{i=1}^n{\bf \alpha}_i\varphi_i ),\sum\limits_{i=1}^n{\bf \alpha}_i\varphi_i)-(F_1,\sum\limits_{i=1}^n{\bf \alpha}_i\varphi_i)\nonumber\\
&\geq&\frac{\tau_{min}}{\Delta t}\|\sum\limits_{i=1}^n{\bf \alpha}_i\varphi_i\|_0^2+\omega_f\|\sum\limits_{i=1}^n{\bf \alpha}_i\varphi_i\|_0^2-(F_1,\sum\limits_{i=1}^n{\bf \alpha}_i\varphi_i)\nonumber\\
&\geq&C(\|\sum\limits_{i=1}^n{\bf \alpha}_i\varphi_i\|_0^2-\|F_1\|_0^2)
\geq C(|{\bf \alpha}|^2-\|F_1\|_0^2)\nonumber\\
 \label{equation:eq-48}
&=&C(|{\bf \gamma}|^2-\|F_1\|_0^2).
\end{eqnarray}
Thus, ${\mathcal L}({\bf \gamma}){\bf \gamma}>0$ if $|{\bf \alpha}|^2=r^2$ provided we select $r>0$ sufficiently large. We apply Lemma \ref{lemma:lem5-1} to conclude that the equation $L({\bf \gamma})=0$ has at least one solution in the set $\{{\bf \gamma}\in {\bf R}^{2n}:|{\bf \gamma}|\leq r\}$ if $r>\|F_1\|_0$. This implies the existence of ${\bf P}_i^h\in N_h$ which solves (\ref{equation:eq-44}).

The uniqueness can be achieved by Lemma \ref{lemma:th5-1} directly. We leave this to the reader.
\end{proof}

Referring to Lemma 3.1-3.3, we  can obtain  the similar results without proof for ${\bf E}_i^h,\ {\bf P}_i^h$.
\begin{lemma}
 For j=1,$\cdots$,n, there is a positive real number C such that
\begin{eqnarray}
 \label{equation:eq-49}
&&\frac{\mu_0\epsilon_0}{2}\|{\bf E}_j^h\|_0^2+\frac{\epsilon_0\mu_0}{2}\sum\limits_{i=1}^j\|{\bf E}_i^h-{\bf E}_{i-1}^h\|_0^2+\frac{\Delta t^2}{2}\sum\limits_{i=1}^j\|\nabla\times{\bf E}_i^h\|_0^2
 +\frac{\tau}{2}\|{\bf P}_j^h\|_0^2+\frac{\tau}{2}\sum_{i=1}^j\|{\bf P}_i^h-{\bf P}_{i-1}^h\|_0^2\leq C,\\
\label{equation:eq-50}
&&\frac{\mu_0\epsilon_0}{2}\|\delta{\bf E}_j^h\|_0^2+\frac{\mu_0\epsilon_0}{2}\sum\limits_{i=1}^j\|\delta{\bf E}_i^h-\delta{\bf E}_{i-1}^h\|_0^2+\frac{1}{2}\|\nabla\times{\bf E}_j^h\|_0^2+\frac{1}{2}\sum\limits_{i=1}^j\|\nabla\times({\bf E}_i^h-{\bf E}_{i-1}^h)\|_0^2\leq C,\\
 \label{equation:eq-52}
&&\|\delta{\bf P}_j^h\|_0^2\leq C.
\end{eqnarray}
\end{lemma}

Next, the  error estimates for the full discrete schemes (\ref{equation:eq-43})-(\ref{equation:eq-44}) can be established by the following theorem. For the convenience, we denote ${\bf \Theta}_j^h={\bf E}_j^h-\Pi^h{\bf E}_j$ and ${\bf \Psi}^h_j={\bf P}_j^h-\pi_h {\bf P}_j$.
\begin{theorem}
\label{theorem:th5-3}
Assume that ${\bf E}\in H^1(0,T;{\bf H}(curl,\Omega))\bigcap H^2(0,T;[H^s(\Omega)]^d)$, ${\bf P}\in H^1(0,T;(L^\infty(\Omega))^2)\bigcap H^1(0,T;[H^s(\Omega)]^d)$.
Let $({\bf E},{\bf P})$ and $({\bf E}_j^h,{\bf P}_j^h)(1\leq j\leq n)$ be the solutions of the problem (\ref{equation:eq-10})-(\ref{equation:eq-11}) and the problem (\ref{equation:eq-43})-(\ref{equation:eq-44}) at time $t=j\Delta t$,
respectively. For  the polynomial degree $k\geq 2$, and ${\bf E}_0\in [H^s(\Omega)]^d, 0<d\leq 3,\frac{1}{2}<\delta<1, \frac{1}{2}+\delta <s\leq k$,
there holds
\begin{eqnarray}
\label{equation:eq-53}
\|{\bf \Theta}_n^h\|_0^2+\sum\limits_{j=1}^n\|\Delta t\nabla\times{\bf \Theta}_j^h\|_0^2+\|\sum\limits_{j=1}^n\Delta t\nabla\times{\bf \Theta}_j^h\|_0^2+\|{\bf \Psi}^h_n\|_0^2\leq C(\Delta t^2+h^{2s}),
\end{eqnarray}
where $C$ is a positive constant independent of the time step length $\Delta t$ and the mesh size $h$.
\end{theorem}
\begin{proof}
 First integrating (\ref{equation:eq-10}) over $[t_{i-1},t_i]$ in time  yields
\begin{eqnarray}
\epsilon_0\mu_0(\partial_t{\bf E}_i-\partial_t{\bf E}_{i-1},{\bf \Phi}^h)&+&\frac{\epsilon_0\mu_0(\epsilon_s-\epsilon_\infty)}{\tau}
({\bf E}_i-{\bf E}_{i-1},{\bf \Phi}^h)-\frac{\epsilon_0\mu_0(\epsilon_s-\epsilon_\infty)}{\tau^2}(\int_{t_{i-1}}^{t_i}f'({\bf P}){\bf E},{\bf \Phi}^h)\nonumber\\
\label{equation:eq-54}
&+&(\int_{t_{i-1}}^{t_i}\nabla\times{\bf E},\nabla\times{\bf \Phi}^h)+\frac{\mu_0}{\tau^2}(\int_{t_{i-1}}^{t_i}f'({\bf P})f({\bf P}),{\bf \Phi}^h)=-\mu_0(\int_{t_{i-1}}^{t_i}{\bf g},{\bf \Phi}^h).
\end{eqnarray}
For $i=1,\cdots,n$, setting
\begin{eqnarray*}
&&R_i^{(1)}=\partial_t{\bf E}_i-\delta{\bf E}_i=\frac{1}{\Delta t}\int_{t_{i-1}}^{t_i}(t-t_{i-1})\partial_{tt}{\bf E}dt,\\
&&R_i^{(2)}=\Delta t{\bf E}_i-\int_{t_{i-1}}^{t_i}{\bf E}dt=\int_{t_{i-1}}^{t_i}(t-t_{i-1})\partial_t{\bf E}dt,\\
&&R_i^{(3)}=\Delta t\nabla\times{\bf E}_i-\int_{t_{i-1}}^{t_i}\nabla\times{\bf E}dt=\int_{t_{i-1}}^{t_i}(t-t_{i-1})\partial_t(\nabla\times{\bf E})dt,
\end{eqnarray*}
we have
\begin{eqnarray}
\label{equation:eq-554}
&&\|R_i^{(1)}\|_0^2\leq C\Delta t\int_{t_{i-1}}^{t_i}\|\partial_{tt}{\bf E}\|_0^2dt,\ \
\|R_i^{(2)}\|_0^2\leq C\Delta t^3\int_{t_{i-1}}^{t_i}\|\partial_t{\bf E}\|_0^2dt,\ \
\|R_i^{(3)}\|_0^2\leq C\Delta t^3\int_{t_{i-1}}^{t_i}\|\partial_t(\nabla\times{\bf E})\|_0^2dt.
\end{eqnarray}
Then, there holds
\begin{eqnarray}
&&\epsilon_0\mu_0(\delta{\bf E}_i-\delta{\bf E}_{i-1},{\bf \Phi}^h)+\frac{\epsilon_0\mu_0(\epsilon_s-\epsilon_\infty)}{\tau}({\bf E}_i-{\bf E}_{i-1},{\bf \Phi}^h)-\frac{\epsilon_0\mu_0(\epsilon_s-\epsilon_\infty)}{\tau^2}B(\Delta t{\bf E}_i,{\bf \Phi}^h)\nonumber\\
&&\ \ \ \ +(\Delta t\nabla\times{\bf E}_i,\nabla\times{\bf \Phi}^h)+\frac{\mu_0}{\tau^2}(\int_{t_{i-1}}^{t_i}f'({\bf P})f({\bf P}),{\bf \Phi}^h)\leq\epsilon_0\mu_0(-R_i^{(1)}+R_{i-1}^{(1)},{\bf \Phi}^h)\nonumber\\
&&\ \ \ \ \ \ \ \ \ \ \ \ \ \ \ \ \ +\frac{\epsilon_0\mu_0(\epsilon_s-\epsilon_\infty)}{\tau^2}B(-R_i^{(2)},{\bf \Phi}^h)+(R_i^{(3)},{\bf \Phi}^h)-\mu_0
(\int_{t_{i-1}}^{t_i}{\bf g},{\bf \Phi}^h).
\end{eqnarray}
Making summation for $i=1,\cdots,j$, leads to
\begin{eqnarray}
\epsilon_0\mu_0(\delta{\bf E}_j,{\bf \Phi}^h)&+&\frac{\epsilon_0\mu_0(\epsilon_s-\epsilon_\infty)}{\tau}({\bf E}_j,{\bf \Phi}^h)-\frac{\epsilon_0\mu_0(\epsilon_s-\epsilon_\infty)}{\tau^2}B(\sum\limits_{i=1}^j\Delta t{\bf E}_i,{\bf \Phi}^h)\nonumber\\
&+&(\sum\limits_{i=1}^j\Delta t\nabla\times{\bf E}_i,\nabla\times{\bf \Phi}^h)+\frac{\mu_0}{\tau^2}(\sum\limits_{i=1}^j\int_{t_{i-1}}^{t_i}f'({\bf P})f({\bf P}),{\bf \Phi}^h)\nonumber\\
&\leq&\epsilon_0\mu_0(-R_j^{(1)},{\bf \Phi}^h)+\frac{\epsilon_0\mu_0(\epsilon_s-\epsilon_\infty)}{\tau^2}B
(\sum\limits_{i=1}^j-R_i^{(2)},{\bf \Phi}^h)+(\sum\limits_{i=1}^jR_i^{(3)},{\bf \Phi}^h)\nonumber\\
\label{equation:eq-56}
&-&\mu_0(\sum\limits_{i=1}^j\int_{t_{i-1}}^{t_i}{\bf g},{\bf \Phi}^h)+\frac{\epsilon_0\mu_0(\epsilon_s-\epsilon_\infty)}{\tau}({\bf E}_0,{\bf \Phi}^h).
\end{eqnarray}
Next, integrating (\ref{equation:eq-43}) over $[t_{i-1},t_i]$ in time
and summing up for $i=1,\cdots,j$, we have
\begin{eqnarray}
\epsilon_0\mu_0(\delta{\bf E}_j^h,{\bf \Phi}^h)&+&\frac{\epsilon_0\mu_0(\epsilon_s-\epsilon_\infty)}{\tau}({\bf E}_j^h,{\bf \Phi}^h)-\frac{\epsilon_0\mu_0(\epsilon_s-\epsilon_\infty)}{\tau^2}B(\sum\limits_{i=1}^j\Delta t{\bf E}_i^h,{\bf \Phi}^h)\nonumber\\
&+&(\sum\limits_{i=1}^j\Delta t\nabla\times{\bf E}_i^h,\nabla\times{\bf \Phi}^h)+\frac{\mu_0}{\tau^2}(\sum\limits_{i=1}^j\int_{t_{i-1}}^{t_i}f'({\bf P}_i^h)f({\bf P}_i^h),{\bf \Phi}^h)\nonumber\\
\label{equation:eq-57}
&\leq&-\mu_0(\sum\limits_{i=1}^j\int_{t_{i-1}}^{t_i}{\bf g}_i,{\bf \Phi}^h)+\frac{\epsilon_0\mu_0(\epsilon_s-\epsilon_\infty)}{\tau}({\bf E}_0^h,{\bf \Phi}^h).
\end{eqnarray}

Then subtracting (\ref{equation:eq-56}) from (\ref{equation:eq-57}) and
multiplying both sides  by $\Delta t$,  we have
\begin{eqnarray}
\epsilon_0\mu_0\Delta t(\delta{\bf \Theta}_j^h,{\bf \Theta}_j^h)&+&\frac{\epsilon_0\mu_0(\epsilon_s-\epsilon_\infty)}{\tau} \Delta t\|{\bf \Theta}_j^h\|_0^2+(\sum\limits_{i=1}^j\Delta t\nabla\times{\bf \Theta}_j^h,\Delta t\nabla\times{\bf \Theta}_j^h)\nonumber\\
& \leq&-\mu_0\Delta t(\sum\limits_{i=1}^j\int_{t_{i-1}}^{t_i}({\bf g}_i-{\bf g}),{\bf \Theta}_j^h)+\epsilon_0\mu_0\Delta t(R_j^{(1)},{\bf \Theta}_j^h)\nonumber\\
&+&\frac{\epsilon_0\mu_0(\epsilon_s-\epsilon_\infty)}{\tau^2}B\Delta t(\sum\limits_{i=1}^jR_i^{(2)},{\bf \Theta}_j^h)+\Delta t(-\sum\limits_{i=1}^jR_i^{(3)},{\bf \Theta}_j^h)+\Delta t(\delta({\bf E}_j-\Pi^h{\bf E}_j),{\bf \Theta}_j^h)\nonumber\\
& +&\Delta t({\bf E}_j-\Pi^h{\bf E}_j,{\bf \Theta}_j^h)+(\Delta t\sum\limits_{i=1}^j\nabla\times({\bf E}_i-\Pi^h{\bf E}_i),\Delta t\nabla\times{\bf \Theta}_j^h)\nonumber\\
& +&\frac{\epsilon_0\mu_0(\epsilon_s-\epsilon_\infty)}{\tau^2}B(\sum\limits_{i=1}^j\Delta t(\Pi^h{\bf E}_i-{\bf E}_i),\Delta t{\bf \Theta}_j^h)+\epsilon_0\mu_0\Delta t(\Pi^h{\bf E}_0^{'}-{\bf E}_0^{'},{\bf \Theta}_j^h)\nonumber\\
& +&\frac{\epsilon_0\mu_0(\epsilon_s-\epsilon_\infty)}{\tau}\Delta t(\Pi^h{\bf E}_0-{\bf E}_0,{\bf \Theta}_j^h)+\frac{\epsilon_0\mu_0(\epsilon_s-\epsilon_\infty)}{\tau^2}B (\Delta t\sum\limits_{i=1}^j{\bf \Theta}_i^h,\Delta t{\bf \Theta}_j^h)\nonumber\\
\label{equation:eq-58}
& -&\frac{\mu_0}{\tau^2}\Delta t(\sum\limits_{i=1}^j\int_{t_{i-1}}^{t_i}(f'({\bf P}_{i-1}^h)f({\bf P}_{i-1}^h)-f'({\bf P})f({\bf P})),{\bf \Theta}_j^h)=:\sum\limits_{i=1}^{12}T_i.
\end{eqnarray}

Now, summing up $j=1,\cdots,n,$ we deal with each term on both sides of the equation (\ref{equation:eq-58}).
Applying Abel's summation rule, we have
\begin{eqnarray*}
&&\epsilon_0\mu_0\Delta t\sum\limits_{j=1}^n(\delta{\bf \Theta}_j^h,{\bf \Theta}_j^h)=\frac{1}{2}\epsilon_0\mu_0\| {\bf \Theta}_n^h\|_0^2,\\
&&\sum\limits_{j=1}^n(\sum\limits_{i=1}^j\Delta t\nabla\times{\bf \Theta}_j^h,\Delta t\nabla\times{\bf \Theta}_j^h)
=\frac{1}{2}\Delta t^2(\sum\limits_{j=1}^n\|\nabla\times{\bf \Theta}_j^h\|_0^2
+\frac{1}{2}\|\sum\limits_{j=1}^n\nabla\times{\bf \Theta}_j^h\|_0^2).
\end{eqnarray*}
Using Yong's inequality and the estimates (\ref{equation:eq-554}), we have the following results:
\begin{eqnarray*}
T_1&=&\mu_0\sum\limits_{j=1}^n\Delta t(\sum\limits_{i=1}^j\int_{t_{i-1}}^{t_i}({\bf g}_i-{\bf g}),{\bf \Theta}_j^h)\\
&\leq&\mu_0\sum\limits_{j=1}^n\Delta t\|\sum\limits_{i=1}^j\int_{t_{i-1}}^{t_i}({\bf g}_i-{\bf g})\|_0\|{\bf \Theta}_j^h\|_0
\leq\mu_0\Delta t^2\frac{1}{4\epsilon_1}\|{\bf g}\|_{H^1(0,T;L^2(\Omega)^2)}^2+\mu_0\epsilon_1\sum\limits_{j=1}^n \|\Delta t{\bf \Theta}_j^h\|_0^2,\\
T_2&=&\epsilon_0\mu_0\Delta t\sum\limits_{j=1}^n(R_j^{(1)},{\bf \Theta}_j^h)\leq\epsilon_0\mu_0\frac{1}{4\epsilon_2}\Delta t^2\|{\bf E}\|_{H^2(0,T;L^2(\Omega)^2)}^2dt+\epsilon_0\mu_0\epsilon_2\Delta t\sum\limits_{j=1}^n\|{\bf \Theta}_j^h\|_0^2\ ,\\
T_3&=&\frac{\epsilon_0\mu_0(\epsilon_s-\epsilon_\infty)B\Delta t}{\tau^2}\sum\limits_{j=1}^n(\sum\limits_{i=1}^jR_i^{(2)},{\bf \Theta}_j^h)\\
&\leq&\frac{\epsilon_0\mu_0(\epsilon_s-\epsilon_\infty)B}{4\tau^2\epsilon_3}  \sum\limits_{j=1}^n \sum\limits_{i=1}^j \|R_i^{(2)}\|_0^2
+\frac{\epsilon_0\mu_0(\epsilon_s-\epsilon_\infty)B\epsilon_3}{\tau^2}\sum\limits_{j=1}^n \|\Delta t{\bf \Theta}_j^h\|_0^2\ ,\\
& \leq &C\frac{\epsilon_0\mu_0(\epsilon_s-\epsilon_\infty)B\Delta t^2}{4\epsilon_3\tau^2}\|{\bf E}\|_{H^1(0,T;L^2(\Omega)^2)}^2dt+\frac{\epsilon_0\mu_0(\epsilon_s-\epsilon_\infty)B\Delta t\epsilon_3}{\tau^2} \sum\limits_{j=1}^n\|\Delta t{\bf \Theta}_j^h\|_0^2\ ,\\
T_4&=&\sum\limits_{j=1}^n\Delta t(\sum\limits_{i=1}^jR_i^{(3)},{\bf \Theta}_j^h)\leq\sum\limits_{j=1}^n\sum\limits_{i=1}^j\|R_i^{(3)}\|_0\|\Delta t{\bf \Theta}_j^h\|_0\\
&\leq& C\frac{\Delta t^2}{4\epsilon_4}\|{\bf E}\|_{H^1(0,T;H(curl,\Omega))}^2+\epsilon_4\sum\limits_{j=1}^n\|\Delta t{\bf \Theta}_j^h\|_0^2\ ,\\
T_5&=&\sum\limits_{j=1}^n\Delta t(\delta({\bf E}_j-\Pi^h{\bf E}_j),{\bf \Theta}_j^h)\leq \sum\limits_{j=1}^n\frac{C\Delta t h^{2s}}{4\epsilon_5}\|\delta({\bf E}_j)\|_{[H^s(\Omega)]^d}^2+\epsilon_5\sum\limits_{j=1}^n\Delta t\|{\bf \Theta}_j^h\|_0^2\\
&\leq& \frac{Ch^{2s}}{4\epsilon_5}\|\partial_t{\bf E}\|^2_{L^\infty(0,T;[H^{s}(\Omega)]^d)}+\epsilon_5\sum\limits_{j=1}^n\Delta t\|{\bf \Theta}_j^h\|_0^2\\
&\leq& \frac{C h^{2s}}{4\epsilon_5}\|{\bf E}\|_{H^2(0,T;[H^s(\Omega)]^d)}^2+\epsilon_5\sum\limits_{j=1}^n\Delta t\|{\bf \Theta}_j^h\|_0^2\ ,\\
T_6&=&\sum\limits_{j=1}^n\Delta t({\bf E}_j-\Pi^h{\bf E}_j,{\bf \Theta}_j^h)\leq \frac{Ch^{2s}}{4\epsilon_6}\sum\limits_{j=1}^n\Delta t\|{\bf E}_j\|_{[H^s(\Omega)]^d}+\epsilon_6\sum\limits_{j=1}^n\Delta t\|{\bf \Theta}_j^h\|_0\\
&\leq& \frac{Ch^{2s}}{4\epsilon_6}\|{\bf E}\|_{L^\infty(0,T;[H^s(\Omega)]^d)}^2+\epsilon_6\Delta t\|{\bf \Theta}_j^h\|_0^2\\
&\leq &\frac{Ch^{2s}}{4\epsilon_6}\|{\bf E}\|_{H^1(0,T;[H^s(\Omega)]^d)}^2+\epsilon_6\Delta t\|{\bf \Theta}_j^h\|_0^2,\\
T_{7}&=&\sum\limits_{j=1}^n(\Delta t\sum\limits_{i=1}^j\nabla\times({\bf E}_i-\Pi^h{\bf E}_i),\Delta t\nabla\times{\bf \Theta}_j^h)\\
& \leq&\frac{C}{4\epsilon_{7}}\sum\limits_{j=1}^n\|\Delta t\sum\limits_{i=1}^j\nabla\times({\bf E}_i-\Pi^h{\bf E}_i)\|_0^2+\epsilon_{7}\sum\limits_{j=1}^n\|\Delta t\nabla\times{\bf \Theta}_j^h\|_0^2\\
&\leq&\frac{C}{4\epsilon_{7}}\sum\limits_{j=1}^n\sum\limits_{i=1}^j\Delta t^2\|{\bf E}_i-\Pi^h{\bf E}_i\|_0^2+{\epsilon_{7}}\sum\limits_{j=1}^n\|\Delta t\nabla\times{\bf \Theta}_j^h\|_0^2\\
&\leq&\frac{Ch^{2s}}{4\epsilon_{7}}\|{\bf E}\|_{L^\infty(0,T;[H^s(\Omega)]^d)}+{\epsilon_{7}}\sum\limits_{j=1}^n \|\Delta t\nabla\times{\bf \Theta}_j^h\|_0^2\\
& \leq&\frac{Ch^{2s}}{4\epsilon_{7}}\|{\bf E}\|_{H^1(0,T;[H^s(\Omega)]^d)}+{\epsilon_{7}}\sum\limits_{j=1}^n\|\Delta t\nabla\times{\bf \Theta}_j^h\|_0^2\ ,\\
T_{8}&=&\frac{\epsilon_0\mu_0(\epsilon_s-\epsilon_\infty)B}{\tau^2}\sum\limits_{j=1}^n
(\sum\limits_{i=1}^j\Delta t(\Pi^h{\bf E}_i-{\bf E}_i),\Delta t{\bf \Theta}_j^h)\\
&\leq&\frac{\epsilon_0\mu_0(\epsilon_s-\epsilon_\infty)}{\tau^2} \frac{C Bh^{2s}}{4\epsilon_{8}}\|{\bf E}\|_{H^1(0,T;[H^s(\Omega)]^d)}+\frac{\epsilon_0\mu_0(\epsilon_s-\epsilon_\infty)}{\tau^2} B\epsilon_{8}\sum\limits_{j=1}^n\|\Delta t{\bf \Theta}_j^h\|_0^2\ ,\\
T_{9}&=&\epsilon_0\mu_0\sum\limits_{j=1}^n\Delta t(\Pi^h{\bf E}_0^{'}-{\bf E}_0^{'},{\bf \Theta}_j^h)\leq\frac{C\epsilon_0\mu_0h^{2s}}{4\epsilon_{9}}\|{\bf E}_0^{'}\|_{[H^s(\Omega)]^d}^2+\epsilon_0\mu_0\epsilon_{9}\sum\limits_{j=1}^n
\Delta t\|{\bf \Theta}_j^h\|_0^2\ ,\\
T_{10}&=&\frac{\epsilon_0\mu_0(\epsilon_s-\epsilon_\infty)}{\tau}\sum\limits_{j=1}^n\Delta t(\Pi^h{\bf E}_0-{\bf E}_0,{\bf \Theta}_j^h)\\
&\leq&\frac{\epsilon_0\mu_0(\epsilon_s-\epsilon_\infty)}{\tau}
\frac{Ch^{2s}}{4\epsilon_{10}}\|{\bf E}_0\|_{[H^s(\Omega)]^d}^2+\frac{\epsilon_0\mu_0(\epsilon_s-\epsilon_\infty)}{\tau}
\epsilon_{10}\sum\limits_{j=1}^n\Delta t\| {\bf \Theta}_j^h\|_0^2\ ,\\
T_{11}&=&\frac{\epsilon_0\mu_0(\epsilon_s-\epsilon_\infty)B}{\tau^2} \sum\limits_{j=1}^n(\Delta t\sum\limits_{i=1}^j{\bf \Theta}_i^h,\Delta t{\bf \Theta}_j^h)\\
&\leq&\frac{\epsilon_0\mu_0(\epsilon_s-\epsilon_\infty)B}{2\tau^2}
\sum\limits_{j=1}^n\|\Delta t{\bf \Theta}_j^h\|_0^2+\frac{\epsilon_0\mu_0(\epsilon_s-\epsilon_\infty)B}{2\tau^2}\| \sum\limits_{j=1}^n\Delta t{\bf \Theta}_j^h\|_0^2.
\end{eqnarray*}
For the last term, we have to use the properties  of the function $f$. Assume $f$ and $f'$ are  local Lipschitz functions and $f, f'\in W^{1,\infty}(\Omega)$ , respectively.
For ${\bf \Psi}^h_i={\bf P}_i^h-\pi_h {\bf P}_i$, we have
\begin{eqnarray*}
&&(f'({\bf P}_{i-1}^h)f({\bf P}_{i-1}^h)-f'({\bf P})f({\bf P})),{\bf \Theta}_j^h)\\
&=&((f'({\bf P}_{i-1}^h)-f'({\bf P}))f({\bf P}_{i-1}^h),{\bf \Theta}_j^h)
+((f({\bf P}_{i-1}^h)-f({\bf P}))f'({\bf P}),{\bf \Theta}_j^h)\\
&=&((f'({\bf P}_{i-1}^h)-f'(\pi_h{\bf P}_{i-1})+f'(\pi_h{\bf P}_{i-1})-f'({\bf P}_{i-1})+f'({\bf P}_{i-1})-f'({\bf P}))f({\bf P}_{i-1}^h),{\bf \Theta}_j^h)\\
&+&((f({\bf P}_{i-1}^h)-f(\pi_h{\bf P}_{i-1})+f(\pi_h{\bf P}_{i-1})-f({\bf P}_{i-1})+f({\bf P}_{i-1})-f({\bf P}))f'({\bf P}),{\bf \Theta}_j^h)\\
&\leq & (C\|{\bf \Psi}^h_i\|_0+Ch^s\|{\bf P}_{i-1}\|_{[H^s(\Omega)]^d}+C\Delta t \| \partial_t {P}\|_0)\|{\bf \Theta}_j^h\|_0.
\end{eqnarray*}
Hence, we have
\begin{eqnarray*}
T_{12}&=&\frac{\mu_0}{\tau^2}\sum\limits_{j=1}^n\Delta t(\sum\limits_{i=1}^j\int_{t_{i-1}}^{t_i}(f'({\bf P}_{i-1}^h)-f'({\bf P}))f({\bf P}_{i-1}^h),{\bf \Theta}_j^h)
 +\frac{\mu_0}{\tau^2}\sum\limits_{j=1}^n\Delta t(\sum\limits_{i=1}^j\int_{t_{i-1}}^{t_i}(f({\bf P}_{i-1}^h)-f({\bf P}))f'({\bf P}),{\bf \Theta}_j^h)\\
&\leq&\frac{\mu_0\Delta t^2}{4\epsilon_{11}\tau^2}\sum\limits_{j=1}^n
(\|{\bf \Psi}^h_j\|_0^2+Ch^{2s}\|{\bf P}\|_{L^\infty(0,T;[H^s(\Omega)]^d)}^2+C\Delta t^2 \| {\bf P}\|_{H^1(0,T;L^\infty(\Omega)}^2)+\epsilon_{11}\sum\limits_{j=1}^n\Delta t\|{\bf \Theta}_j^h\|_0^2.
\end{eqnarray*}

Next, subtracting (\ref{equation:eq-13}) from (\ref{equation:eq-44}),  multiplying both sides by $\Delta t$, we have
\begin{eqnarray*}
\tau({\bf P}_i^h-{\bf P}_i,{\bf \Psi}^h)-\tau({\bf P}_{i-1}^h-{\bf P}_{i-1},{\bf \Psi}^h)+\Delta t(f({\bf P}_i^h)-f({\bf P}_i),{\bf \Psi})=\epsilon_0(\epsilon_s-\epsilon_\infty)\Delta t({\bf E}_i^h-{\bf E}_i,{\bf \Psi}^h).
\end{eqnarray*}
Replaced  $\Psi^h$ by ${\bf \Psi}^h_i$, we have
\begin{eqnarray*}
&&\tau({\bf P}_i^h-{\bf P}_i,{\bf \Psi}^h_i)-\tau({\bf P}_{i-1}^h-{\bf P}_{i-1},{\bf \Psi}^h_i)+\Delta t(f({\bf P}_i^h)-f(\pi_h{\bf P}_i),{\bf \Psi}^h_i)\\
&&\ \ \ \ =\epsilon_0(\epsilon_s-\epsilon_\infty)\Delta t({\bf E}_i^h-{\bf E}_i,{\bf \Psi}^h_i)+\Delta t(f({\bf P}_i)-f(\pi_h{\bf P}_i),{\bf \Psi}^h_i).
\end{eqnarray*}
From (\ref{equation:eq-006}),
$(f({\bf P}_i^h)-f(\pi_h{\bf P}_i),{\bf \Psi}^h_i)\geq 0,$
we have
\begin{eqnarray}
\tau({\bf \Psi}^h_i-{\bf \Psi}_{i-1}^h,{\bf \Psi}^h_i)&\leq & \epsilon_0(\epsilon_s-\epsilon_\infty) \Delta t({\bf E}_i^h-{\bf E}_i,{\bf \Psi}^h_i)+\Delta t(f({\bf P}_i)-f(\pi_h{\bf P}_i),{\bf \Psi}^h_i)\nonumber\\
\label{equation:eq-62}
&+&\tau({\bf P}_i-\pi_h{\bf P}_i,{\bf \Psi}^h_i)+\tau(\pi_h{\bf P}_{i-1}-{\bf P}_{i-1},{\bf \Psi}^h_i).
\end{eqnarray}
Then, we sum (\ref{equation:eq-62}) up for $i=1,\cdots,n$ to obtain
\begin{eqnarray}
&&\sum\limits_{i=1}^n\tau({\bf \Psi}^h_i-{\bf \Psi}^h_{i-1},{\bf \Psi}^h_i)=\epsilon_0(\epsilon_s-\epsilon_\infty)\sum\limits_{i=1}^n \Delta t({\bf E}_i^h-{\bf E}_i,{\bf \Psi}^h_i)+\sum\limits_{i=1}^n\Delta t(f({\bf P}_i)-f(\pi_h{\bf P}_i),{\bf \Psi}^h_i)\nonumber\\
\label{equation:eq-63}
&&\ \ \ \ \ \ \ \ \ \ \ \ \ \ \ \ \ \ \ \ \ \ \ \ \ \ \ \ \ \ \ \ +\tau({\bf P}_n-\pi_h{\bf P}_n+\pi_h{\bf P}_0-{\bf P}_0,\sum\limits_{i=1}^n{\bf \Psi}^h_i).
\end{eqnarray}

The rest of the work is focus on the  error estimates on the right hand side of (\ref{equation:eq-63}), which concludes the nonlinear error estimates. We note a $L^\infty$ bound for the exact solution and its interpolation
\begin{eqnarray}
\label{equation:eq-apri-1}
 \|{\bf P}_i\|_{L^\infty(\Omega)}\leq C^*, \ \ \ \| \pi_h {\bf P}_i\|_{L^\infty(\Omega)}\leq C^*,
 \end{eqnarray}
and
\begin{eqnarray}
\label{equation:eq-apri-2}
 \|{\bf P}_i-\pi_h {\bf  P}_i\|_{L^\infty(\Omega)}\leq Ch^{s+1}|ln h|.
 \end{eqnarray}
{\bf An a-priori $L^\infty$ assumption up to time step $t_i, i\leq n-1.$ }
We also assume
a-priori that the numerical error function for ${\bf P}$ has a  $L^\infty$ bound at time steps $t_i, i\leq n-1,$
\begin{eqnarray}
\label{equation:eq-apri-3}
\|{\bf \Psi}_i^h\|_{L^\infty(\Omega)}\leq 1, \ \ i=1,2,\cdots, n-1,
\end{eqnarray}
so that a $L^\infty$  bound for the numerical solution $ {\bf P}_i^h$
is available
\begin{eqnarray}
\label{equation:eq-apri-4}
\|{\bf P}_i^h\|_{L^\infty(\Omega)}=\|\pi_h{\bf P}_i^h-{\bf \Psi}_i^h\|_{L^\infty(\Omega)} \leq \|\pi_h{\bf P}_i^h\|_{L^\infty(\Omega)}
+\|{\bf \Psi}_i^h\|_{L^\infty(\Omega)}\leq C^*+1.
\end{eqnarray}
This assumption will be recovered in later analysis.

Now, we deal with each term on both sides of the equation (\ref{equation:eq-63}). Using the inequality $a(a-b)\geq a^2/2-b^2/2$ yields
\begin{eqnarray}
\label{equation:eq-888}
\sum\limits_{i=1}^n\tau({\bf \Psi}^h_i-{\bf \Psi}^h_{i-1},{\bf \Psi}^h_i)\geq\frac{\tau}{2} \|{\bf \Psi}^h_n\|_0^2-\frac{\tau}{2}\|{\bf \Psi}^h_0\|_0^2=\frac{\tau}{2}\|{\bf \Psi}^h_n\|_0^2.
\end{eqnarray}
Using Cauchy's inequality and Young's inequality, and applying the result of Lemma \ref{lemma:lem4-1}, we have the following estimates
\begin{eqnarray}
\epsilon_0(\epsilon_s-\epsilon_\infty)\sum\limits_{i=1}^n \Delta t({\bf E}_i^h-{\bf E}_i,{\bf \Psi}^h_i)
&=&\epsilon_0(\epsilon_s-\epsilon_\infty)\sum\limits_{i=1}^n\Delta t({\bf \Theta}_i^h,{\bf \Psi}_i^h)+
\epsilon_0(\epsilon_s-\epsilon_\infty)\sum\limits_{i=1}^n\Delta t(\Pi^h{\bf E}_i-{\bf E}_i,{\bf \Psi}_i^h)\nonumber\\
&\leq&\frac{\epsilon_0(\epsilon_s-\epsilon_\infty)}{4\epsilon_{12}}
\sum\limits_{i=1}^n\Delta t\|{\bf \Theta}_i^h\|_0^2+
\epsilon_{12}\epsilon_0(\epsilon_s-\epsilon_\infty)\sum\limits_{i=1}^n\Delta t\|{\bf \Psi}_i^h\|_0^2\nonumber\\
\label{equation:eq-880}
&+&\epsilon_{13}\epsilon_0(\epsilon_s-\epsilon_\infty)
\sum\limits_{i=1}^n\Delta t\|{\bf \Psi}_i^h\|_0^2
+\frac{\epsilon_0(\epsilon_s-\epsilon_\infty)}{4\epsilon_{13}}
h^{2s}\|{\bf E}\|_{[H^s(curl,\Omega)]^d},\\
\sum\limits_{i=1}^n\Delta t(f({\bf P}_i)-f(\Pi^h{\bf P}_i),{\bf \Theta}_i^h)
&\leq&\sum\limits_{i=1}^n\Delta t\|f({\bf P}_i)
-f(\Pi^h{\bf P}_i)\|_0\|{\bf \Psi}^h_i\|_0\nonumber\\
&\leq&\frac{1}{4\epsilon_{14}}\sum\limits_{i=1}^nB^2\Delta t\|{\bf P}_i-\Pi^h{\bf P}_i\|_0^2+\epsilon_{14}\sum\limits_{i=1}^n\Delta t\|{\bf \Psi}^h_i\|_0^2\nonumber \\
\label{equation:eq-881}
&\leq& \frac{Ch^{2s}}{4\epsilon_{14}}\|{\bf P}\|_{H^1(0,T;[H^s(\Omega)]^d)}^2+\epsilon_{14}\sum\limits_{i=1}^n\Delta t\|{\bf \Psi}^h_i\|_0^2\ ,\\
\tau({\bf P}_n-\Pi^h{\bf P}_n+\Pi^h{\bf P}_0-{\bf P}_0,{\bf \Psi}^h_i)
&\leq& C\tau\|{\bf P}_n-\Pi^h{\bf P}_n\|_0\|{\bf \Psi}^h_i\|_0+C\tau\|\Pi^h{\bf P}_0-{\bf P}_0\|_0\|{\bf \Psi}^h_i\|_0\nonumber\\
&\leq& C\epsilon_{15}\tau h^{2s}\|{\bf P}_n\|_{[H^s(\Omega)]^d}^2+C\tau\frac{1}{4\epsilon_{15}}\sum\limits_{i=1}^n\|{\bf \Psi}^h_i\|_0^2 \nonumber\\
\label{equation:eq-882}
&+& \epsilon_{16}C\tau h^{2s}\|{\bf P}_0\|_{[H^s(\Omega)]^d}^2+C\tau\frac{1}{4\epsilon_{16}}
\sum\limits_{i=1}^n\|{\bf \Psi}^h_i\|_0^2.
\end{eqnarray}
Thus, by selecting  suitable $\epsilon_{i}, i=1,\cdots,16$, adding $T_i,i=1,\cdots,12$ to the estimates (\ref{equation:eq-880})-(\ref{equation:eq-882}), and applying Gr\"{o}nwall's inequality, we obtain
\begin{eqnarray}
\label{equation:eq-69}
\|{\bf \Theta}_n^h\|_0^2+\sum\limits_{j=1}^n\|\Delta t\nabla\times{\bf \Theta}_j^h\|_0^2+\|\sum\limits_{j=1}^n\Delta t\nabla\times{\bf \Theta}_j^h\|_0^2+\|{\bf \Psi}^h_n\|_0^2\leq C(\Delta t^2+h^{2s}).
\end{eqnarray}
The above constant $C$ is independent of time step $\Delta t$ and mesh size $ h$.

{\bf Recovery of the a-priori bound (\ref{equation:eq-apri-3}).}
With the help of the $L_2$ error estimate
(\ref{equation:eq-69}) and an application of inverse inequality, the following inequality is available,
for $0<d \leq  3, s\geq 2$,
\begin{eqnarray}
\label{equation:eq-70}
\|{\bf \Psi}^h_i\|_{L^\infty(\Omega)}\leq \frac{\|{\bf \Psi}^h_i\|_{L^2(\Omega)}}{h^{\frac{d}{2}}}\leq \frac{C(\Delta t+h^{s})}{h^{\frac{d}{2}}}\leq C_0^*,
\end{eqnarray}
under a  requirement $\Delta t=O(h^{\frac{d}{2}})$.

\end{proof}

\begin{corollary}
Under the assumptions of Theorem \ref{theorem:th5-3},
there holds
\begin{eqnarray}
\label{equation:eq-999}
\max\limits_{1\leq i\leq n}\|{\bf E}_i-{\bf E}_i^h\|_{H(curl,\Omega)}^2+\max\limits_{1\leq i\leq n}\|{\bf P}_i-{\bf P}_i^h\|_0^2\leq C(\Delta t^2+h^{2s}).
\end{eqnarray}
\end{corollary}

\section{The super-convergence of the lowest Raviart-Thomas-N$\acute{e}$d$\acute{e}$lec element: k=1}

From the above section, we observe that the convergence order estimate  in (\ref{equation:eq-69}) has played a crucial role to recover the a-priori bound (\ref{equation:eq-70}). In more details, its spatial accuracy has to be stronger than $O(h^{\frac{d}{2}})$, that is, the estimate (\ref{equation:eq-70}) holds only when $k\geq2$ for $d=3$ and $k=1$ for $d\leq2$.  In order to improve the convergence order for the lowest Raviart-Thomas-N$\acute{e}$d$\acute{e}$lec element $k=1,d=3$, we can employ a super-convergence technique on a uniform mesh, seeing \cite{MR032} for the related theoretical tools. Now we consider the lowest Raviart-Thomas-N$\acute{e}$d$\acute{e}$lec element space in three dimension
\begin{eqnarray*}
&&N_h=\{ {\bf \Phi}\in {\bf H}(curl,\Omega),\phi|_K\in Q_{0,1,1}\times Q_{1,0,1} \times Q_{1,1,0}, \forall K\in \mathcal{T}_h  \},\\
&&W_h=\{ {\bf \Psi}\in [L^2(\Omega)]^3,\psi|_K\in Q_{1,0,0}\times Q_{0,1,0}\times Q_{0,0,1} , \forall K\in \mathcal{T}_h  \}.
\end{eqnarray*}

The following results are needed in the later analysis and the detailed proofs can be found in \cite{MR63}.

\begin{lemma} \label{lem:super convergence}
For any ${\bf \Phi}^h\in N_h,{\bf \Psi}^h\in W_h$, denote $\Pi^h$ and $\pi_h$ as the interpolation operator on $N_h$ and $W_h$, respectively, we have
\begin{eqnarray}
\label{equation:eq-a66666}
&&(\pi_h{\bf P}-{\bf P},{\bf \Psi}^h)=0,\ \ \forall\ {\bf \Psi}^h\in W_h,\\
\label{equation:eq-a63}
&&({\bf E}-\Pi^h {\bf E},{\bf \Phi}^h)=O(h^2)\|{\bf E}\|_2  \|{\bf \Phi}^h\|_0,\\
\label{equation:eq-a64}
&&(\nabla\times({\bf E}-\Pi^h {\bf E}),{\bf \Psi}^h)=O(h^2)\|{\bf E}\|_2  \|{\bf \Psi}^h\|_0.
\end{eqnarray}
There exists the post-processing operators $\Pi_{2h}^1 {\bf w} \in Q_{1,1,1}(\overline{K}),\pi_{2h}^1 {\bf v} \in Q_{1,1,1}(\overline{K})$ \cite{MR032,MR63}, such that
\begin{eqnarray}
\label{equation:eq-a65}
&&(i)~\|\Pi_{2h}^1 {\bf w}-{\bf w}\|_0\leq Ch^2\|{\bf w}\|_0,~~~~\|\pi_{2h}^1 {\bf v}-{\bf v}\|_0\leq Ch^2\|{\bf v}\|_2, \qquad \ \ \forall~{\bf w},{\bf v} \in [ H^2(\Omega)]^3,\\
\label{equation:eq-a66}
&&(ii)~\|\Pi_{2h}^1 {\bf w}\|_0\leq C\|{\bf w}\|_0,~~~~~~~~~~~~~\|\pi_{2h}^1 {\bf v}\|_0\leq C\|{\bf v}\|_0,~~~~~~~~~~~~~~~~~~~~~\forall~{\bf w} \in N_h,{\bf v} \in W_h,\\
\label{equation:eq-a67}
&&(iii)~\Pi_{2h}^1 {\bf w}=\Pi_{2h}^1\Pi^h {\bf w},~~~~~~~~~~~~~\pi_{2h}^1 {\bf v}=\pi_{2h}^1\pi_h {\bf v},~~~~~~~~~~~~~~~~~~~~~~~~~\forall~{\bf w} \in N_h,{\bf v} \in W_h ,
\end{eqnarray}
for the adjoint element $\overline{K}=\bigcup K_i,i=1,2,3,4$.
\end{lemma}

Using these post-processing operators, we can achieve the following global super-convergence for all three dispersive media.

\begin{theorem} \label{thm:super convergence}
Assume the partition $\mathcal{T}^h$ of $\Omega$ is uniform \cite{MR032}, $\Pi^h$ and $\pi_h$ are the interpolation on $N_h$ and $W_h$, respectively. If \ ${\bf E }\in H^2(0,T;[ H^2(\Omega)]^3)$, ${\bf H} \in H^1(0,T;[ H^2(\Omega)]^3)$, for the lowest Raviart-Thomas-N$\acute{e}$d$\acute{e}$lec element space, there exists the following super-convergence estimate under the condition that $\Delta t = O(h^{\frac{d}{2}})$
\begin{eqnarray}
\label{equation:eq-a68}
&&\max_{1\leq j \leq n}\|{\bf E}_j-\Pi_{2h}^1 {\bf E}^h_j\|_0\leq C(\Delta t+h^2),\ \
\max_{1\leq j \leq n}\|{\bf P}_j-\pi_{2h}^1 {\bf P}^h_j\|_0\leq C(\Delta t+h^2),
\end{eqnarray}
in which C is independent of $\Delta t$ and h.
\end{theorem}
\begin{proof}
From $(\ref{equation:eq-a65})$-$(\ref{equation:eq-a67})$, we have
\begin{eqnarray}
\label{equation:eq-a70}
&&\|{\bf E}_j-\Pi_{2h}^1 {\bf E}^h_j\|_0=\|\Pi_{2h}^1 ({\bf E}^h_j-\Pi^h {\bf E}_j)+(\Pi_{2h}^1 {\bf E}_j-{\bf E}_j)\|_0\\
&&\leq C\|{\bf E}^h_j-\Pi_h {\bf E}_j\|_0+\|\Pi_{2h}^1 {\bf E}_j-{\bf E}_j\|_0.\nonumber
\end{eqnarray}
Similarly, we have
\begin{eqnarray}
\label{equation:eq-a71}
&&\|{\bf P}_j-\pi_{2h}^1 {\bf P}^h_n\|_0=\|\pi_{2h}^1 ({\bf P}^h_n-\pi_h {\bf P}_j)+(\pi_{2h}^1 {\bf P}_j-{\bf P}_j )\|_0\\
&&\leq C\|{\bf P}^h_n-\pi_h {\bf P}_j\|_0+\|\pi_{2h}^1 {\bf P}_j-{\bf P}_j\|_0.\nonumber
\end{eqnarray}
From the proof of Theorem $ \ref{theorem:th5-3} $ and Lemma \ref{lem:super convergence},  the super-close of $L^2$ error estimate for $ {\bf \Theta}^h_j,{\bf \Psi}^h_j$ in a similar way can be derived by
\begin{eqnarray}
\label{equation:eq-a73}
&&\| {\bf \Theta}^h_j\|_0^2+\|{\bf \Psi}^h_j\|_0^2\leq C(\Delta t+h^{2}),
\end{eqnarray}
under the a-priori $L^{\infty}$ assumption $(\ref{equation:eq-apri-3})$. As a result, such an assumption could be similarly recovered as
\begin{eqnarray}
\label{equation:eq-a74}
&&\|{\bf \Psi}^h_j\|_{L^{\infty}}\leq \frac{C\|{\bf \Psi}^h_j\|_0}{h^{\frac{d}{2}}}\leq \frac{C(\Delta t+h^2)}{h^{\frac{d}{2}}}\leq C.
\end{eqnarray}
This finishes the argument for the a-priori bound $(\ref{equation:eq-70})$.

Finally, with the help of (\ref{equation:eq-a65})-(\ref{equation:eq-a67}), we obtain
\begin{eqnarray}
\label{equation:eq-a76}
&&\|{\bf E}_j-\Pi_{2h}^1 {\bf E}^h_j\|_0\leq Ch^2\|{\bf E}_j\|_2,~~~\|{\bf P}_j-\pi_{2h}^1 {\bf P}^h_j\|_0\leq Ch^2\|{\bf P}_j\|_2,
\end{eqnarray}
which completes the proof of Theorem \ref{thm:super convergence}.
\end{proof}

\section{Numerical Examples}

In this section, we provide some numerical examples in the transverse electromagnetic(TE) case to confirm our theoretical analysis,
 with ${\bf E}=[E_1,E_2,0]$ and ${\bf P}=[P_1,P_2,0]$. For convenience, we still denote ${\bf E}=[E_1,E_2]$ and ${\bf P}=[P_1,P_2]$. The computations are performed using the Matlab code.
In these numerical examples, we observe that, numerical results have shown that, the stability and convergence are well preserved with a  relaxed constraint for the time step, $\Delta t = O (h)$. For the experiments, the parameters are taken as
$\epsilon_0=1,\mu_0=1, \tau=1,\epsilon_s=2,\epsilon_\infty=1.$

Define
\begin{eqnarray*}
&&errE=\|{\bf E}_n-{\bf E}^h_n\|_0, \quad errP=\|{\bf P}_n-{\bf P}^h_n\|_0,\ \
errCurlE= \|\nabla\times({\bf E}_n-{\bf E}^h_n)\|_0,\\
&&SerrE=\| {\bf E}_n-\Pi_{2h}^1 {\bf E}^h_n\|_0, \quad SerrP=\|{\bf  P}_n-\pi_{2h}^1{\bf P}^h_n\|_0.
\end{eqnarray*}

\subsection{Example One}
Denoting the real solution
\begin{eqnarray}
{\bf E}=\exp(t)[sin((1+x)y)(y-1)*|2x-1|^{\alpha}),
sin((1+y)x)(x-1)|2y-1|^{\alpha}],
\end{eqnarray}
and
letting $\alpha=2.1$, ${\bf P}(x,y,t)={\bf E}(x,y,t)$, we can see ${\bf n}\times {\bf E}=0.$

From the Table 1, we can see that the convergent order  in spatial is $O(h)$ with respect to the lowest Raviart-Thomas-N$\acute{e}$delec element as well as that of the $L^2$ super-convergence in Table 2.
In figure 1, we demonstrate the numerical solution for $E_{1h}$ , $ E_{2h}$ (the two left ) and $P_{1h}$, $P_{2h}$ (the two right) at grids on the mesh $32\times 32$ after 100 time steps by $ \Delta t=1e-5$. In figure 2, we show the error for two components of ${\bf E}_{h}$ (the two left ) and ${\bf P}_{h}$ (the two right), respectively. In figure 3 and figure 4, we present the super-convergent solutions and error. In figure 5, we also give the vector values at grids on the mesh for the numerical solutions $ {\bf E}^h_n,\ {\bf P}^h_n$ and the super-convergent solutions $ \Pi_{2h}^1 {\bf E}^h_n,\ \pi_{2h}^1{\bf P}^h_n$, respectively.

\begin{table}[!htbp]
\begin{center}
\label{tab:1}       
\caption{Convergence error results for ${\bf E}$, ${\bf P}$  in example 1, with time step size $\Delta t = 10^{-5}$.}
\begin{tabular}{ccccccccc}
\hline\noalign{\smallskip}
$N\times N$& $errE$ &order & $errP$ &order  & errCurlE &order\\
\noalign{\smallskip}\hline\noalign{\smallskip}
4$\times$4   & 0.0575    &--        &0.0346     &--       & 0.2397   &-- \\
8$\times$8   & 0.0296    & 0.9582   &0.0176     & 0.9717    & 0.1239  &0.9524 \\
16$\times$16 & 0.0149 &0.9898     &0.0089   &0.9923  &0.0624 & 0.9901\\
32$\times$32 & 0.0075  &0.9975      &0.0046     & 0.9940  &0.0312   &  0.9982\\
\noalign{\smallskip}\hline
\end{tabular}
\end{center}
\end{table}


\begin{table}[!htbp]
\begin{center}
\label{tab:1}       
\caption{Super-convergence error results for ${\bf E}$, ${\bf P}$ in example 1, with time step size $\Delta t = 10^{-5}$.}
\begin{tabular}{ccccccccc}
\hline\noalign{\smallskip}
$N\times N$& $SerrE$ &order & $SerrP$ &order  \\
\noalign{\smallskip}\hline\noalign{\smallskip}
4$\times$4   & 0.0305   &--        & 0.0473    &--        \\
8$\times$8   &  0.0073   & 2.0625  &0.0119     & 1.9968     \\
16$\times$16 & 0.0018  &2.0518  &0.0029     &2.0414  \\
32$\times$32 & 0.0004  &2.0279  &0.0007    & 2.0114   \\
\noalign{\smallskip}\hline
\end{tabular}
\end{center}
\end{table}


\begin{figure}
\centering
        \hbox{
\includegraphics [width=1.0in,height=1.0in]{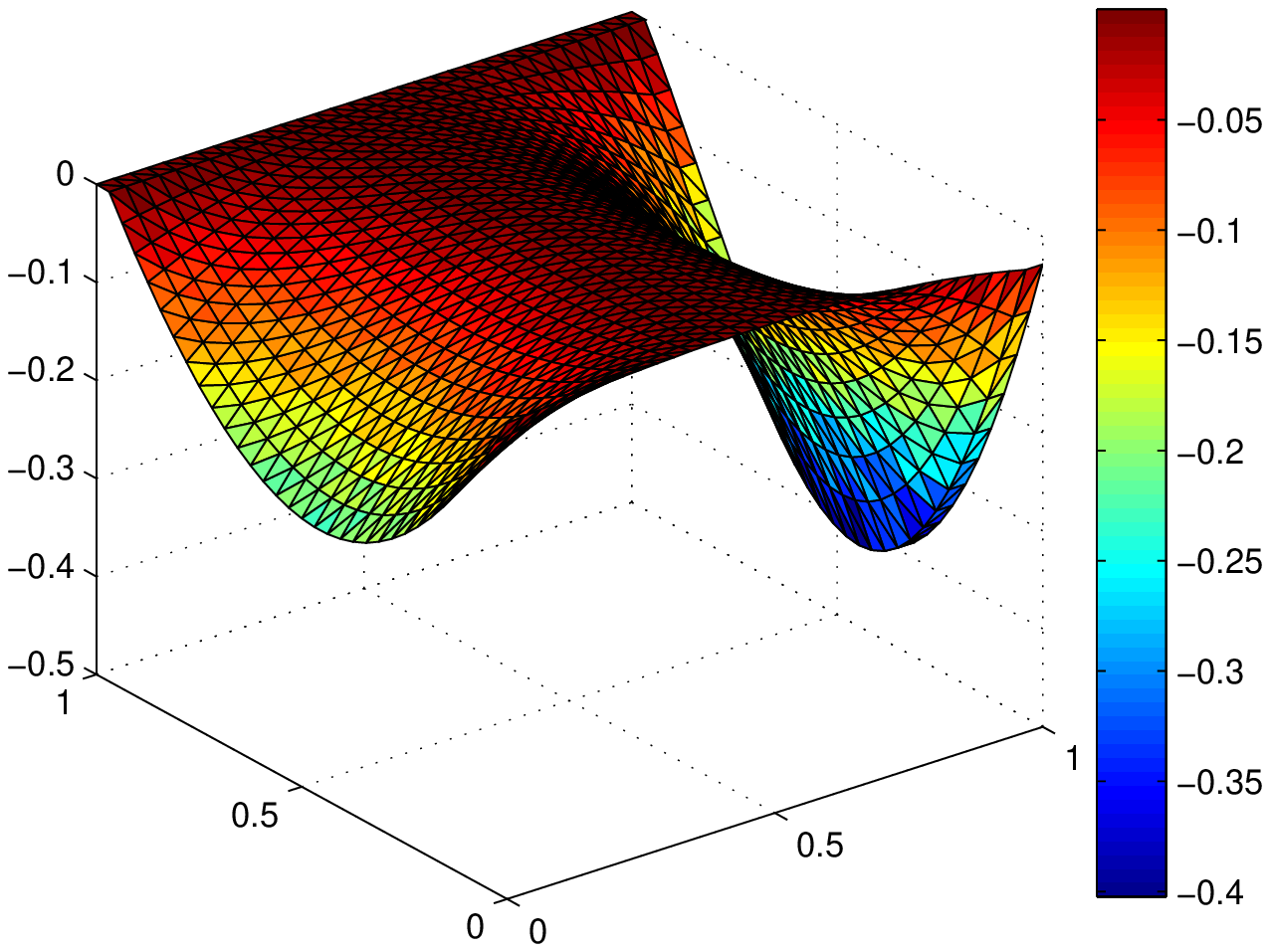}
\includegraphics [width=1.0in,height=1.0in]{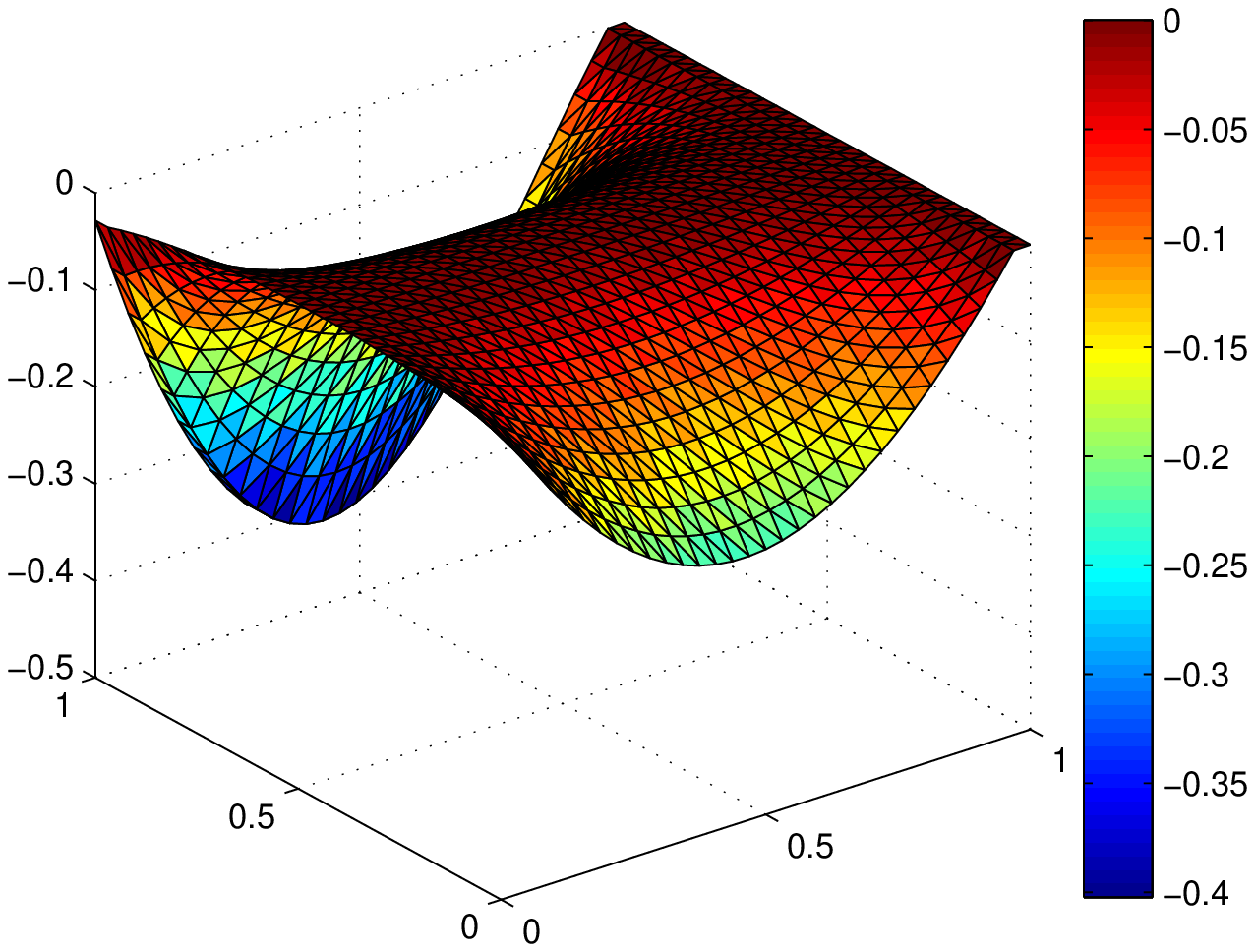}
\includegraphics [width=1.0in,height=1.0in]{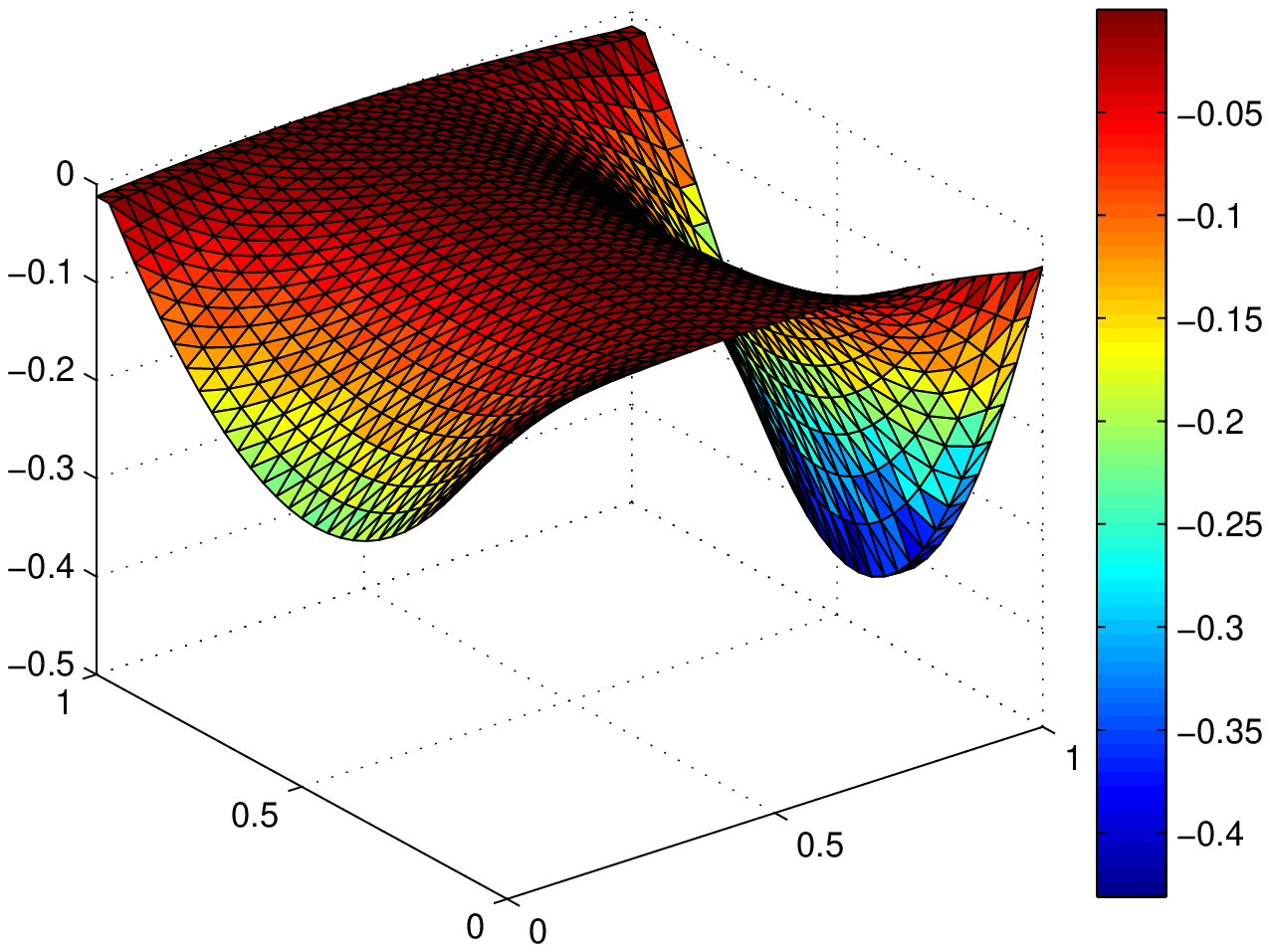}
\includegraphics [width=1.0in,height=1.0in]{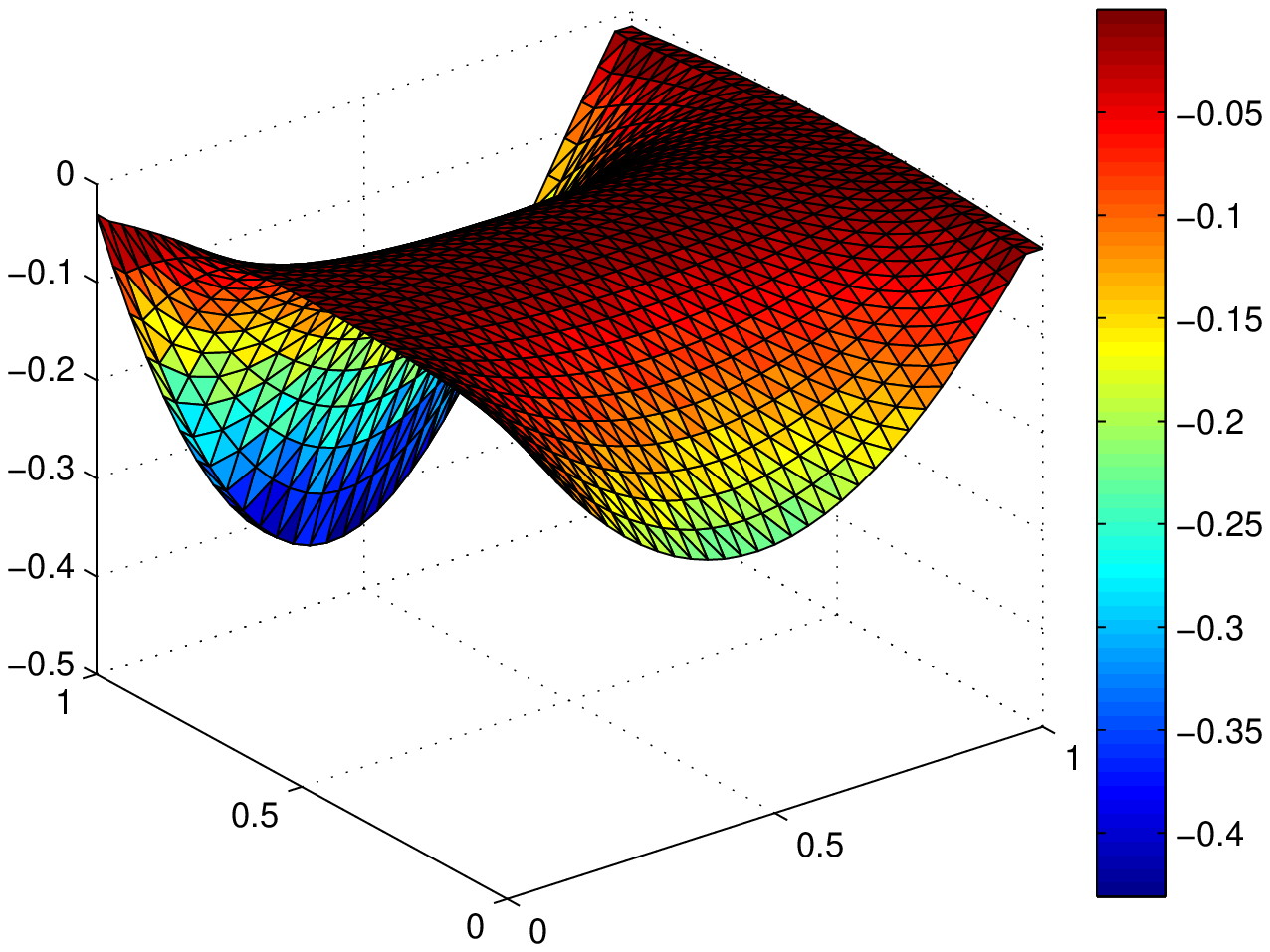}
 }
\label{fig:1}\caption{Numerical solution for $E_{1h}$ (the first left) , $E_{2h}$ (the second left ) and $P_{1h}$ (the first right) $ P_{2h}$ (the second right).}
\end{figure}

\begin{figure}
\centering
        \hbox{
\includegraphics [width=1.0in,height=1.0in]{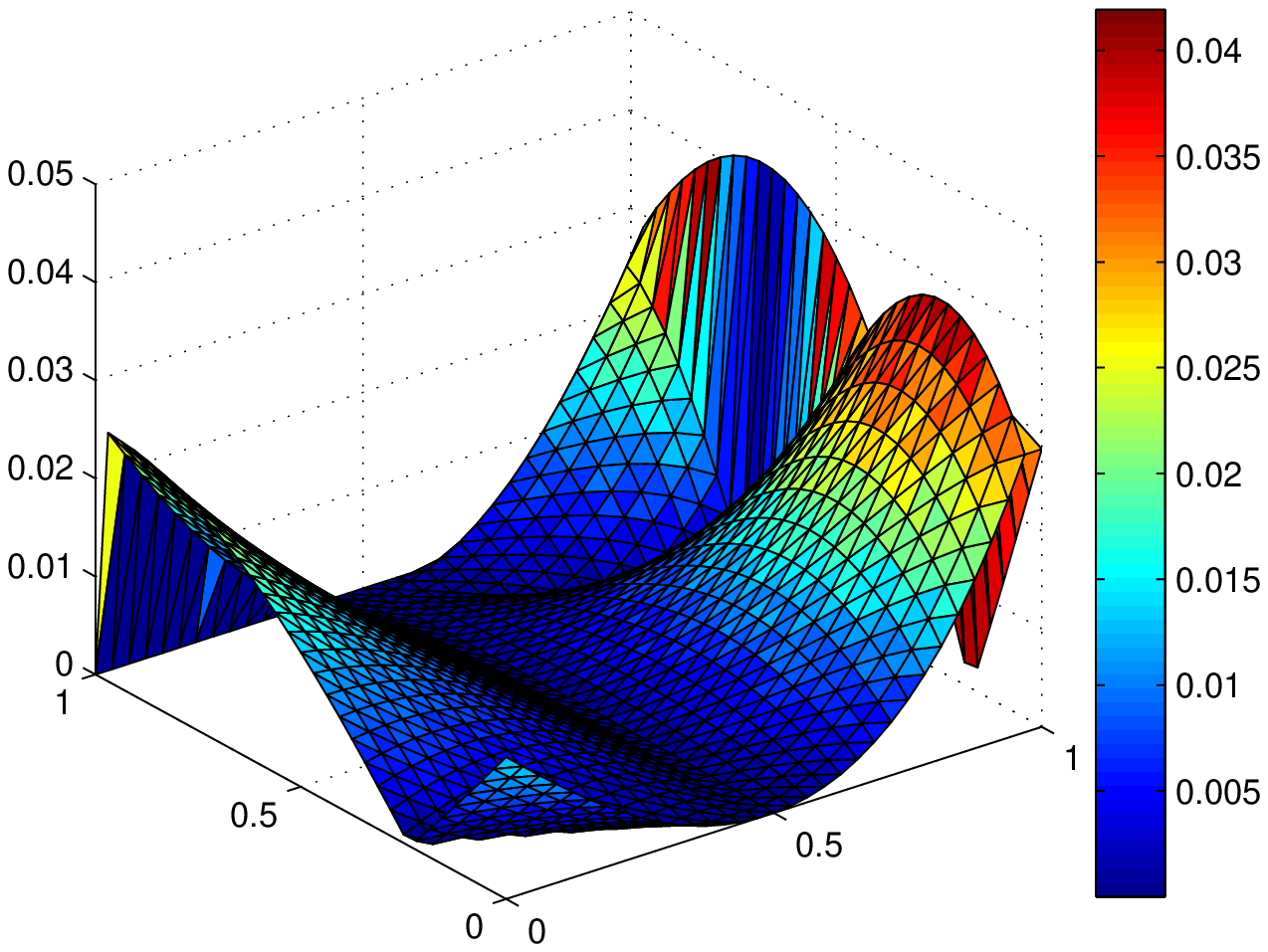}
\includegraphics [width=1.0in,height=1.0in]{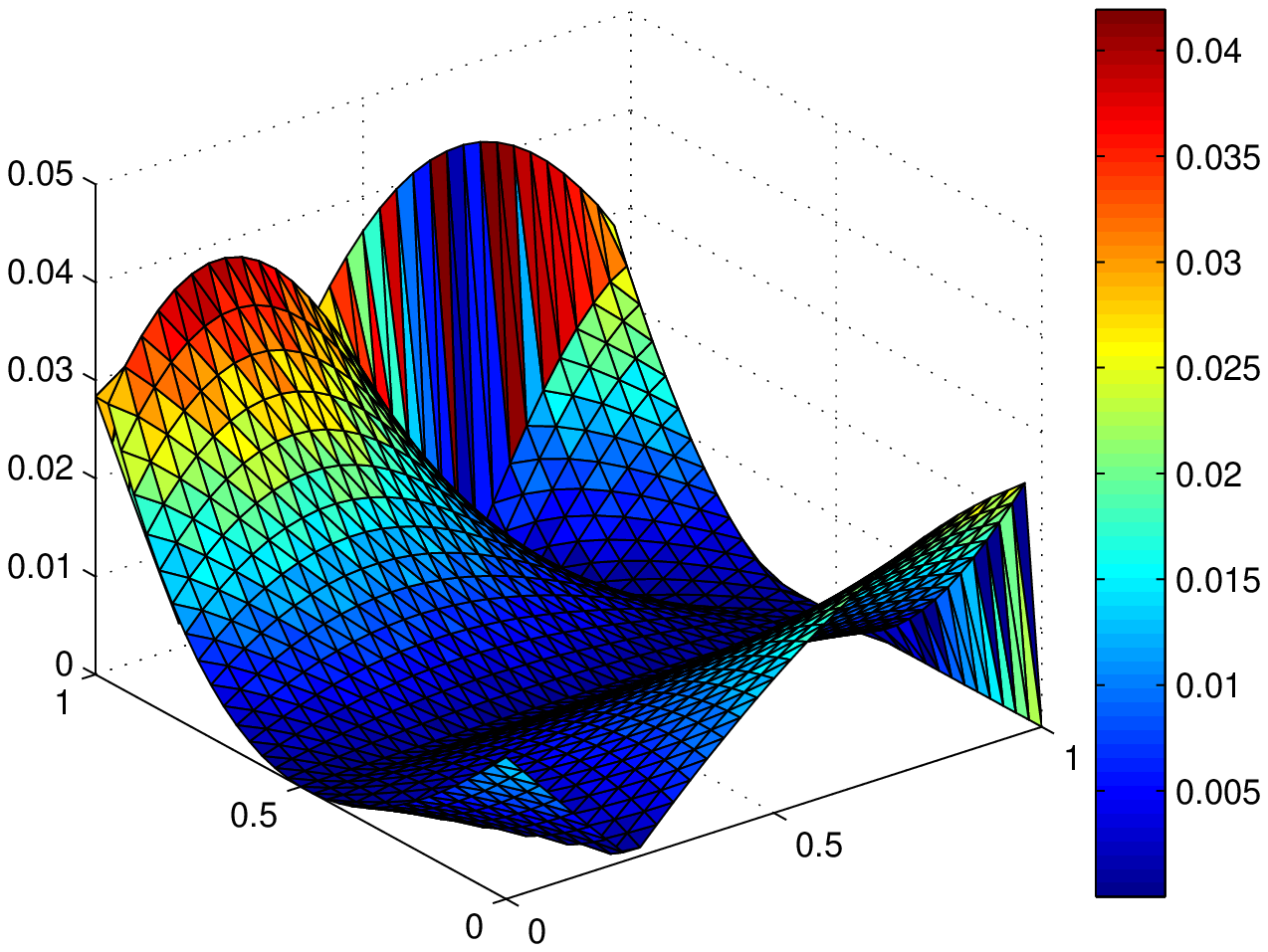}
 \includegraphics [width=1.0in,height=1.0in]{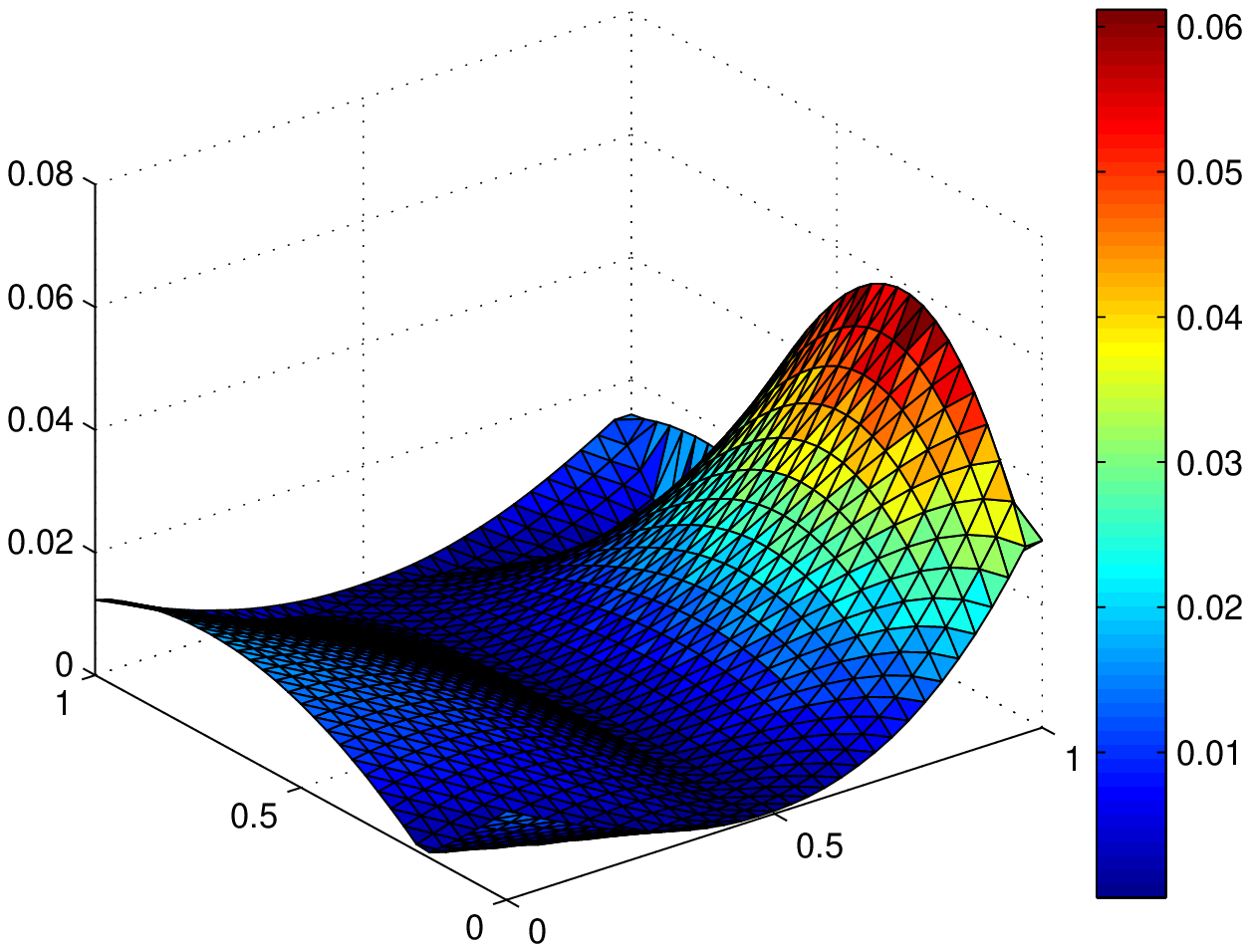}
\includegraphics [width=1.0in,height=1.0in]{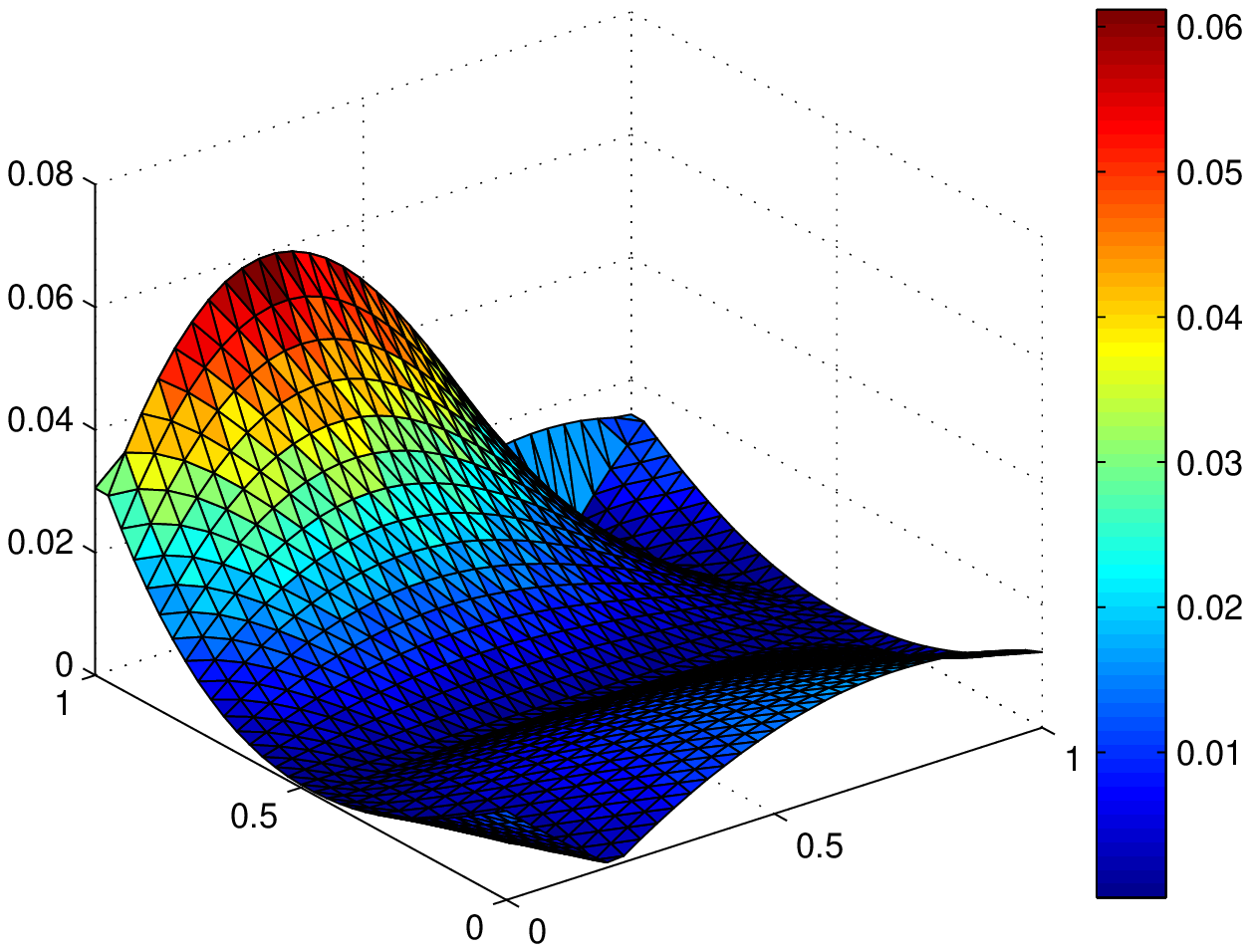}}
\label{fig:2}\caption{Error values for $E_{1h}$ (the first left) , $E_{2h}$ (the second left ) and $P_{1h}$ (the first right) $P_{2h}$ (the second right).}
\end{figure}

\begin{figure}
\centering
        \hbox{
\includegraphics [width=1.0in,height=1.0in]{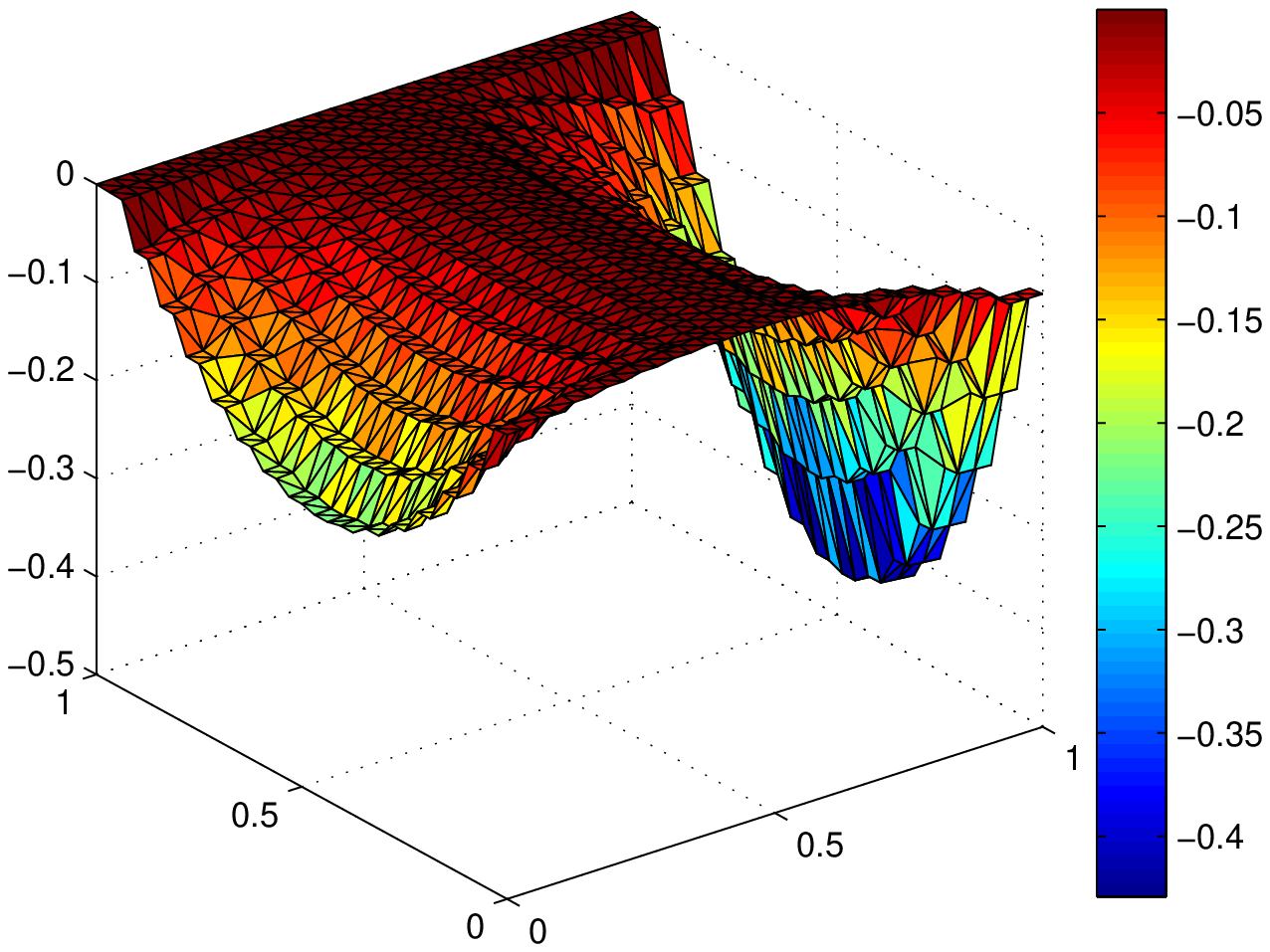}
\includegraphics [width=1.0in,height=1.0in]{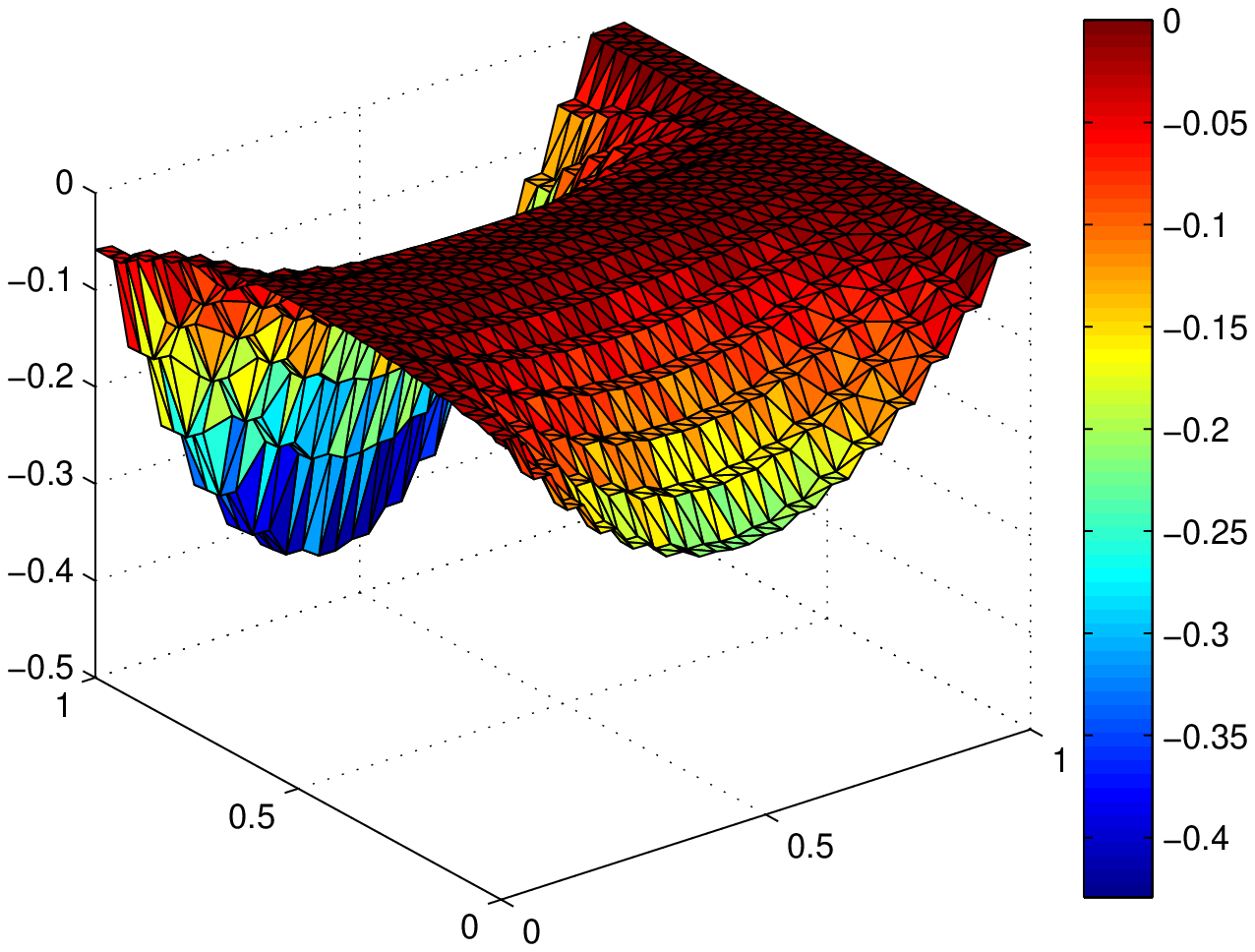}
\includegraphics [width=1.0in,height=1.0in]{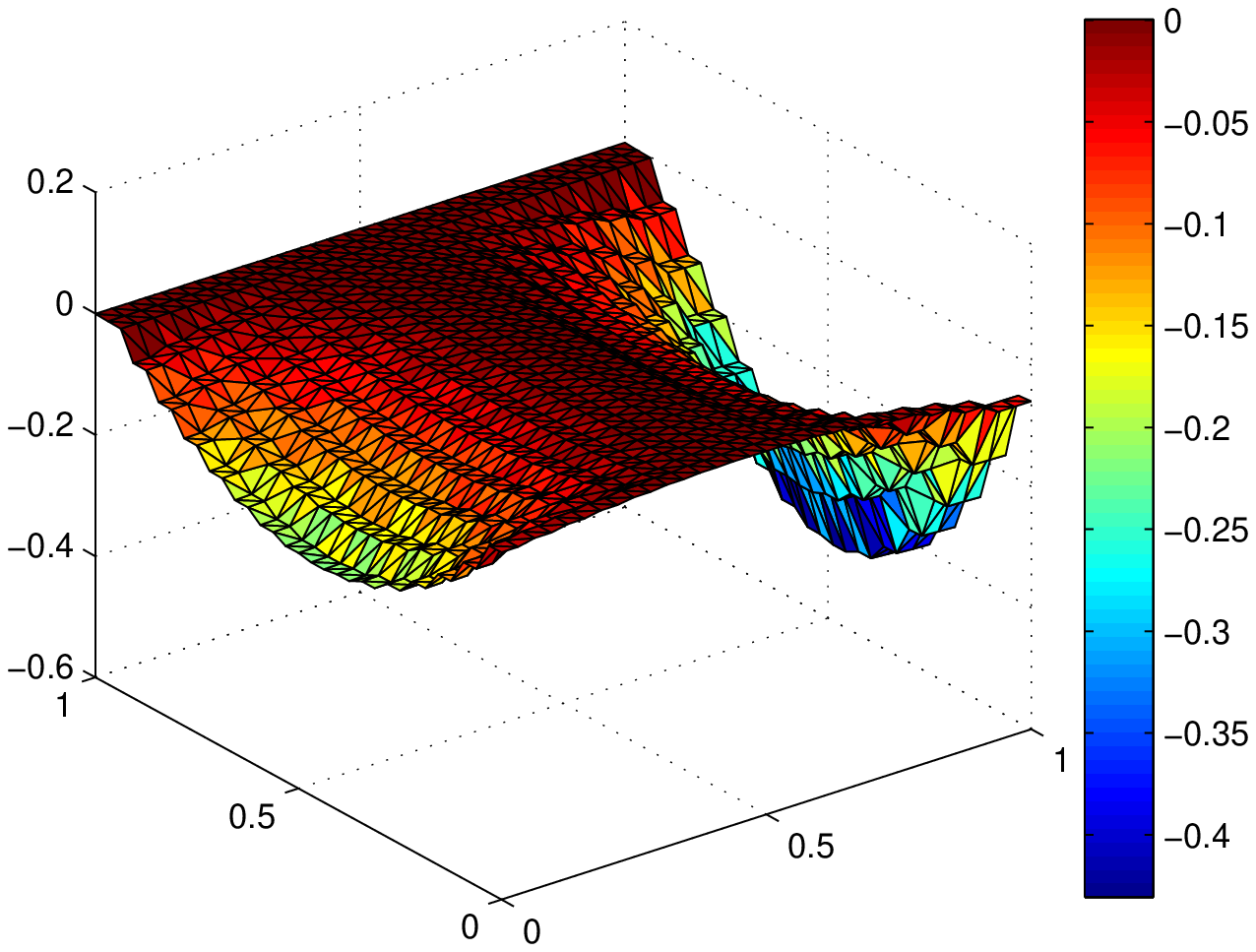}
\includegraphics [width=1.0in,height=1.0in]{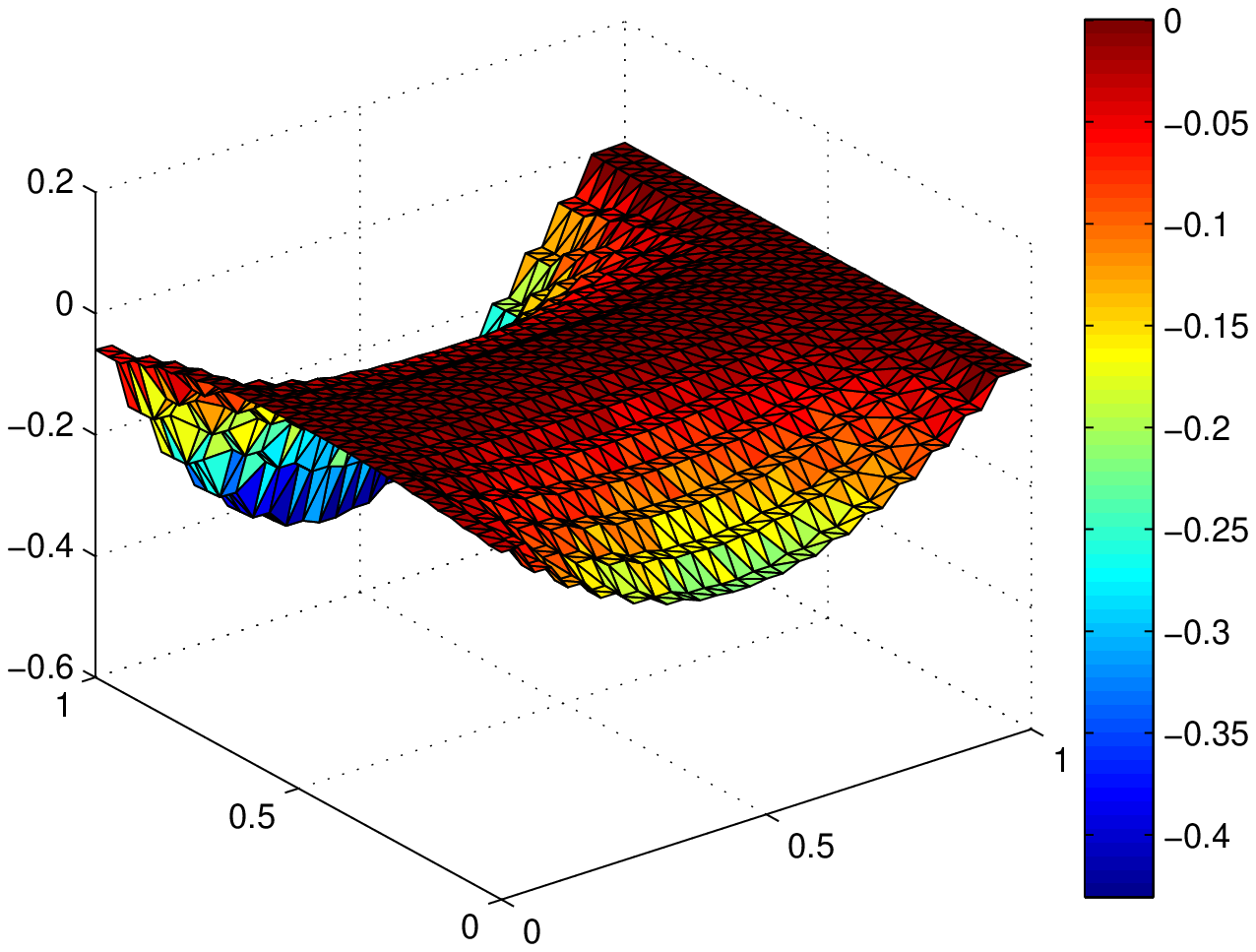}
 }
\label{fig:3}\caption{ Numerical solution for $\Pi_{2h}^1 E_{1h}$ (the first left) , $\Pi_{2h}^1E_{2h}$ (the second left ) and $\pi_{2h}^1P_{1h}$ (the first right) $\pi_{2h}^1P_{2h}$ (the second right) by super-convergence technique.}
\end{figure}

\begin{figure}
\centering
        \hbox{
\includegraphics [width=1.0in,height=1.0in]{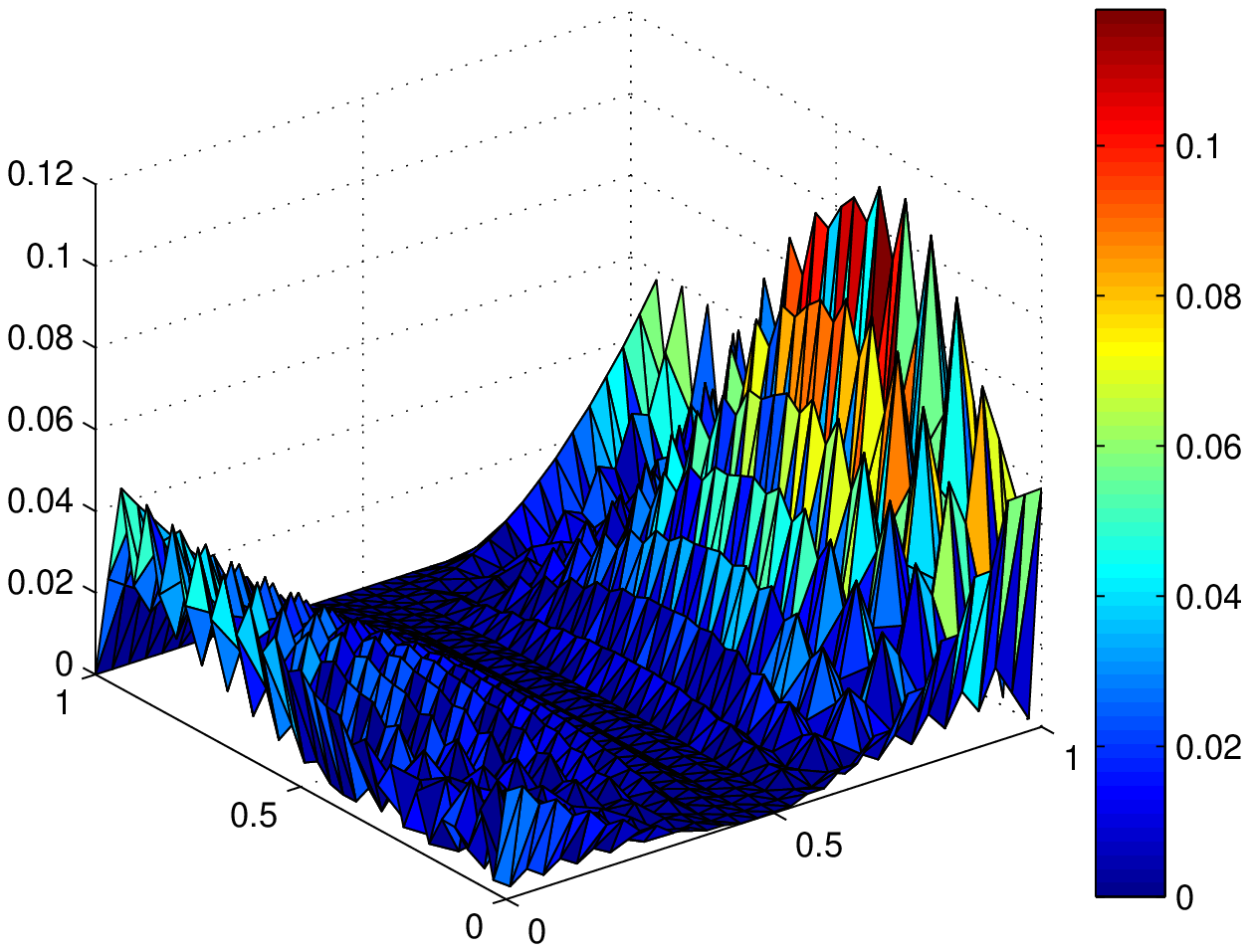}
\includegraphics [width=1.0in,height=1.0in]{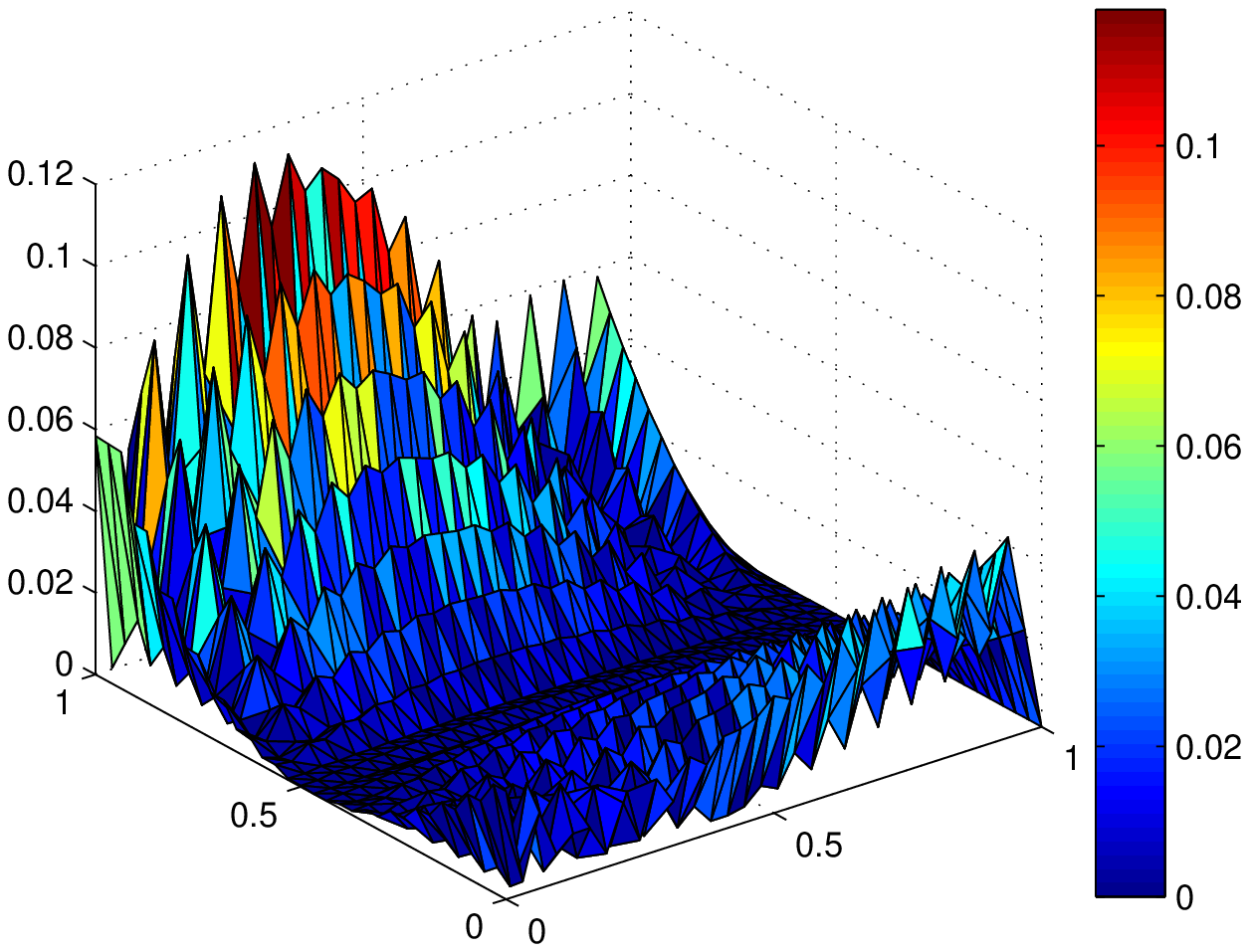}
 \includegraphics [width=1.0in,height=1.0in]{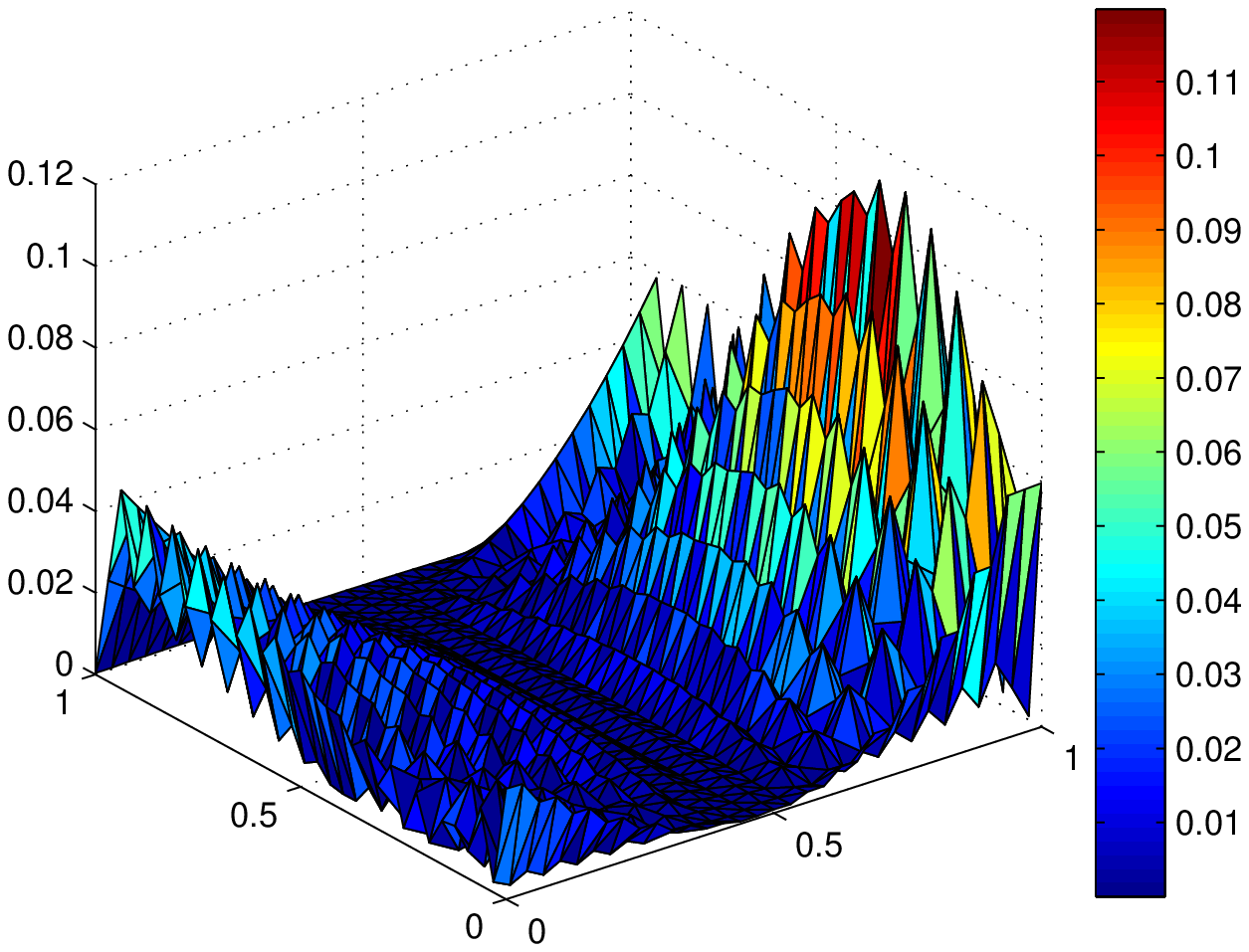}
\includegraphics [width=1.0in,height=1.0in]{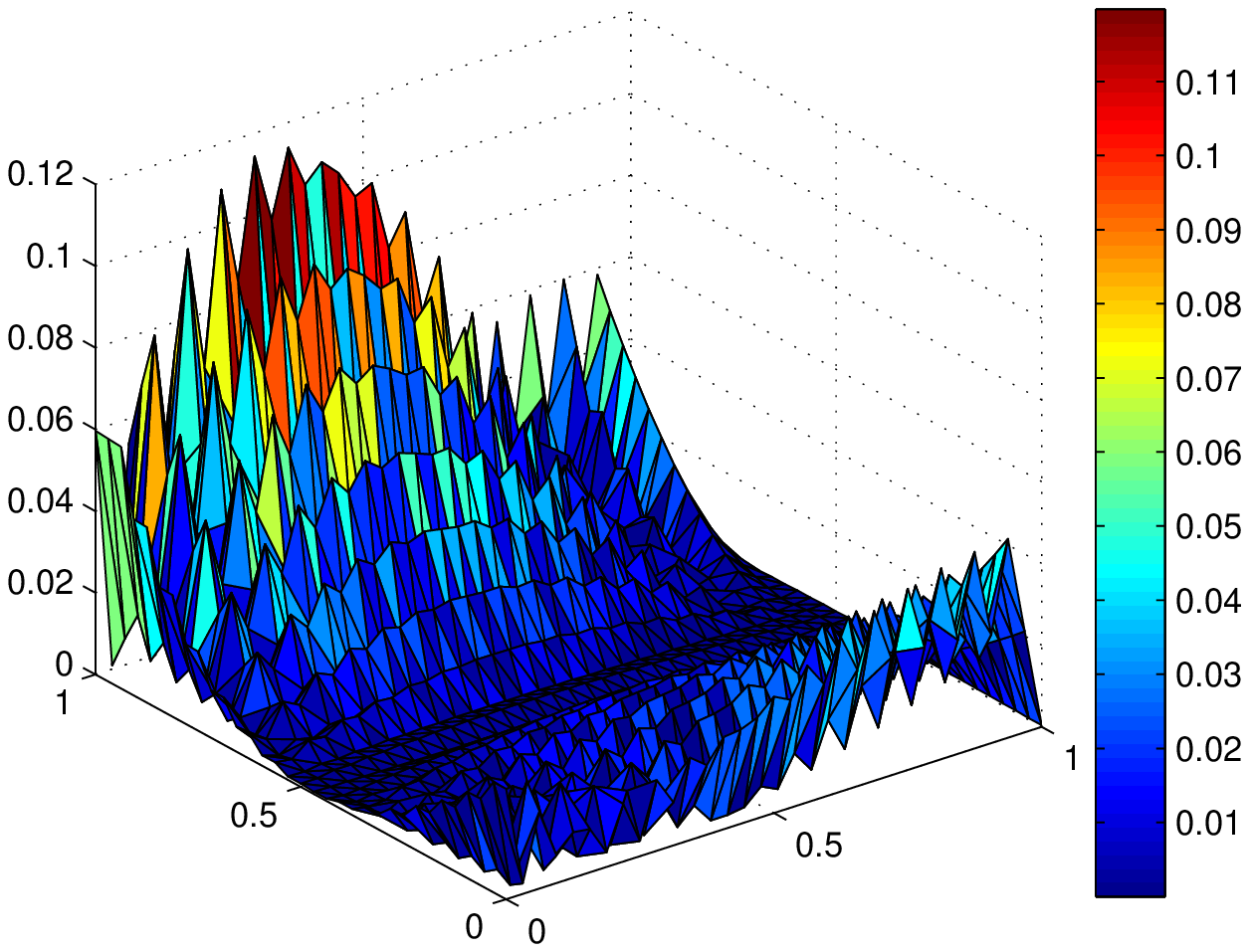}}
\label{fig:4}\caption{Error values for $E_{1h}$ (the first left) , $E_{2h}$ (the second left ) and $P_{1h}$ (the first right) $P_{2h}$ (the second right) by the  super-convergence technique.}
\end{figure}


\begin{figure}
\centering
        \hbox{
\includegraphics [width=1.0in,height=1.0in]{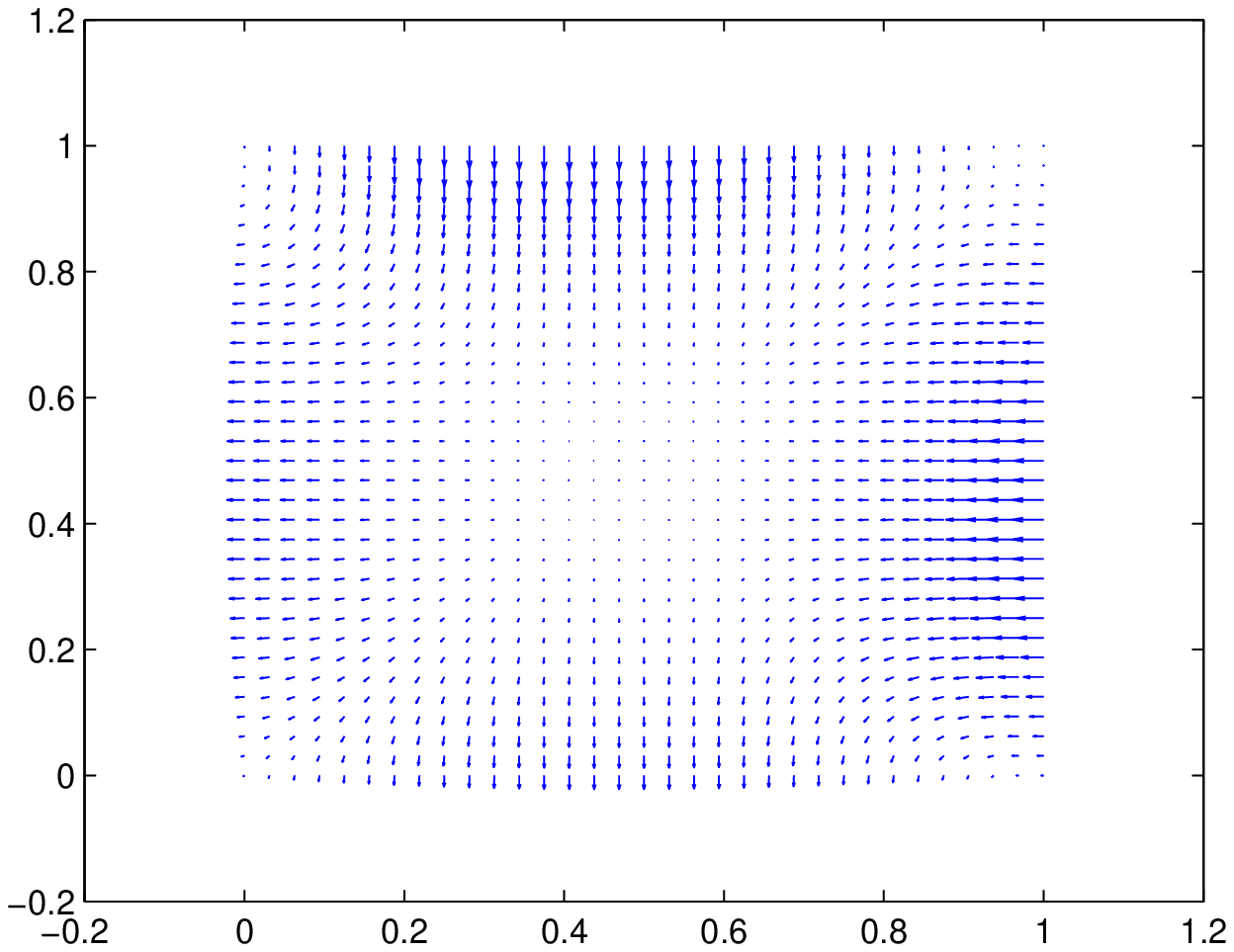}
\includegraphics [width=1.0in,height=1.0in]{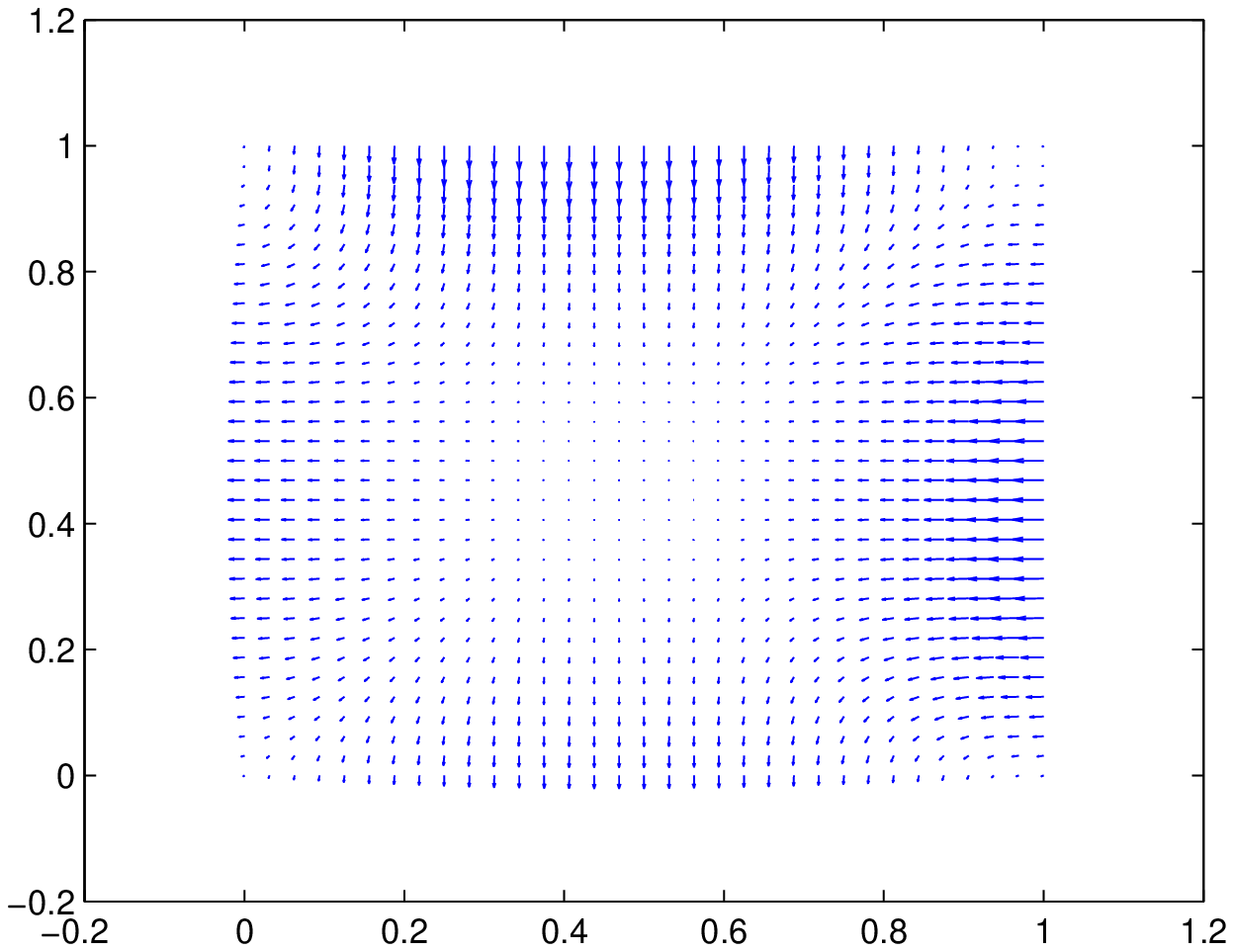}
\includegraphics [width=1.0in,height=1.0in]{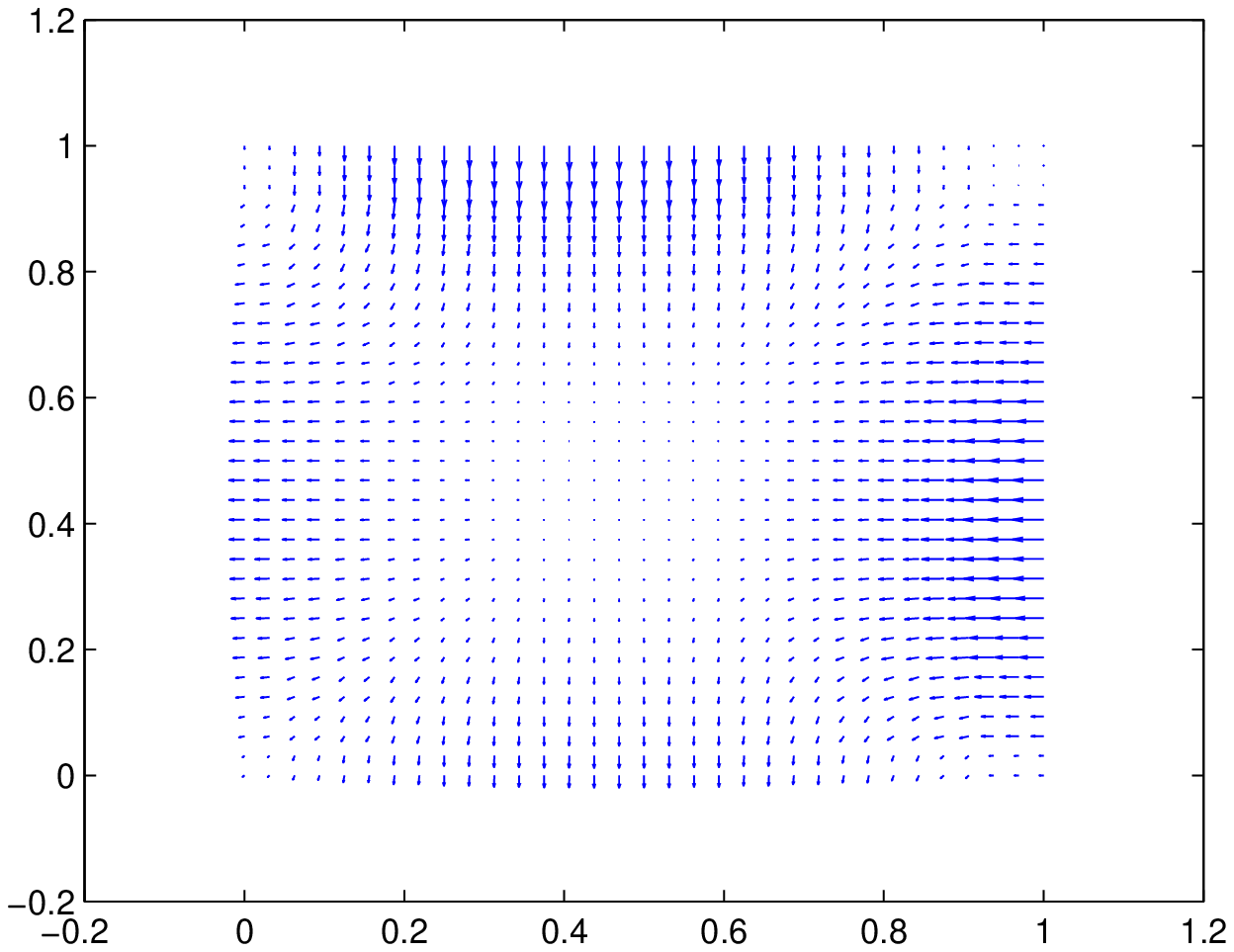}
\includegraphics [width=1.0in,height=1.0in]{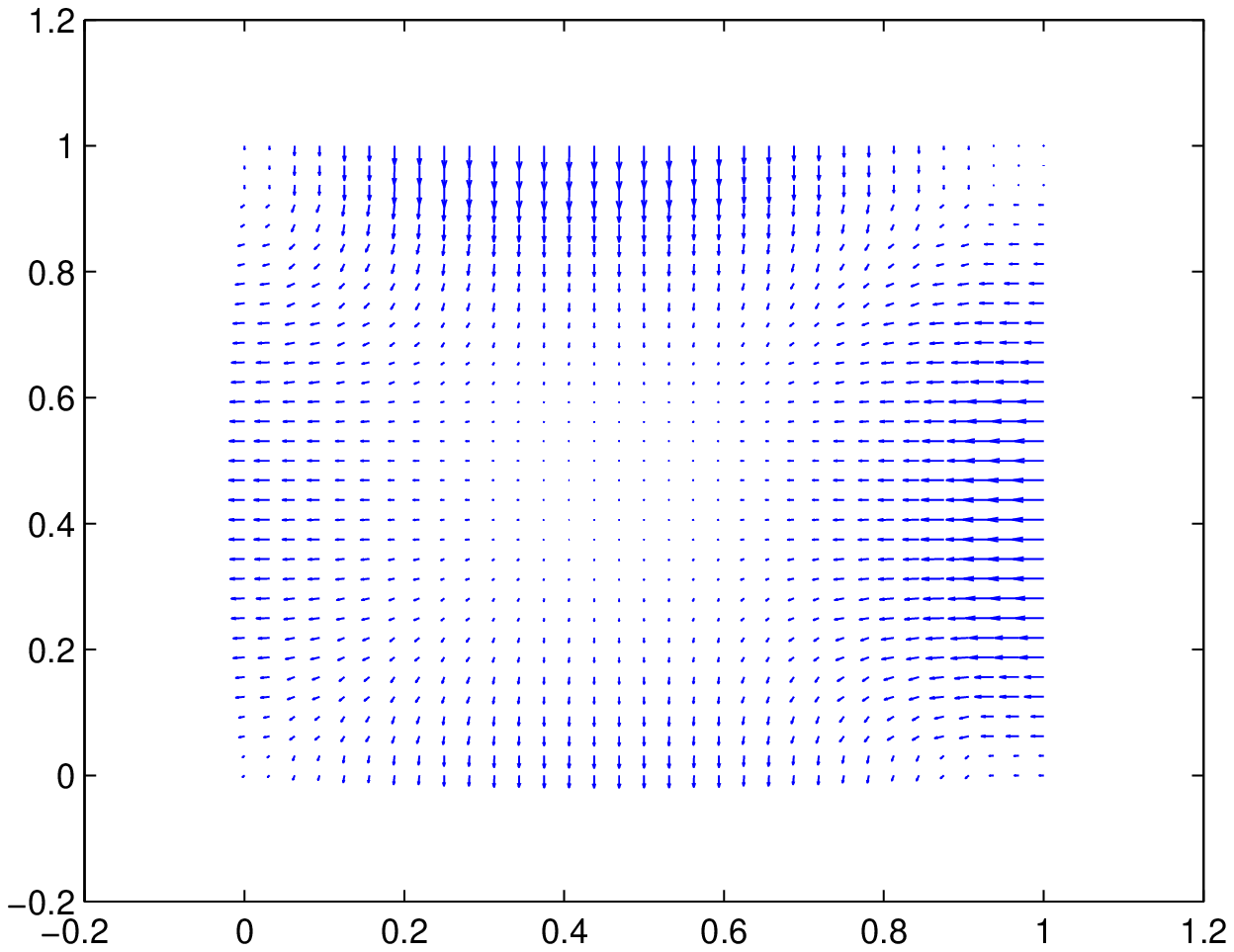}

 }
\label{fig:5}\caption{Vector values at the grids for the numerical solutions of ${\bf E}_h$, ${\bf P}_h$ (two left) and super-convergence vector values
${\bf E}_h$, ${\bf P}_h$ (two right).}
\end{figure}

\subsection{Example two: L-type domain}
In this subsection, we consider the domain is $L-$type, $\Omega=[0,1]^2/([0.5,1]\times [0, 0.5])$. The analysis solution are given by
\begin{eqnarray*}
{\bf E}&=&\exp(t)[sin(xy)y(y-0.5)(y-1)|2x-1|^\alpha,sin(xy)x(x-0.5)(x-1)|2y-1|^\alpha],\\
{\bf P}&=&{\bf E}.
\end{eqnarray*}
In figure 6-10, we show the numerical solutions, error values, the super-convergent solutions, error values by super-convergent technique and vector values at grids on the mesh $32\times 32$ after 100 time steps by $\Delta t=1e-5.$

\begin{figure}
\centering
        \hbox{
\includegraphics [width=1.0in,height=1.0in]{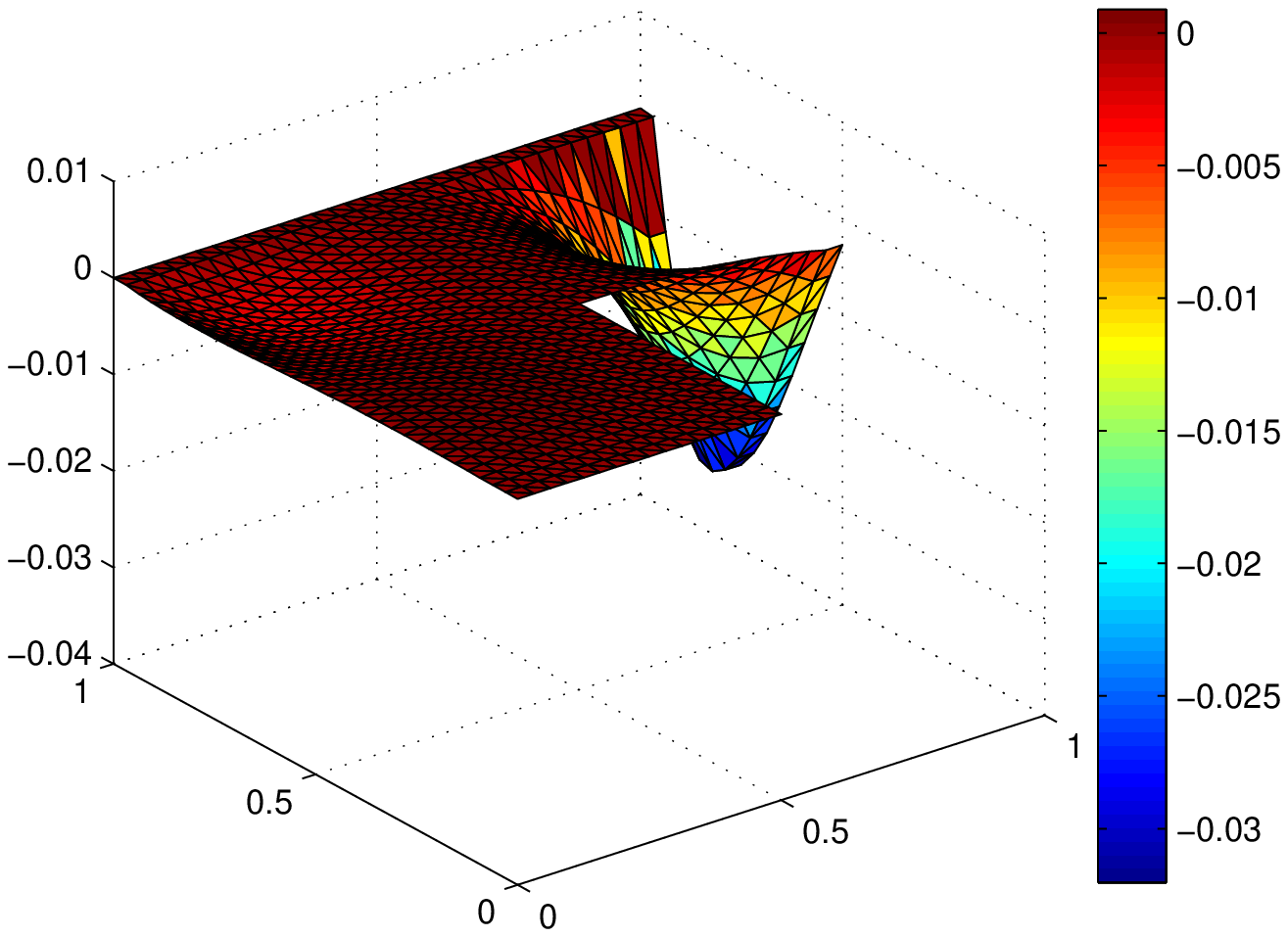}
\includegraphics [width=1.0in,height=1.0in]{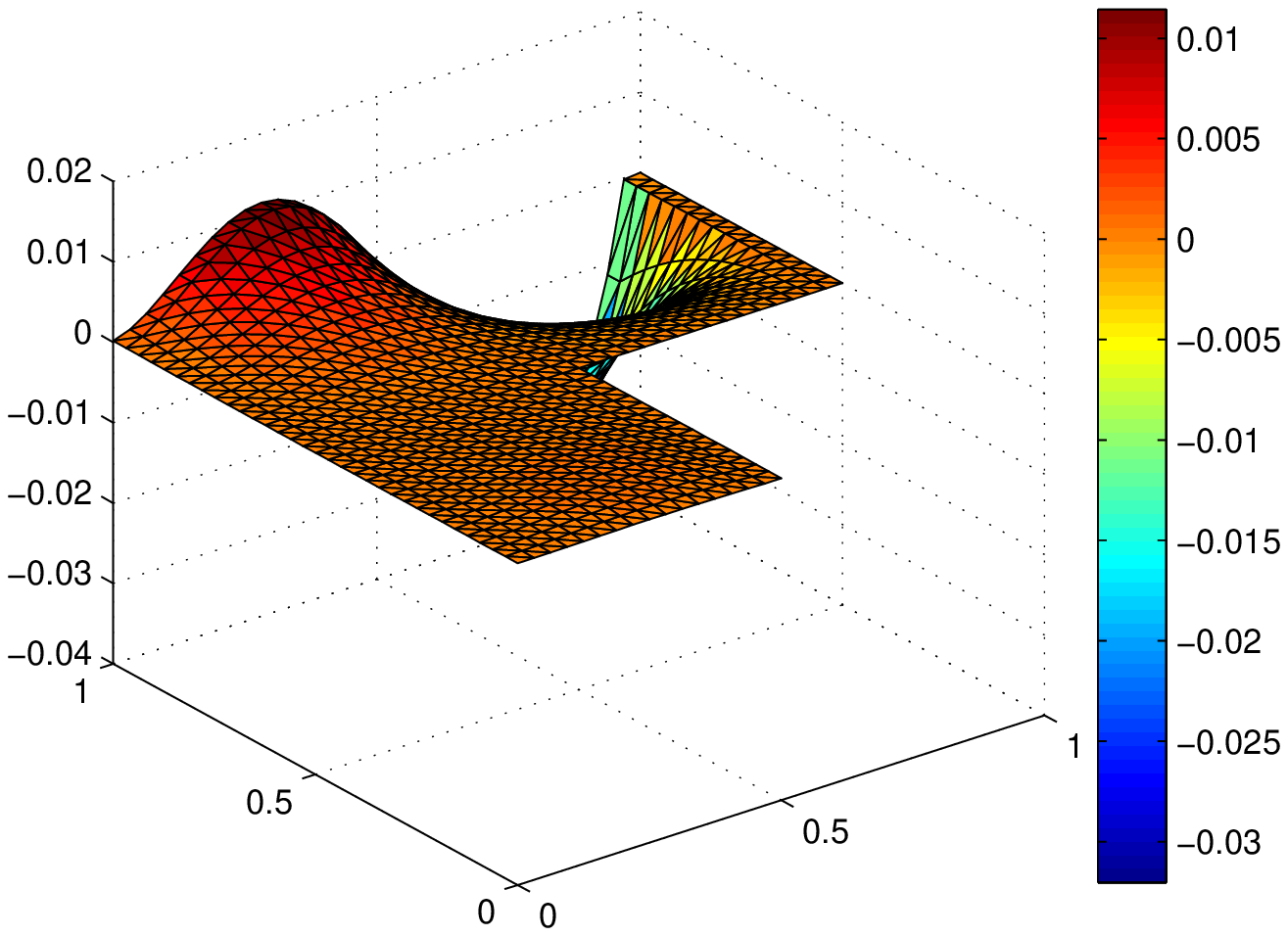}
\includegraphics [width=1.0in,height=1.0in]{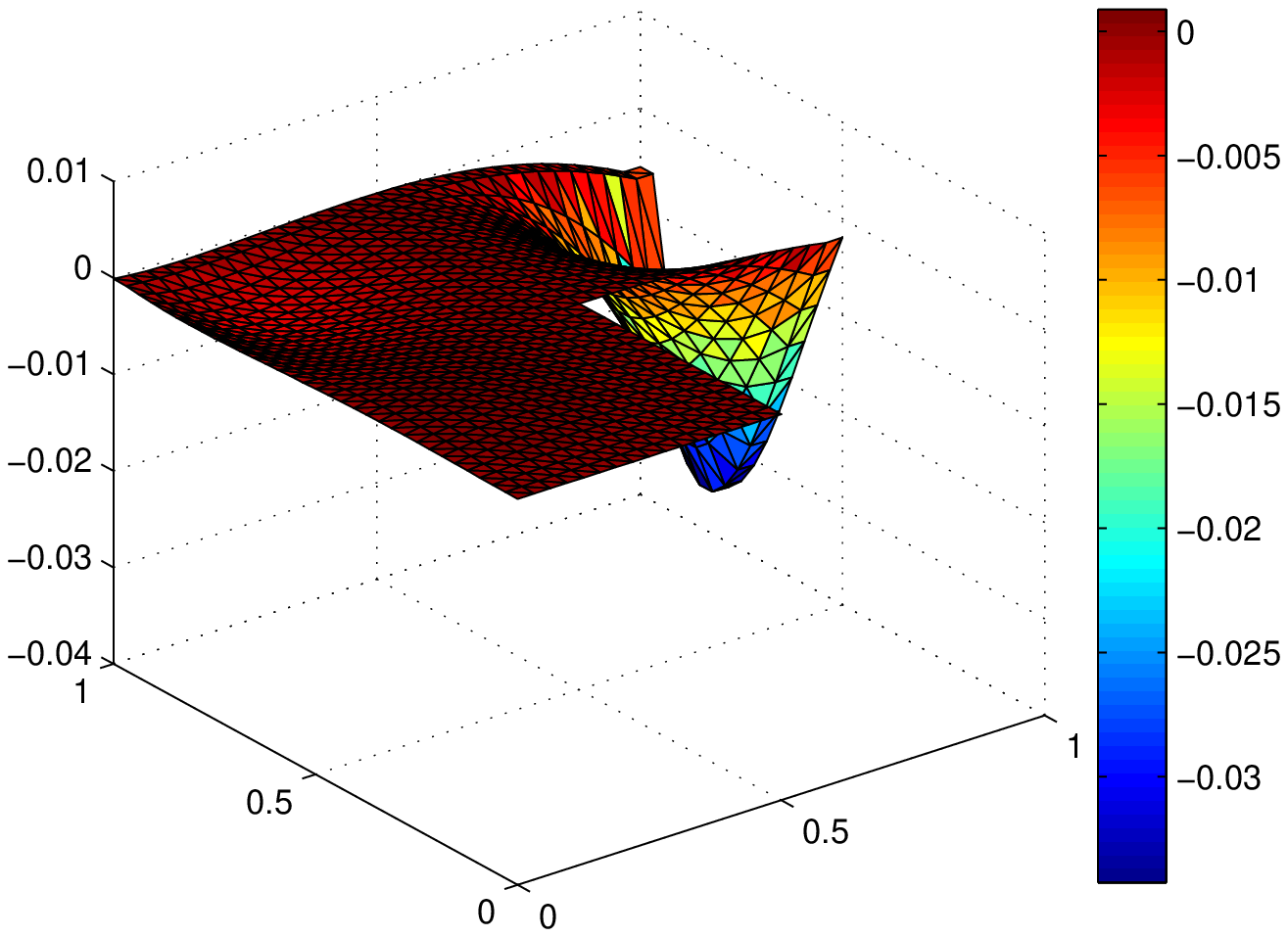}
\includegraphics [width=1.0in,height=1.0in]{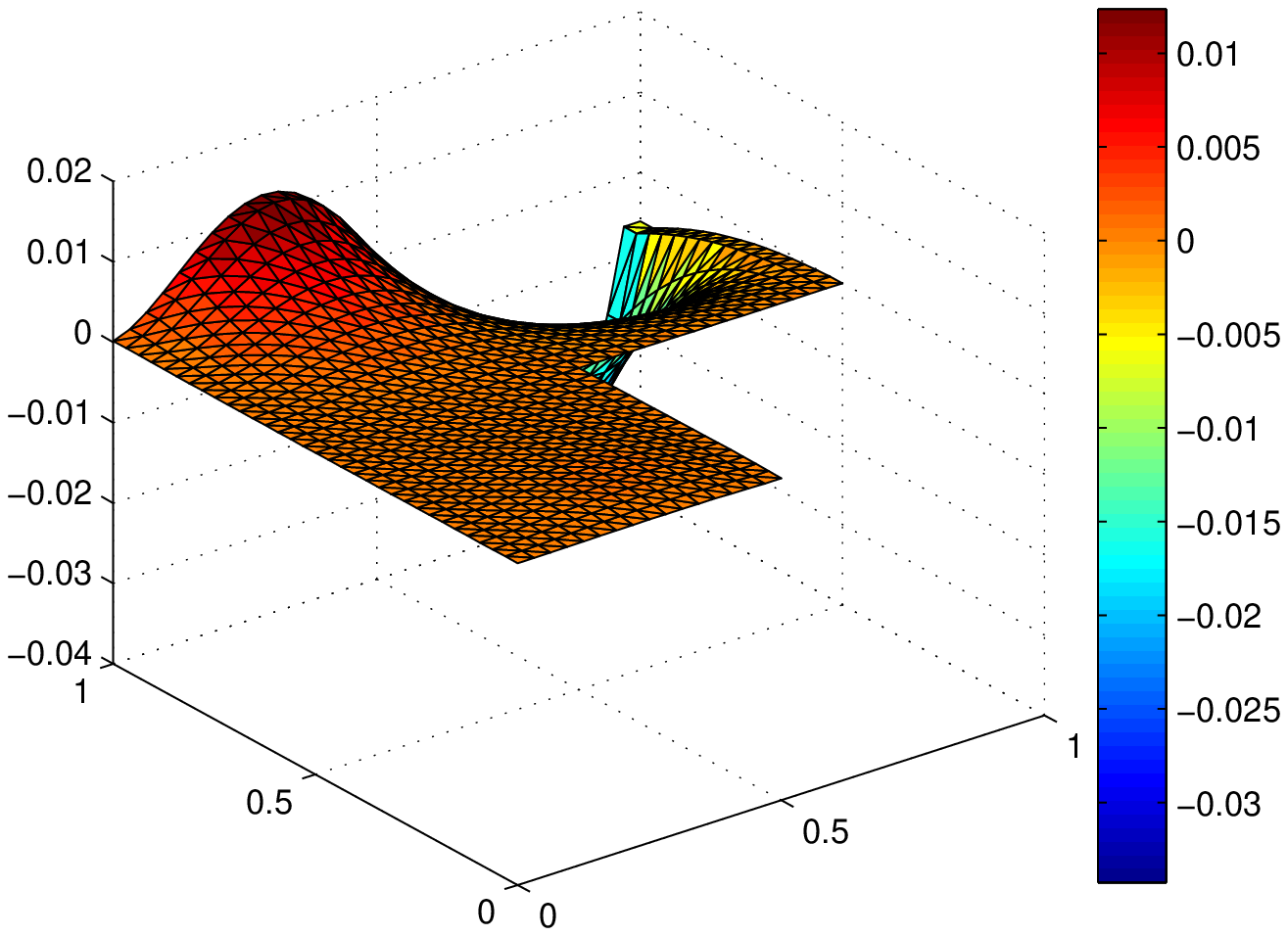}
 }
\label{fig:1}\caption{Numerical solution for $E_{1h}$ (the first left) , $E_{2h}$ (the second left ) and $P_{1h}$ (the first right) $P_{2h}$ (the second right) on the $L-$type domain.}
\end{figure}

\begin{figure}
\centering
        \hbox{
\includegraphics [width=1.0in,height=1.0in]{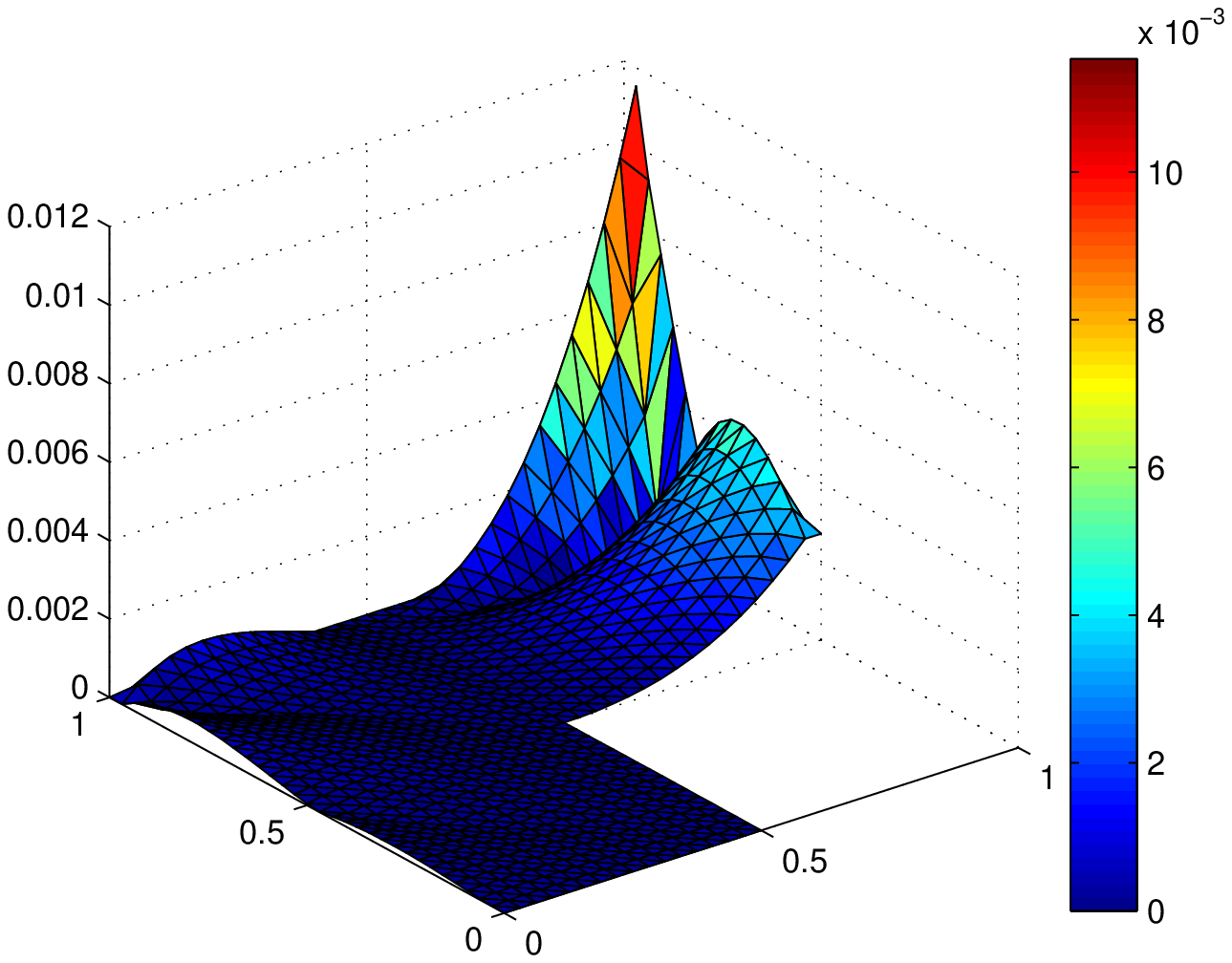}
\includegraphics [width=1.0in,height=1.0in]{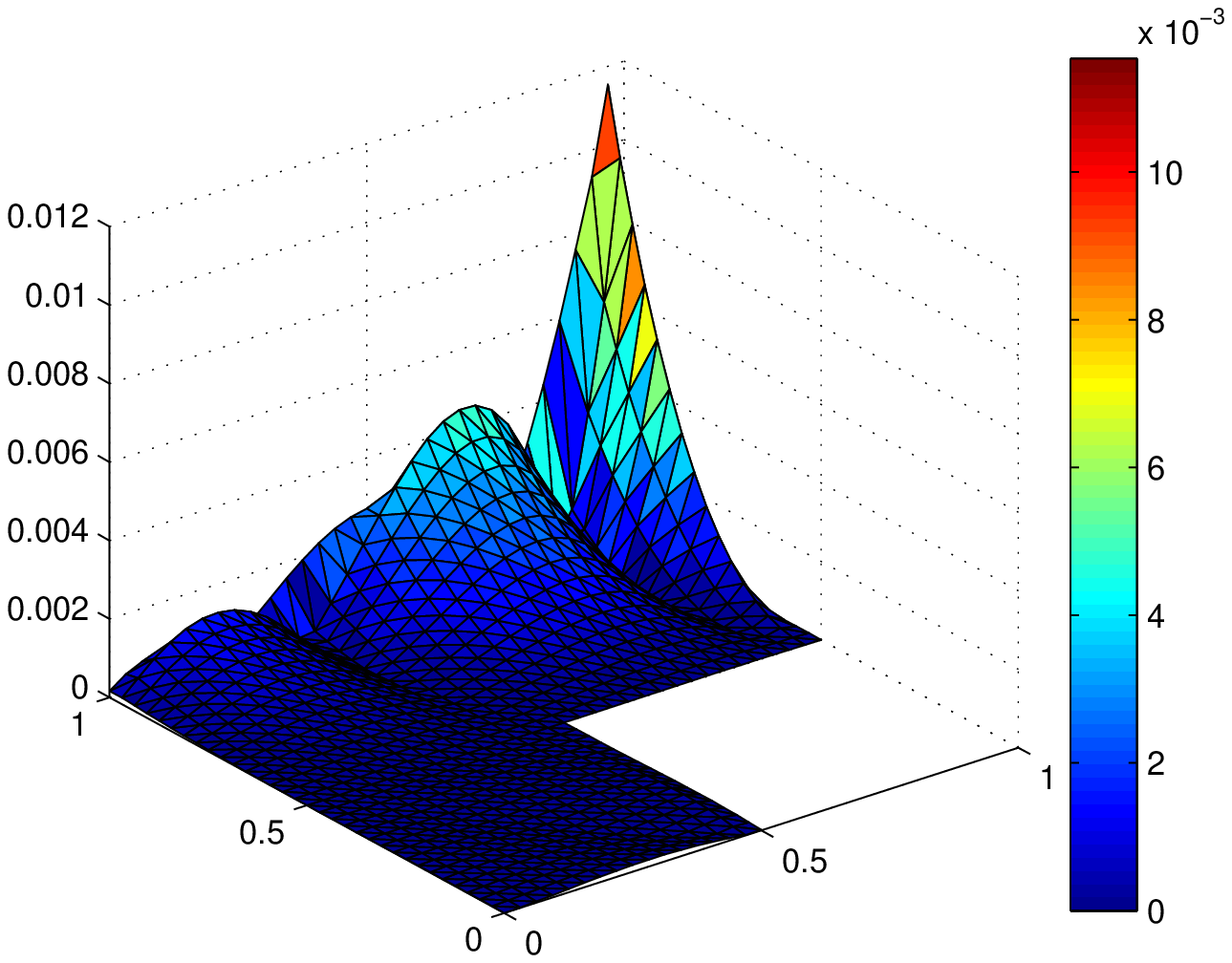}
 \includegraphics [width=1.0in,height=1.0in]{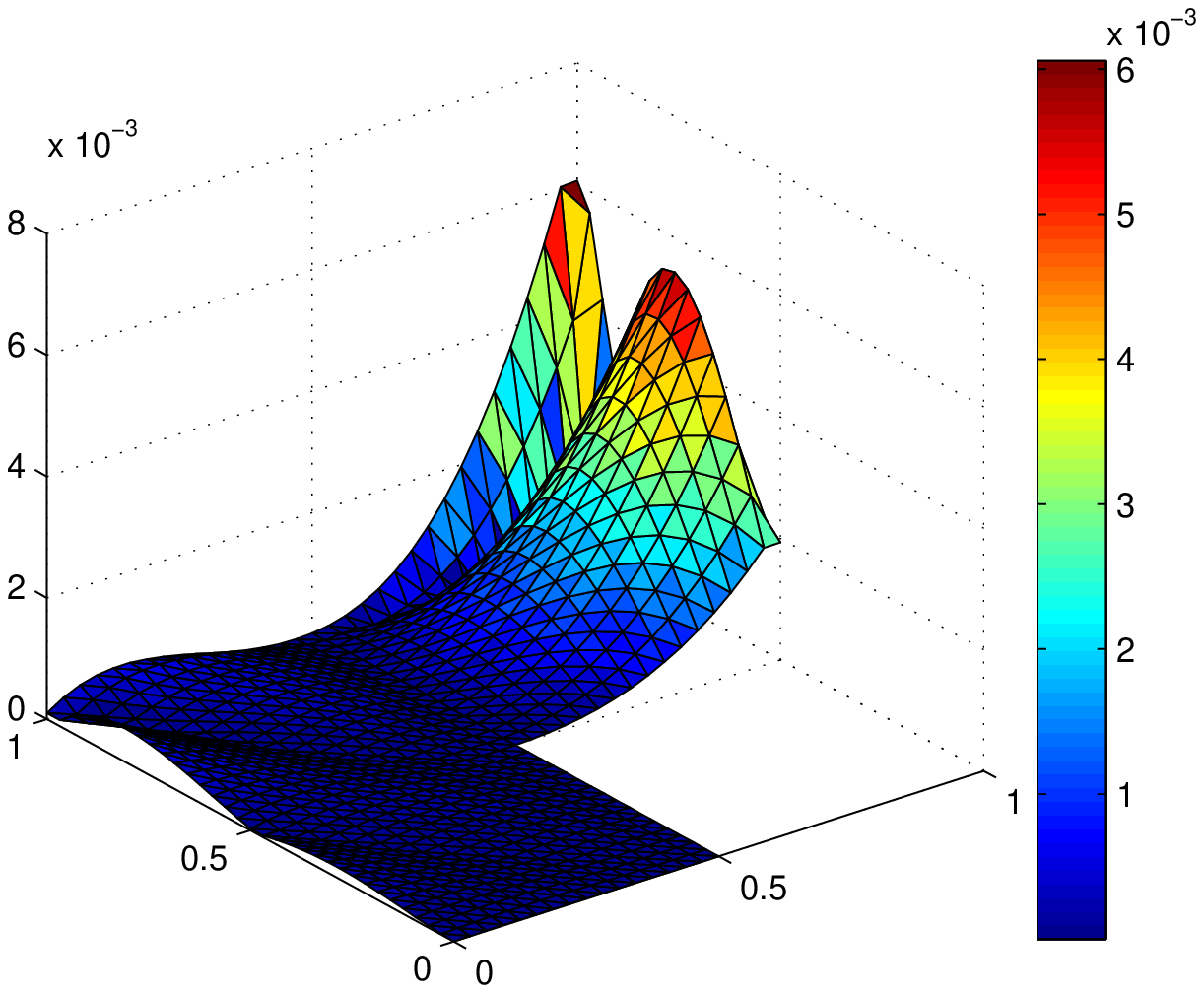}
\includegraphics [width=1.0in,height=1.0in]{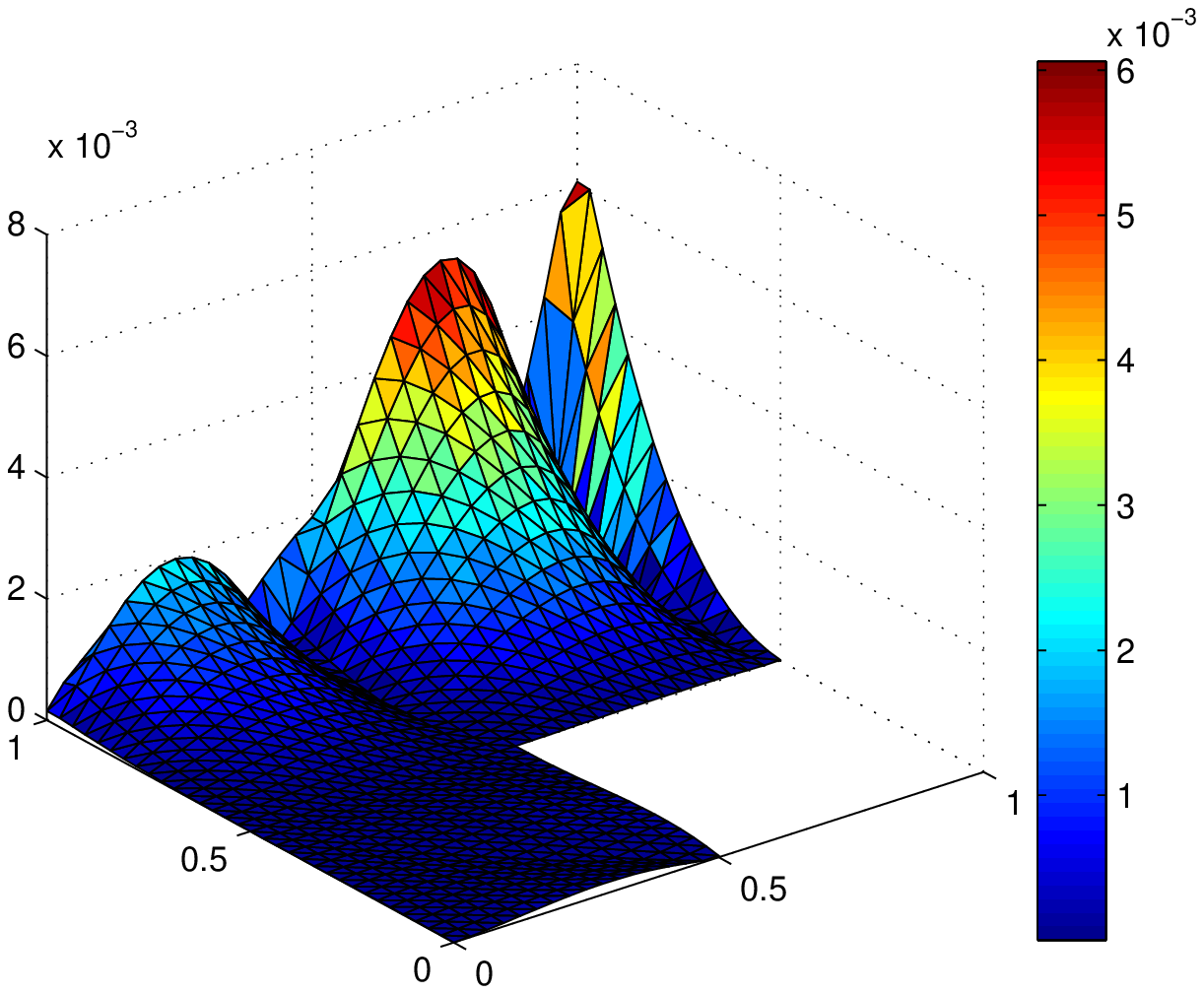}}
\label{fig:2}\caption{Error values for $E_{1h}$ (the first left) , $E_{2h}$ (the second left ) and $P_{1h}$ (the first right) $P_{2h}$ (the second right) on the $L-$type domain.}
\end{figure}

\begin{figure}
\centering
        \hbox{
\includegraphics [width=1.0in,height=1.0in]{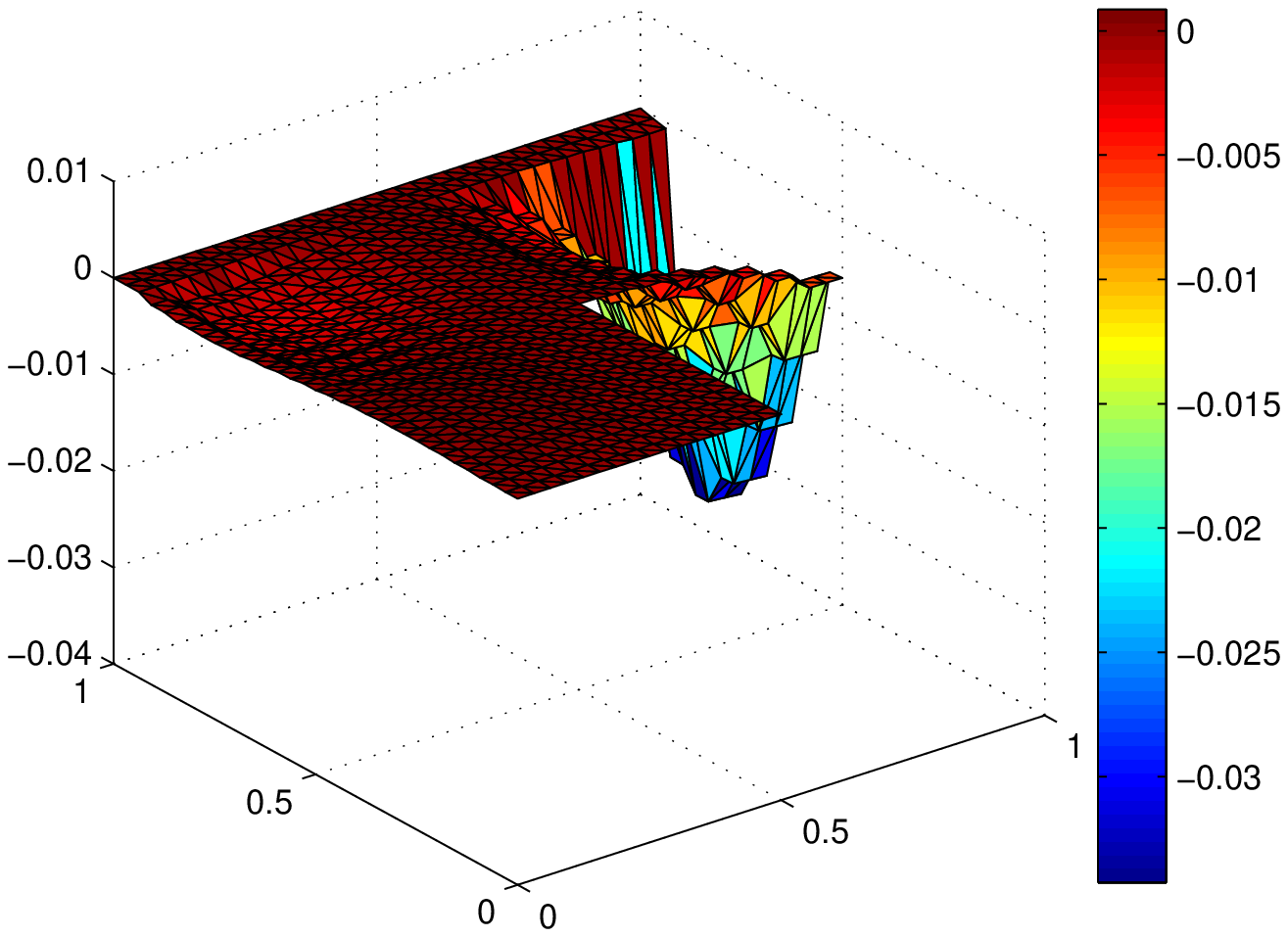}
\includegraphics [width=1.0in,height=1.0in]{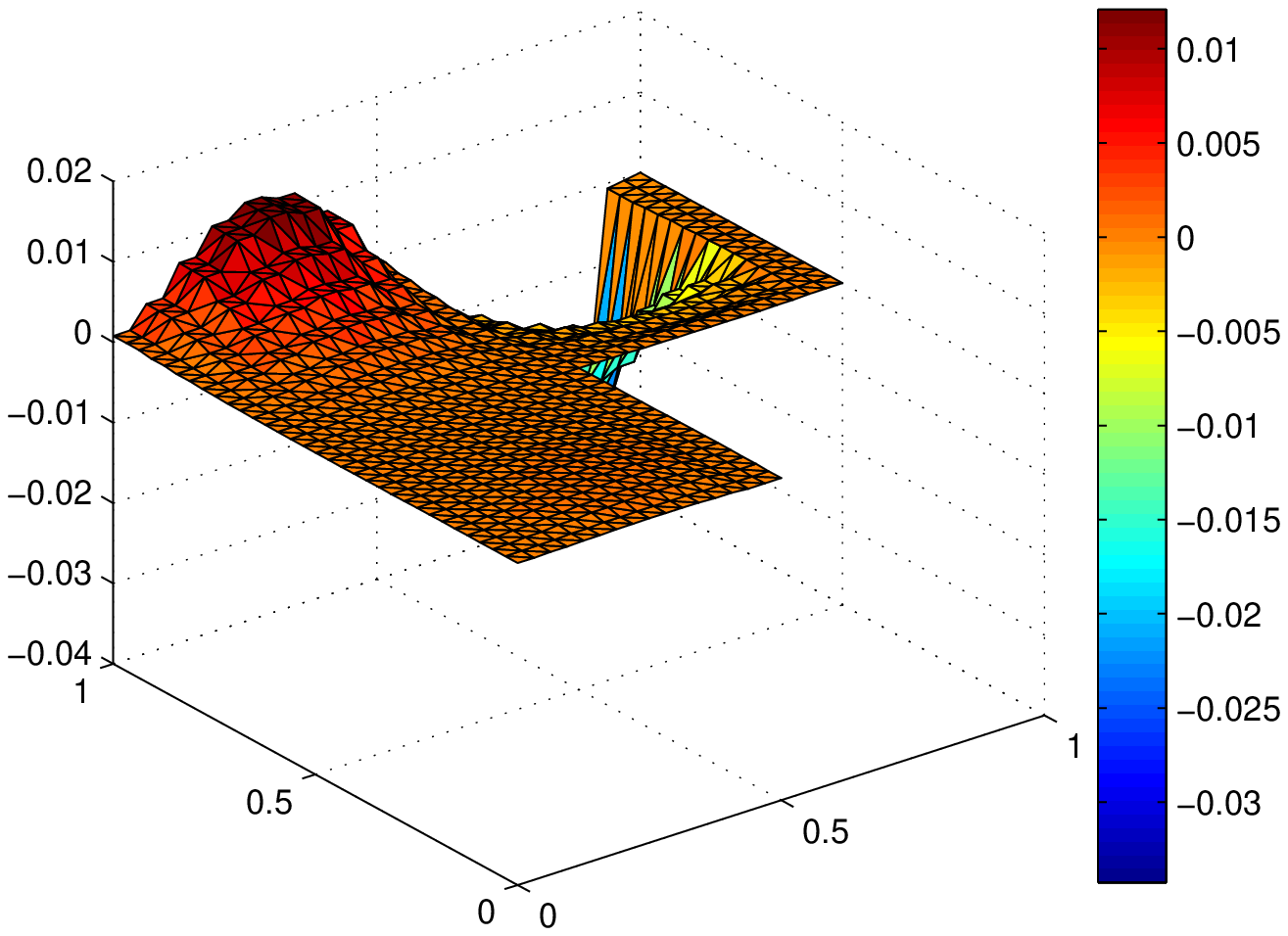}
\includegraphics [width=1.0in,height=1.0in]{L_Ph_32_tau1e-5_100_al21_P1.eps}
\includegraphics [width=1.0in,height=1.0in]{L_Ph_32_tau1e-5_100_al21_P2.eps}
 }
\label{fig:3}\caption{ Numerical solution for $E_{1h}$ (the first left) , $E_{2h}$ (the second left ) and $P_{1h}$ (the first right) $P_{2h}$ (the second right) by super-convergence technique on the $L-$type domain.}
\end{figure}

\begin{figure}
\centering
        \hbox{
\includegraphics [width=1.0in,height=1.0in]{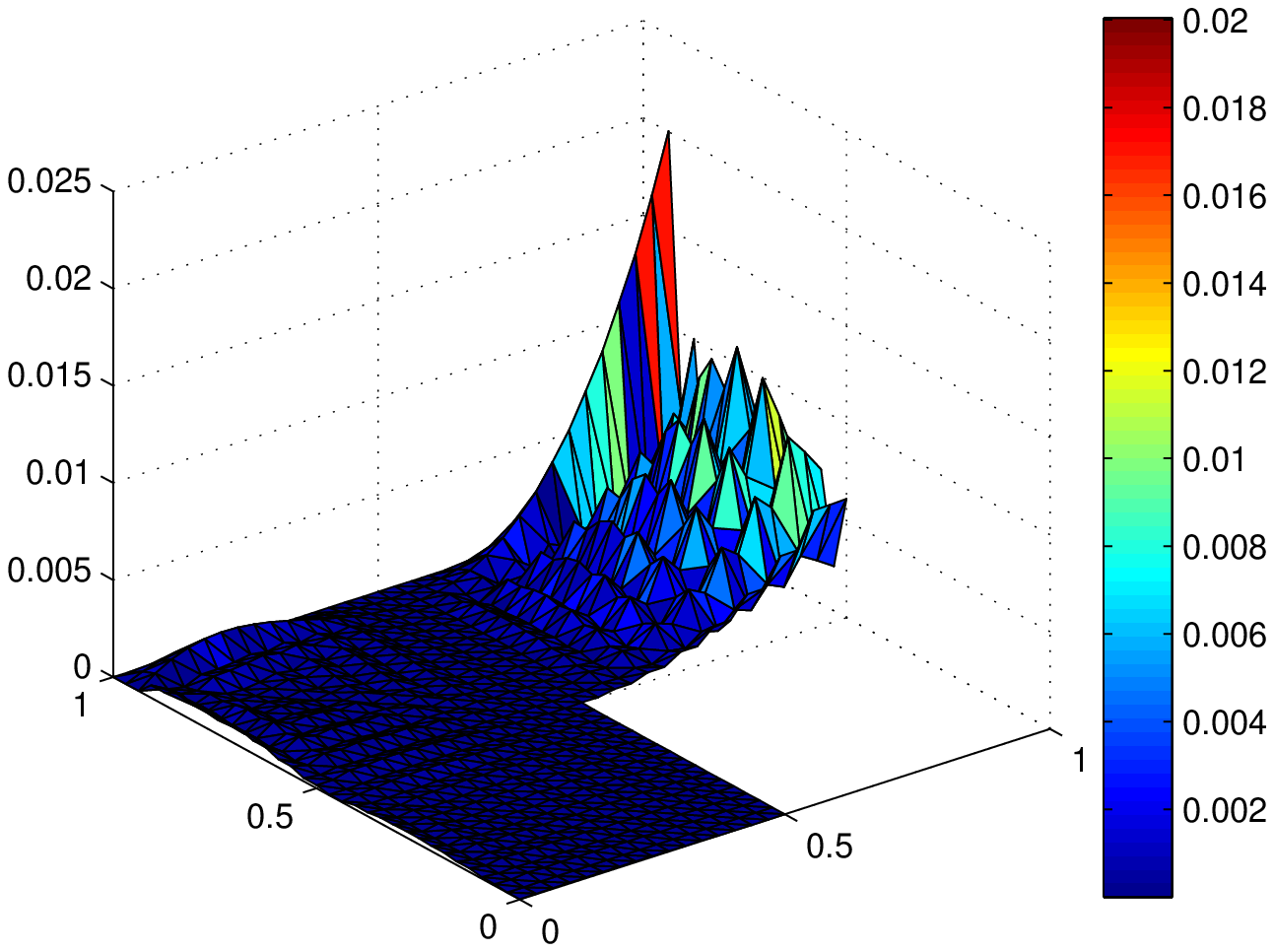}
\includegraphics [width=1.0in,height=1.0in]{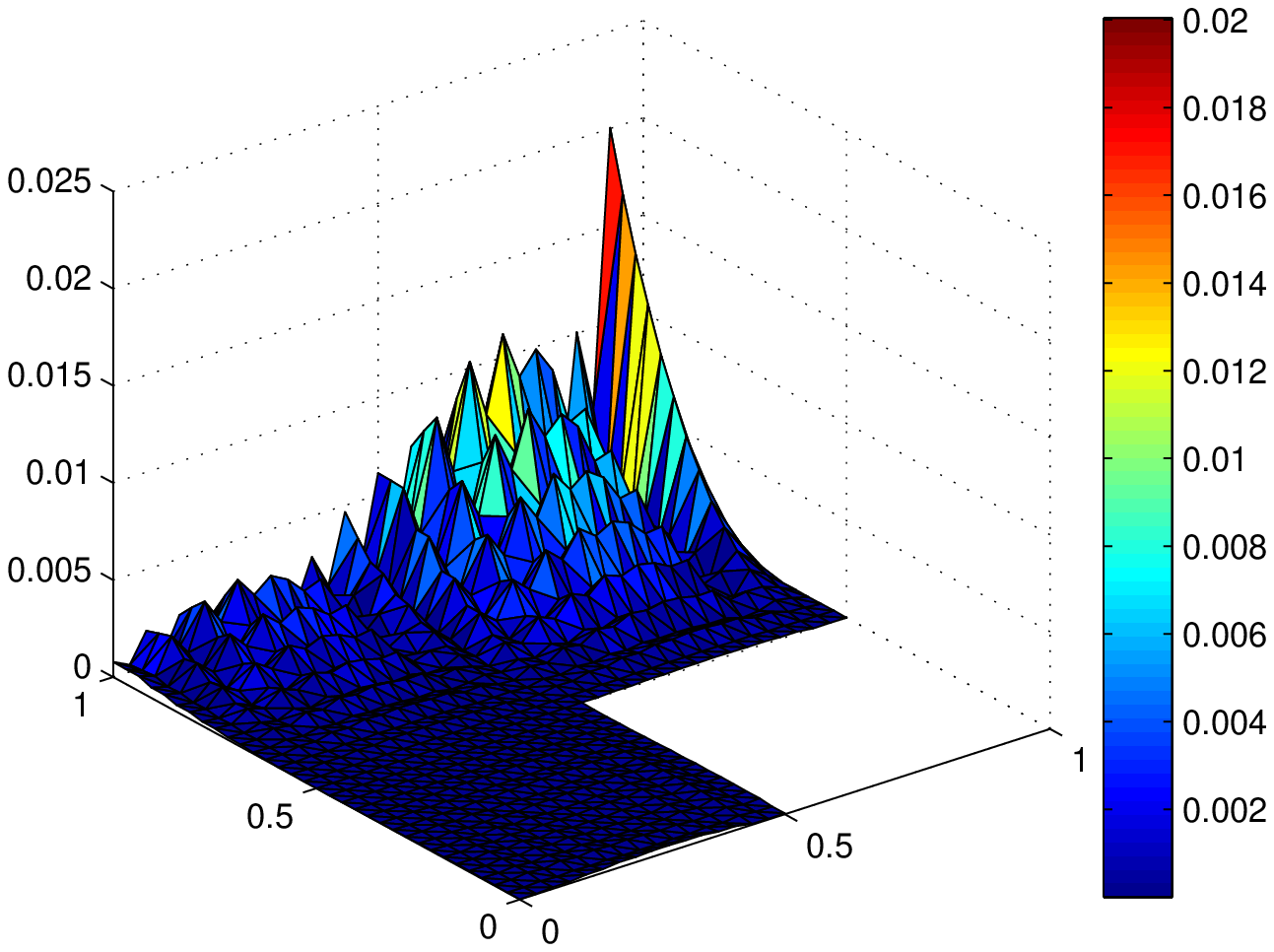}
 \includegraphics [width=1.0in,height=1.0in]{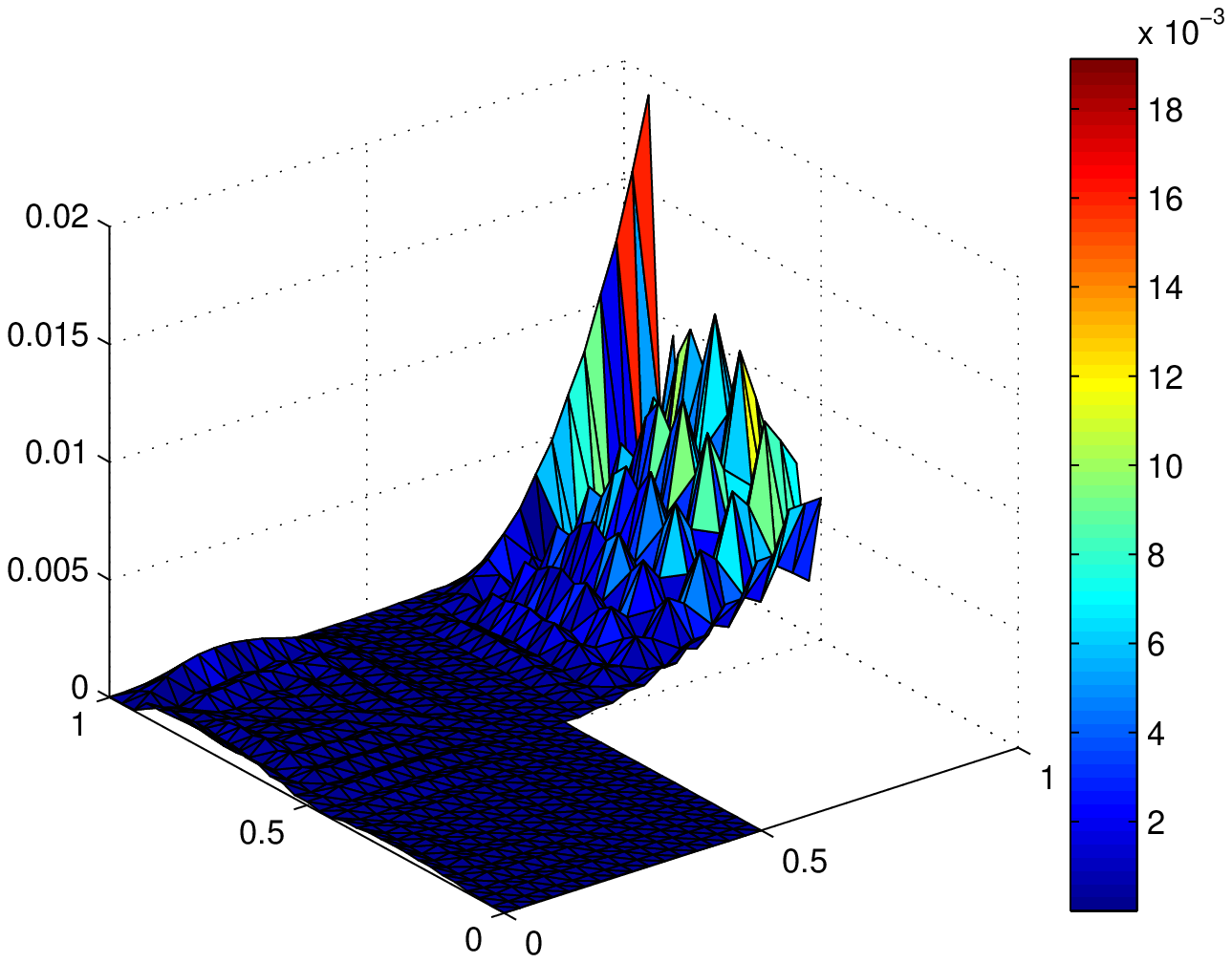}
\includegraphics [width=1.0in,height=1.0in]{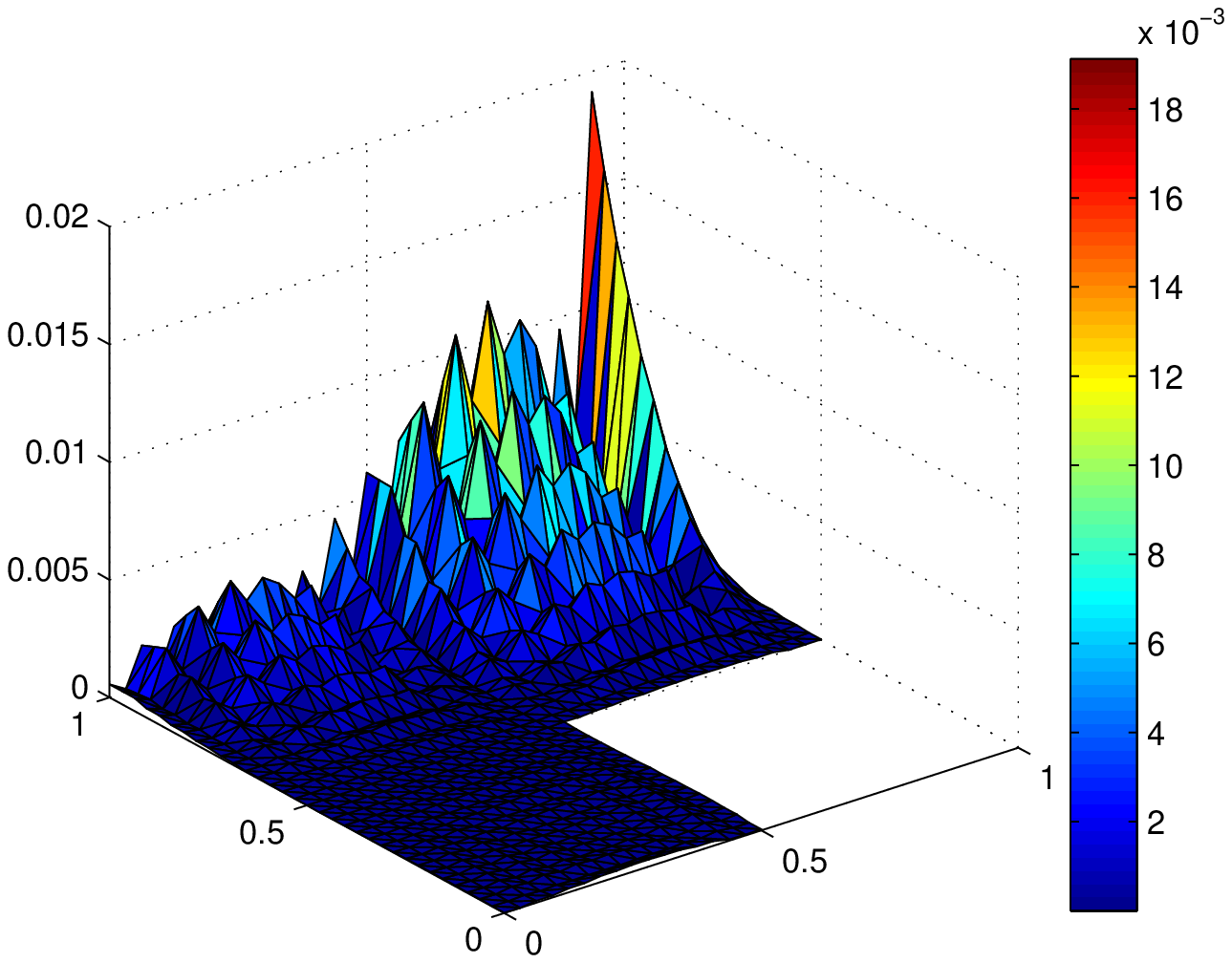}}
\label{fig:4}\caption{Error values for $E_{1h}$ (the first left) , $E_{2h}$ (the second left ) and $P_{1h}$ (the first right) $P_{2h}$ (the second right) by the  super-convergence technique on the $L-$type domain.}
\end{figure}


\begin{figure}
\centering
        \hbox{
\includegraphics [width=1.0in,height=1.0in]{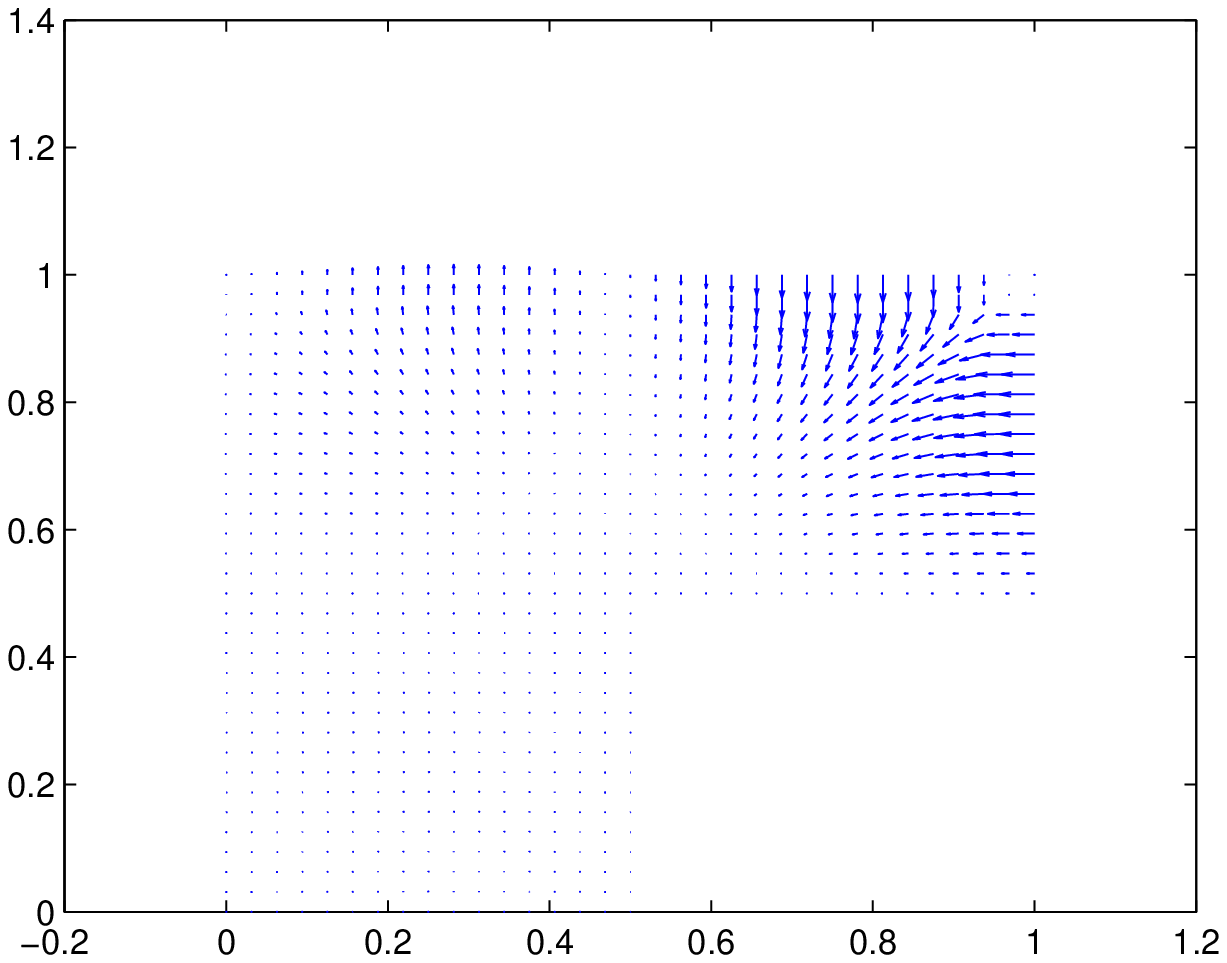}
\includegraphics [width=1.0in,height=1.0in]{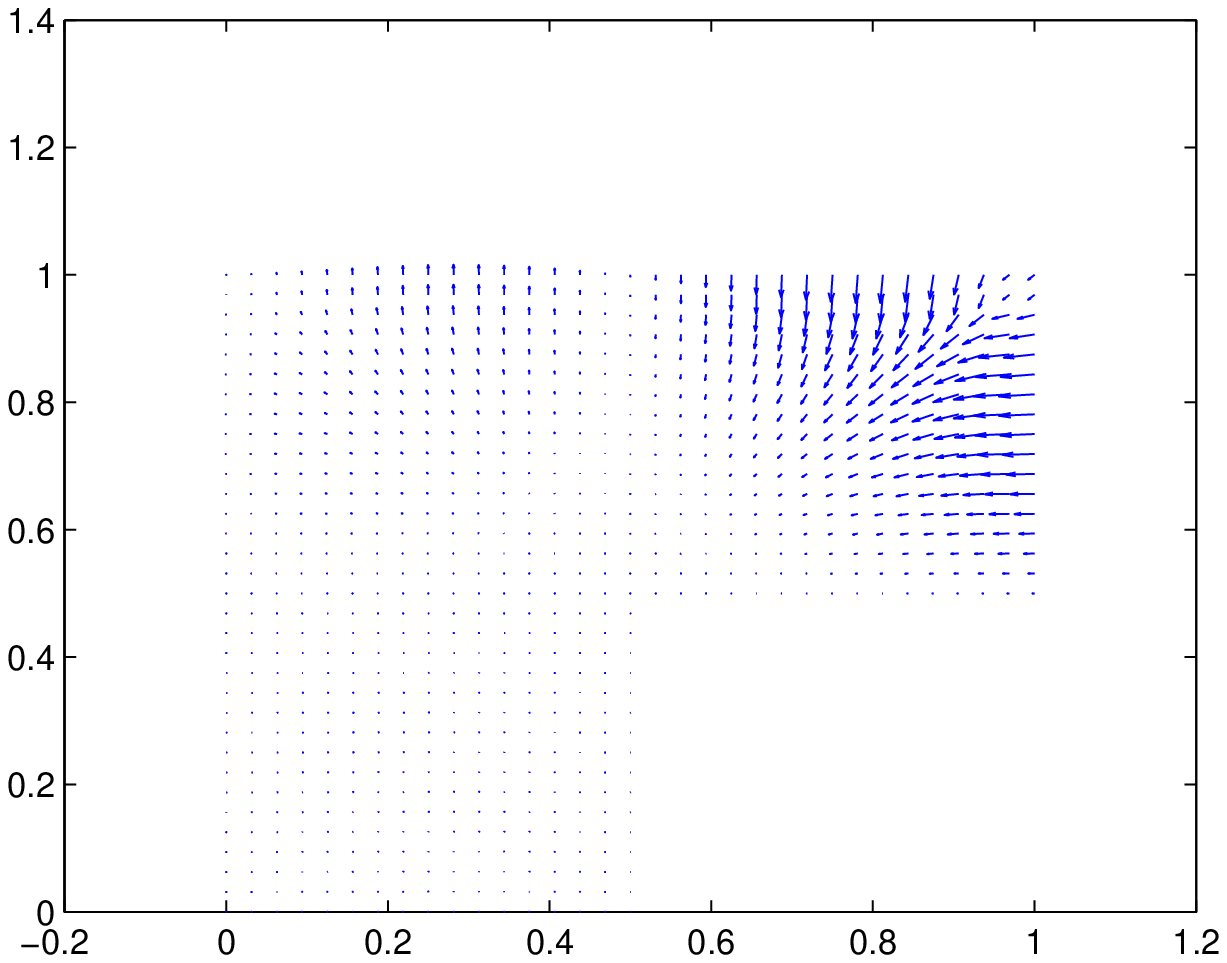}
\includegraphics [width=1.0in,height=1.0in]{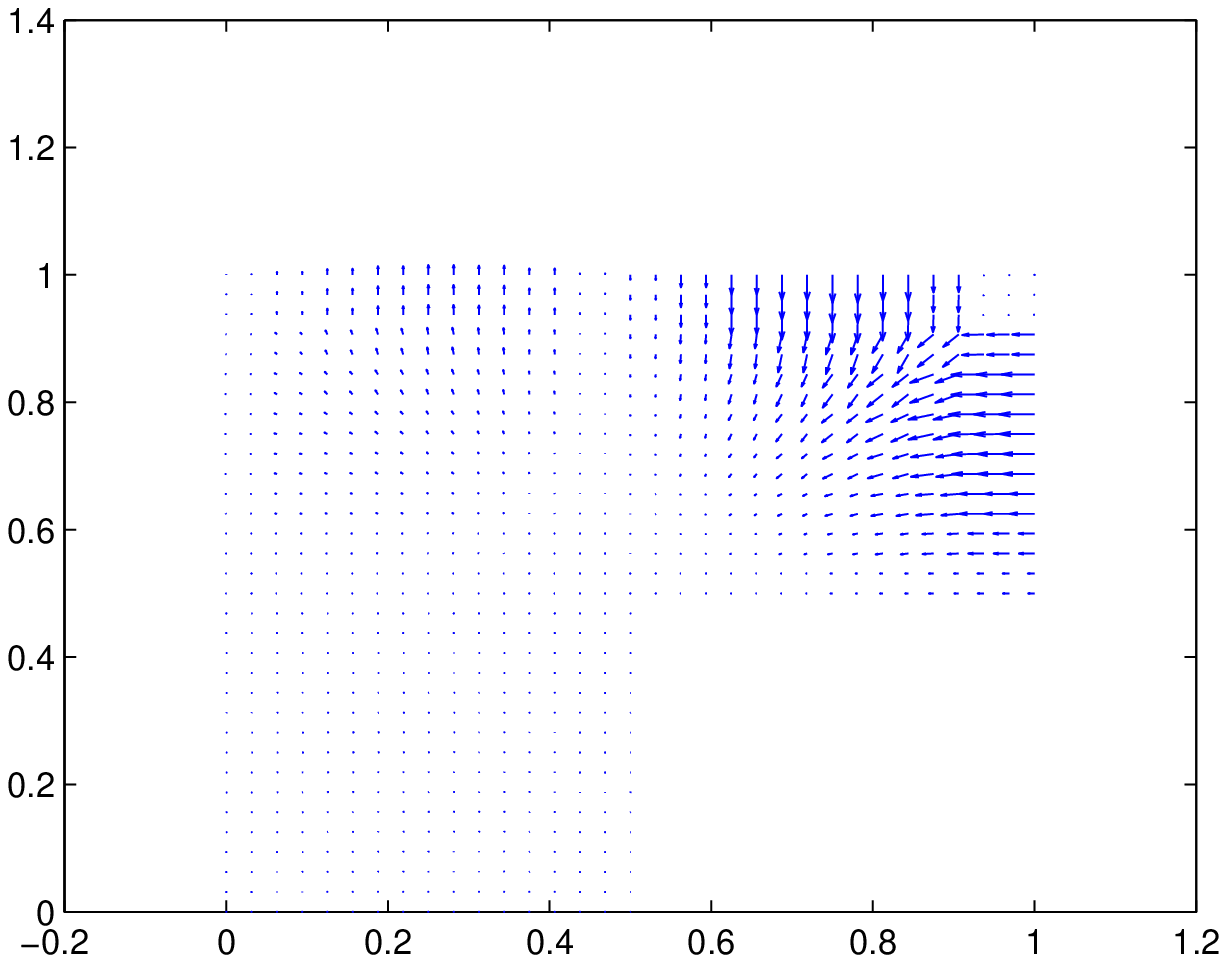}
\includegraphics [width=1.0in,height=1.0in]{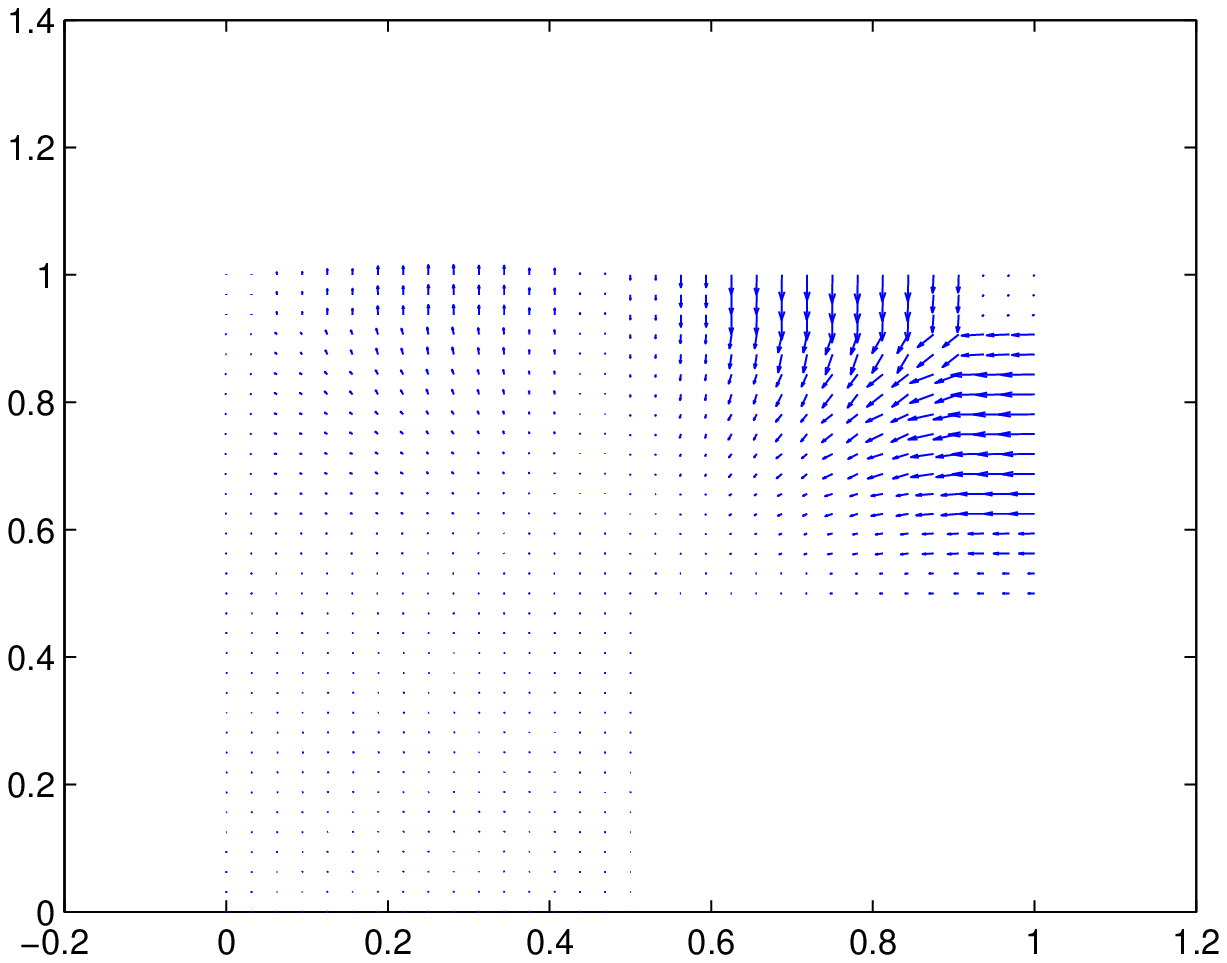}

 }
\label{fig:5}\caption{Vector values at the grids for the numerical solutions of ${\bf E}_h$, ${\bf P}_h$ (two left) and super-convergence vector values
${\bf E}_h$, ${\bf P}_h$ (two right)
 on the $L-$type domain.}
\end{figure}

\section{Conclusion}
In this paper, we first give the variational form in $ H(curl,\Omega)$ space, for the electric fields, different from $H^1(\Omega)$ in the previous work. In order to prove the existence and uniqueness of the variational form, we consider that the nonlinear function $f(x)$ is strongly monotone, Lipschitz continuous and bounded. By employing the monotone theory, we present the existence and uniqueness of the semi-discretization scheme. With the help of reflexive, weak convergence and Arzela-Ascoli theorem, we derive that the solutions of semi-discretization scheme in time converges strongly to the solutions of variational form.

Numerically, we employ the Raviart-Thomas-N$\acute{e}$d$\acute{e}$lec element to approximate the space and a decoupled scheme to discrete the time. To guarantee the $L^\infty$ boundedness of the numerical solutions, we utilize a-priori assumption, which leads to the condition of mesh partition. The optimal error estimates can be obtained under such an assumption for higher finite element space. For the lowest finite element space, we have to use the super-convergent technique. At last we give the numerical examples to demonstrate our methods, in convex domain and $L$-type domain.



\section{Acknowledgements}The first author is supported by NSF (No. 11571027) and the Beijing Nova Program (No. Z1511000003150140). The second author is supported  by NSFC. China (NO.11201501, 11571389,11671165) . The last author is supported by NSFC. China (NO.11471296, 11101384).
At the same time, the authors  gratefully acknowledge the referees for their
great efforts and valuable suggestions or questions on our
manuscript.

\end{document}